\documentclass[10pt,a4paper, reqno]{amsart}
\usepackage{graphicx}
\usepackage{amsmath, color, verbatim}
\usepackage[numbers]{natbib}
\usepackage{amsfonts, mathtools} 
\usepackage{hyperref}
\usepackage[hyperref, dvipsnames]{xcolor}
\hypersetup{
  colorlinks=true,
}
\usepackage{amssymb}
\usepackage{amsthm}
\usepackage[left=2cm,right=2cm,top=2cm,bottom=2cm]{geometry}
\DeclareMathOperator*{\esssup}{ess\,sup}

\newcommand{\norm}[2]{\left\lVert #1\right\rVert_{#2}}

\newcommand{\md}{\partial^\bullet}
\newtheorem{theorem}{Theorem}[section]
\newtheorem{defn}[theorem]{Definition}
\newtheorem{lem}[theorem]{Lemma}

\newtheorem{cor}[theorem]{Corollary}
\newtheorem{remark}[theorem]{Remark}

\newcommand{\R}{\mathbb{R}}

\renewcommand{\bar}{\overline}

\newcommand{\labelForVar}{a}

\newcommand{\grad}{\nabla}
\newcommand{\sgrad}{\nabla_\Gamma}

\newcommand{\lap}{\Delta}
\newcommand{\slap}{\Delta_\Gamma}

\newcommand{\weaklyto}{\rightharpoonup}

\newcommand{\cts}{\hookrightarrow}

\begin{document}
\hypersetup{
  urlcolor     = blue, 
  linkcolor    = Bittersweet, 
  citecolor   = Cerulean
}

\title[A coupled ligand-receptor bulk-surface system on a moving domain]{A coupled ligand-receptor bulk-surface system on a moving domain: well posedness, regularity and convergence to equilibrium}
\author{Amal Alphonse}
\address[Amal Alphonse]{Weierstrass Institute, Mohrenstrasse 39, 10117 Berlin, Germany} 
\email{amal.alphonse@wias-berlin.de}
\author{Charles M. Elliott}
\address[Charles M. Elliott]{Mathematics Institute, University of Warwick, Coventry CV4 7AL, United Kingdom} 
\email{c.m.elliott@warwick.ac.uk}
\author{Joana Terra}
\address[Joana Terra]{FaMAF--CIEM,  Universidad Nacional de C\'ordoba, CP. 5000 C\'ordoba, Argentina}
\email{joanamterra@gmail.com}

\begin{abstract}
We prove existence, uniqueness, and regularity for a reaction-diffusion system of coupled bulk-surface equations on a moving domain modelling receptor-ligand dynamics in cells. The nonlinear coupling between the three unknowns is through the Robin boundary condition for the bulk quantity and the right hand sides of the two surface equations. Our results are new even in the non-moving setting, and in this case we also show exponential convergence to a steady state. The primary complications in the analysis are indeed the nonlinear coupling and the Robin boundary condition. For the well posedness and essential boundedness of solutions we use several De Giorgi-type arguments, and we also develop some useful estimates to allow us to apply a Steklov averaging technique for time-dependent operators to prove that solutions are strong. Some of these auxiliary results presented in this paper are of independent interest by themselves.
\medskip


\end{abstract}
\maketitle

\section{Introduction}\label{sec:intro}
In this paper, we are interested in a reaction-diffusion system (motivated by biology) involving an equation on a moving bulk domain which has a nonlinear coupling to two surface equations on the boundary of the moving domain. We address the issues of well posedness and regularity as well as convergence of the solution to a steady state. The precise geometric setting is as follows. For each $t \in [0,T]$, let $D(t) \subset \mathbb{R}^{d+1}$ be a Lipschitz domain containing a $C^2$-hypersurface $\Gamma(t)$ which separates $D(t) = I(t) \cup \Omega(t)$ into an interior region $I(t)$ and an exterior (Lipschitz) domain $\Omega(t)$ (see Figure \ref{fig:1}). 
\begin{figure}[h]\label{fig:1}
\centering
\includegraphics[scale=0.7]{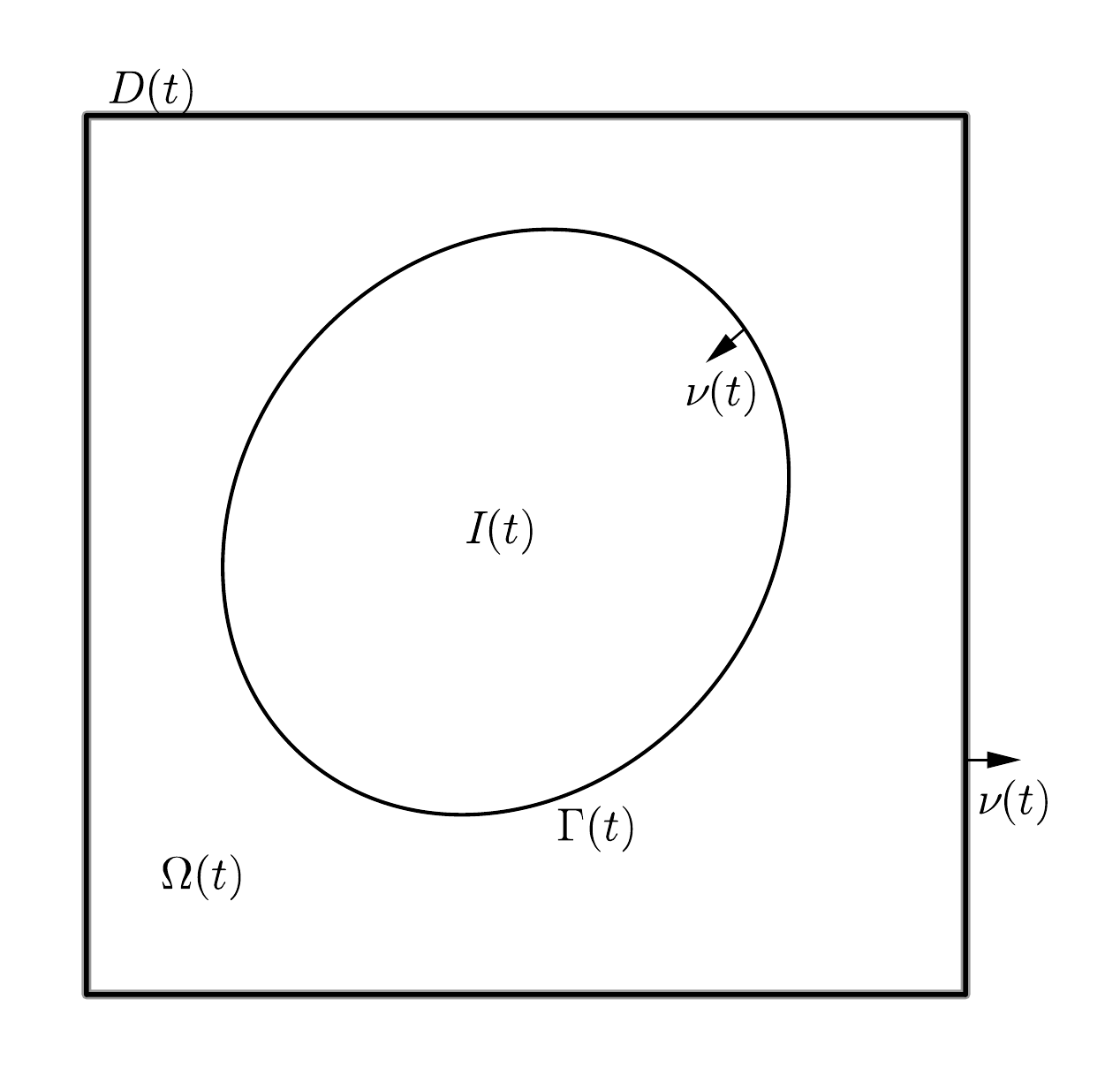}
\caption{A sketch of the geometric set-up.}
\end{figure}
We suppose that the surface $\Gamma(t)$ and the outer boundary $\partial D(t)$ both evolve in time with normal kinematic velocity fields $\mathbf V$ and $\mathbf V_o$ respectively. In addition, the points in $\Omega(t)$ and $\Gamma(t)$ are subject to material velocity fields $\mathbf V_\Omega$ and $\mathbf V_\Gamma$ respectively. These material velocity fields may arise from physical processes in the regions such as fluid flow. In order for the material velocity $\mathbf V_\Gamma$ to be compatible with the movement of the surface, its normal component must agree with the evolution: $(\mathbf V_\Gamma \cdot \nu)\nu = \mathbf V.$ We propose to study the following reaction-diffusion system of equations
\begin{equation}\label{eq:model}
\begin{aligned}
 u_t +\nabla\cdot\left(u\mathbf V_\Omega\right) - \delta_\Omega\Delta u &=0 &&\text{in }\Omega(t)\\
\delta_\Omega\nabla u \cdot \nu + u\left(\mathbf V_{\Gamma}-\mathbf V_\Omega\right) \cdot \nu &=  r(u,w,z) &&\text{on }\Gamma(t)\\
\grad u \cdot \nu &=0 &&\text{on $\partial D(t)$}\\
\partial^\circ w + w\sgrad \cdot \mathbf V  + \nabla_\Gamma\cdot\left(w\mathbf V_\Gamma^\tau\right)-\delta_\Gamma \Delta_\Gamma w&=r(u,w,z)&&\text{on }\Gamma(t)\\
\partial^\circ z + z\sgrad \cdot \mathbf V  + \nabla_\Gamma\cdot\left(z\mathbf V_\Gamma^\tau\right)-\delta_{\Gamma'}\Delta_\Gamma z&=-r(u,w,z)&&\text{on }\Gamma(t)
\end{aligned}
\end{equation}
where $\partial^\circ w=w_t+\nabla w \cdot \mathbf{V}$ is the normal time derivative (see \cite{CFG, Dziuk2013} and also Appendix \ref{sec:derivation}), $\sgrad$ stands for the tangential gradient on $\Gamma(t)$ and $\Delta_\Gamma$ denotes the Laplace--Beltrami operator on $\Gamma(t)$, understood as the tangential divergence of the tangential gradient. In \eqref{eq:model}, $\nu$ means the unit normal on $\Gamma(t)$ pointing out of $\Omega(t)$ and $\mathbf{V}_\Gamma^\tau$ means the tangential component of $\mathbf{V}_\Gamma$. The particular reaction term $r$ we take to be
\begin{equation}\label{reaction}
r(u,w,z)=\frac{1}{\delta_{k'}}z-\frac{1}{\delta_k}uw.
\end{equation}
The constants $\delta_\Omega, \delta_\Gamma, \delta_{\Gamma'},\delta_k$ and $\delta_{k'}$ are positive (physically based) parameters 
and we supplement the system above with non-negative and bounded initial data:
\begin{equation*}\label{IC}
\begin{aligned}
(u(0), w(0), z(0)) &=(u_0, w_0, z_0) \in L^{\infty}(\Omega_0)\times L^\infty(\Gamma_0)^2  &&\text{and }\;\; u_0, w_0, z_0\geq0,
\end{aligned}
\end{equation*}
where $\Omega_0:=\Omega(0)$ and $\Gamma_0:=\Gamma(t)$. The system \eqref{eq:model} is a reaction-diffusion system on an evolving space and it can be derived using conservation and mass balance laws involving fluxes that reflect the presence of the different velocity fields in the model we have in mind; details of this derivation can be found in Appendix \ref{sec:derivation}. Although this paper is mainly focused on the mathematical analysis of \eqref{eq:model}, our motivation for studying the model (with the particular reaction term, the geometry, range of parameters and initial data) stems from a biological application to receptor-ligand dynamics that we shall describe in \S \ref{sec:bio}.

As we already wrote, we are interested in questions of existence and uniqueness of weak solutions and their regularity. We will prove that solutions are in fact strong solutions, meaning that the equations hold pointwise almost everywhere. To achieve this for $u$  (which has a challenging Robin boundary condition) we apply a Steklov averaging technique for which we develop some tools since the elliptic operators are time-dependent due to the domain movement. These results have wider applicability: they can be used for showing regularity to other parabolic equations on moving domains or with time-dependent coefficients. It is also worth highlighting that our results, which we shall present in \S \ref{sec:mainResults}, are new even in the non-moving case when there is no domain evolution or material flow. In this case, the system \eqref{eq:model} simplifies to
\begin{equation}\label{eq:modelStationary}
\begin{aligned}
 u_t - \delta_\Omega\Delta u &=0 &&\text{in }\Omega\\
\delta_\Omega \nabla u \cdot \nu &=  r(u,w,z) &&\text{on }\Gamma\\
\grad u \cdot \nu &=0 &&\text{on $\partial D$}\\
w_t -\delta_\Gamma \Delta_\Gamma w&=r(u,w,z)&&\text{on }\Gamma\\
z_t-\delta_{\Gamma'}\Delta_\Gamma z&=-r(u,w,z)&&\text{on }\Gamma.
\end{aligned}
\end{equation}
Furthermore, we will specialise to this special case later on when we show exponential convergence to an equilibrium as $t \to \infty$.

Apart from the domain evolution and the various velocity fields in the model, another interesting feature of \eqref{eq:model} is that it is a system with \emph{cross-diffusion} \cite{B813825G, Ni, SHIGESADA197983}. Indeed, setting $v := w + z$ and eliminating for example $z$, we see that $v$ solves
\begin{equation}\label{eq:crossDiff}
\partial^\circ v + v\sgrad \cdot \mathbf V   + \sgrad \cdot (v\mathbf V_\Gamma^\tau ) -\delta_{\Gamma'}\slap v  = (\delta_\Gamma - \delta_{\Gamma'})\slap w
\end{equation}
and the presence of the extraneous Laplacian of $w$ justifies this classification. This cross term causes considerable problems in the existence proof as we shall we see later. Some papers emphasising the nature, difficulties and peculiarities of systems with cross-diffusion that are somewhat related to ours include \cite{Bothe2012, Madzvamuse2014, Madzvamuse2016}.

This paper extends the work by Elliott, Ranner and Venkataraman in \cite{ERV} where the authors considered the system \eqref{eq:model} on a fixed domain with the $z$ variable absent, i.e., they studied the $2\times 2$ system involving only $u$ and $w$. As we will see later, the inclusion of the $z$ species considerably complicates the analysis. In \cite{MacKenzie}, the authors consider a two-component coupled bulk-surface system on a moving domain in a similar type of domain as ours. Some other papers featuring bulk-surface interaction on moving domains are \cite{Marth2014, MR3423226, AESApplications, Madzvamuse2016}. The nearest models to the $3\times 3$ system \eqref{eq:model} with the nonlinear coupling \eqref{reaction} studied in the literature are all posed on the same domain with nonlinearities present as right hand side source terms (as opposed to a combination of source terms and the Robin boundary condition as we have), i.e., systems of the form
\begin{equation}\label{eq:otherModel}
\begin{aligned}
 u_t - \delta_1\Delta u &= r(u,w,z) &\text{ in }\Omega\\
 w_t - \delta_2 \Delta w &=  r(u,w,z)&\text{ in }\Omega\\
z_t - \delta_3\Delta z &= -r(u,w,z)&\text{ in }\Omega
\end{aligned}
\end{equation}
with the reaction term \eqref{reaction} and appropriate boundary conditions and initial data on a stationary bounded domain $\Omega$. These are used to model chemical reactions \cite{Rothe}. The properties of such systems are well understood and have a wide literature, eg. see \cite{Rothe, Morgan, Feng} for existence results and \cite{DesFellner} for asymptotic behaviour. The main difficulty of the problem in consideration in the current paper is indeed the presence of the nonlinear coupling through the boundary condition, which in some aspects requires a rather more delicate analysis. 

\subsection{Weak formulation of the problem}\label{sec:weakFormulation}
The system of interest is set in a domain which evolves with time and scalar fields are subject to a material velocity field. Therefore, we have to make clear the precise functional setting we will work on and what we mean by a solution to \eqref{eq:model}. One way to consider this problem is to define the evolution as a family of diffeomorphisms that, for each time, pull back the domain to a fixed reference domain. In other words, the evolution of the domain $\Omega(t)$ with boundary $\Gamma(t)$ can be thought as the transport of a fixed reference domain $\Omega_0$ with boundary $\Gamma_0$. More precisely, we assume that there exists a flow 
$\Phi\colon [0,T]\times\R^{d+1}\rightarrow\R^{d+1}$
such that 
\begin{enumerate}
\item $\Phi_t:=\Phi(t,\cdot)\colon \overline{\Omega_0}\to \overline{\Omega(t)}$ is a $C^2$-diffeomorphism with $\Phi_t(\Gamma_0) = \Gamma(t)$
\item $\Phi_t$ solves the ODE 
\begin{align*}
\frac{d}{dt}\Phi_t(\cdot)&=\mathbf V_p(t,\Phi_t(\cdot))\\
\Phi_0&=\rm{Id}.
\end{align*}
\end{enumerate}
Here, the map $\mathbf V_p\colon [0,T] \times \mathbb{R}^{d+1} \to \mathbb{R}^{d+1}$ is a continuously differentiable velocity field representing a particular parametrisation of the domain, and it must satisfy
\[(\mathbf V_p|_{\Gamma}\cdot \nu)\nu = \mathbf V\quad\text{and}\quad (\mathbf V_p|_{\partial D}\cdot \nu)\nu = \mathbf V_o\]
in order to be compatible with the prescribed evolution of $\Gamma(t)$ and $\partial D(t)$ given by the velocity fields $\mathbf{V}$ and $\mathbf{V}_o$. One may always choose $\mathbf V_p$ to coincide with $\mathbf V$ and $\mathbf V_o$ but for numerical realisations it can be beneficial to choose $\mathbf V_p$ differently to avoid mesh degeneration, see \cite{Dziuk2013, Elliott2012}. 
Moreover, this vector field $\mathbf V_p$ is essential in deriving a useful notion of time derivative in a setting of moving domains. It is well known that for a sufficiently smooth quantity $u$ defined in $\Omega(t)$ its (classical or strong) material derivative is given by
\begin{equation}\label{eq:smd}
\md_\Omega u (t)=u_t(t)+\nabla u(t)\cdot\mathbf V_p(t)
\end{equation}
(see \cite{AESAbstract, AESApplications} and references therein). This is equivalent to a total derivative (taking into account that space points $x$ also depend on time and their trajectory has been parametrized with a velocity given by $\mathbf V_p$). 

Under the circumstances given above, the mapping $\phi_{-t}$ defined by $\phi_{-t}u := u \circ \Phi_t$ defines a linear homeomorphism between the spaces $L^p(\Omega(t))$, $H^1(\Omega(t))$ and the reference spaces $L^p(\Omega_0)$, $H^1(\Omega_0)$ \cite{AESApplications, AEStefan}. The same is true for the corresponding Sobolev spaces on $\Gamma$ \cite{AEStefan}. If we also assume
\begin{enumerate}
\item $\Phi_t:=\Phi(t,\cdot)\colon \overline{\Gamma_0}\to \overline{\Gamma(t)}$ is a $C^3$-diffeomorphism
\item $\Phi_{(\cdot)} \in C^3([0,T]\times \Gamma_0)$
\end{enumerate}
then the above property also holds for the fractional Sobolev space $H^{1\slash 2}(\Gamma(t))$ \cite[\S 5.4.1]{AESApplications}. 

We may then define the Banach spaces $L^p_{L^q(\Omega)}$, $L^2_{H^1(\Omega)}$, $L^p_{L^q(\Gamma)}$, $L^2_{H^1(\Gamma)}$, $L^2_{H^{1\slash 2}(\Gamma)}$, which are Hilbert spaces for $p=q=2$. Here, the notation $L^p_X$ stands for 
\begin{align*}
L^p_X &:= \left\{u:[0,T] \to \bigcup_{t \in [0,T]}\!\!\!\! X(t) \times \{t\},\quad t \mapsto (\hat u(t), t) \mid \phi_{-(\cdot)} \hat u(\cdot) \in L^p(0,T;X_0 )\right\}
\end{align*}
where for each $t \in [0,T]$, the map $\phi_{-t}\colon X(t) \to  X_0$ is a linear homeomorphism; the corresponding norm is
\[\norm{u}{L^p_{X}} := \left(\int_0^T \norm{u(t)}{X(t)}^p\right)^{\frac 1p}\]
(for $p=\infty$ we make the obvious modification). These are generalisations of Bochner spaces to handle (sufficiently regular) time-evolving Banach spaces $X \equiv \{X(t)\}_{t \in [0,T]}$. These spaces were first defined in the (Hilbertian) Sobolev space setting by Vierling in \cite{Vierling} and the theory was subsequently generalised by the first two of the present authors along with Stinner in \cite{AESAbstract} to an abstract Hilbertian setting and by the first two present authors in \cite{AEStefan} to a more general Banach space setting. In \cite{AESAbstract} it is also defined a notion of a weak time derivative (or weak material derivative), which in the context of this paper is as follows. We say that a function $u \in L^2_{H^1(\Omega)}$ has a weak material derivative $\md_\Omega u \in L^2_{H^{1}(\Omega)^*}$ if 
\begin{equation*}
\int_0^T \langle \md_\Omega u(t), \eta(t) \rangle_{H^{1}(\Omega(t))^*, H^1(\Omega(t))} = - \int_0^T \int_{\Omega(t)}u(t) \md_\Omega \eta(t) - \int_0^T \int_{\Omega(t)}u(t)\eta(t)\grad \cdot \mathbf V_p
\end{equation*}
holds for all smooth and compactly supported (in time) functions $\eta$, where $\md_\Omega \eta(t)$ is the classical material derivative given by the formula \eqref{eq:smd}. A similar formula with the correct modifications defines the weak material derivative $\md_\Gamma u \in L^2_{H^{-1}(\Gamma)}$ for a function $u \in L^2_{H^1(\Gamma)}$ --- for this, all the integrals and duality products involving $\Omega$ are replaced by $\Gamma$ and also the term $\grad \cdot \mathbf V_p$ becomes $\sgrad \cdot \mathbf V_p$. 

Instead of the cumbersome notation $\md_\Omega, \md_\Gamma$, we will just write the weak material derivative as $\dot u$; the reader should bear in mind that this is an abuse of notation since there are two different derivatives on two different domains. 

With these objects at hand, we define the evolving Sobolev--Bochner spaces
\begin{align*}
H^1_{H^{1}(\Omega)^*} &= \{ u \in L^2_{H^{_1}(\Omega)} \mid \dot u \in L^2_{H^{1}(\Omega)^*}\}\qquad\text{and}\qquad
H^1_{L^2(\Omega)} = \{ u \in L^2_{L^2(\Omega)} \mid \dot u \in L^2_{L^2(\Omega)}\}
\end{align*}
and the Sobolev--Bochner spaces on the surfaces in the obvious manner. We do not give the precise technical details and properties of these spaces here but refer to \cite{AESAbstract, AESApplications, AEStefan} for the interested reader.
\begin{remark}
Strictly speaking, it is a misnomer to call $\dot u$ as defined above the weak \emph{material} derivative since the velocity field $\mathbf V_p$ associated to it in its very definition is not (in general) the material velocity field but a parametrised velocity field. We nonetheless persist with this terminology.
\end{remark}
In the model \eqref{eq:model} a material derivative does not feature explicitly, so, in order to be able to formulate the problem in an appropriate functional analysis setting, we add and subtract the term $\nabla u\cdot\mathbf V_p$ to both sides of the first equation in \eqref{eq:model} to obtain (after some manipulation)
\begin{equation*}\label{eqbulkALE}
\dot{u}+u\nabla\cdot\mathbf V_\Omega-\delta\Delta u+\nabla u\cdot\left(\mathbf V_\Omega-\mathbf V_p\right)=0,
\end{equation*}
and a similar equation can be derived for the surface PDEs. For convenience, we will define the jumps
\[\mathbf J_\Omega := \mathbf V_\Omega - \mathbf V_p, \qquad \mathbf{J}_{\Gamma} :=\mathbf V_\Gamma - \mathbf V_p\qquad\text{and}\qquad j:= (\mathbf{V}_\Omega-\mathbf{V}_\Gamma)\cdot \nu = \mathbf J_\Omega|_\Gamma \cdot \nu\] where $j$ is the jump in the normal velocities on $\Gamma$. Observe $\mathbf{J}_{\Gamma} = \mathbf V_\Gamma^\tau - \mathbf V_p^\tau$ has no normal component. 
Finally, the problem \eqref{eq:model} we consider can be rewritten as
\begin{equation}\label{eq:modelALE}
\begin{aligned}
 \dot{u}+u\nabla\cdot\mathbf V_p-\delta_\Omega\Delta u+\nabla \cdot\left(\mathbf{J}_\Omega u\right)&=0&\text{ in }\Omega(t)\\
\delta_\Omega\nabla u \cdot \nu - uj &=  r(u,w,z)&\text{ on }\Gamma(t)\\
\dot{w} + w\nabla_\Gamma\cdot\mathbf V_p-\delta_\Gamma\Delta w+\nabla_\Gamma \cdot\left(\mathbf{J}_\Gamma w\right)&= r(u,w,z)&\text{ on }\Gamma(t)\\
\dot{z} + z\nabla_\Gamma\cdot\mathbf V_p-\delta_{\Gamma'}\Delta z+\nabla_\Gamma \cdot\left(\mathbf{J}_\Gamma z\right)&=-r(u,w,z)&\text{ on }\Gamma(t)
\end{aligned}
\end{equation}
and it is this form of the system that we work with in the rest of the paper. Now we define what we mean by a weak solution to \eqref{eq:modelALE}. 
\begin{defn}\label{defn:weakForm}
A weak solution of \eqref{eq:modelALE} is a triple $(u,w,z) \in H^1_{H^{1}(\Omega)^*}\cap L^2_{H^1(\Omega)} \times (H^1_{L^2(\Gamma)}\cap L^2_{H^1(\Gamma)})^2$ with $w \in L^\infty_{L^\infty(\Gamma)}$ satisfying 
\begin{align*}
&\langle \dot u, \eta \rangle + \int_{\Omega(t)}u\eta \grad \cdot \mathbf V_p + \delta_\Omega \int_{\Omega(t)}\nabla u \cdot\nabla \eta + \int_{\Omega(t)}\grad \cdot (\mathbf{J}_\Omega u)\eta = \int_{\Gamma(t)}r(u,w,z)\eta + \int_{\Gamma(t)}ju\eta&&\text{$\forall \eta \in L^2_{H^1(\Omega)}$}\\
&\langle \dot w, \psi \rangle + \int_{\Gamma(t)}w\psi\sgrad \cdot \mathbf V_p  + \delta_\Gamma \int_{\Gamma(t)}\sgrad w\cdot\sgrad \psi + \int_{\Gamma(t)}\sgrad \cdot (\mathbf{J}_{\Gamma} w)\psi = \int_{\Gamma(t)}r(u,w,z)\psi&&\text{$\forall \psi \in L^2_{H^1(\Gamma)}$}\\
&\langle \dot z, \xi \rangle + \int_{\Gamma(t)}z\xi\sgrad \cdot \mathbf V_p  + \delta_{\Gamma'} \int_{\Gamma(t)}\sgrad z\cdot\sgrad \xi + \int_{\Gamma(t)}\sgrad \cdot (\mathbf{J}_{\Gamma} z)\xi = -\int_{\Gamma(t)}r(u,w,z)\xi&&\text{$\forall \xi \in L^2_{H^1(\Gamma)}$}    
\end{align*}
for almost every $t \in [0,T]$.
\end{defn}
 From now on, for reasons of readability, we will omit the $\cdot$ from the dot products above 
\begin{remark}Our definition of the weak solution is written in terms of the chosen parametrisation $\mathbf V_p$ and the spaces where we look for solutions depend on $\mathbf V_p$ since they depend on $\Phi_{(\cdot)}$. So a natural and reasonable question is whether we get a different solution if we pick a different parametrisation of the velocity field and different $\Phi_{(\cdot)}$. Let us see that this is not the case, at least formally. Working in the regime of strong solutions,  given two solutions $(u_i, w_i, z_i)$ corresponding to parametrisations $\mathbf{V}_{p}^i$ solving \eqref{eq:modelALE}, we know that they also solve the original model \eqref{eq:model}, which has no dependence on $\mathbf{V}_p^i$. Taking then the difference of the equations and testing appropriately, we find then that $u_1\equiv u_2$, $w_1 \equiv w_2$ and $z_1 \equiv z_2$ by the same uniqueness argument as the one we present below in the proof of Theorem \ref{thm:uniqueness}.
\end{remark}
A consequence of one of the formulae in Appendix \ref{sec:prelim} allows us to rewrite the weak formulation for $u$ given in Definition \ref{defn:weakForm} as
\begin{equation*}\label{eq:weakFormForu}
\langle \dot u, \eta \rangle + \int_{\Omega(t)}u\eta \grad \cdot \mathbf V_p + \delta_\Omega \int_{\Omega(t)}\nabla u \nabla \eta -\int_\Omega u(\mathbf{J}_\Omega\cdot \grad \eta) = \int_{\Gamma(t)}r(u,w,z)\eta \quad\text{for all $\eta \in L^2_{H^1(\Omega)}$}.
\end{equation*}
By testing with unity the $u$ and $z$ equations and adding the resulting equalities together, and doing the same for the $w$ and $z$ equations, we find 
\begin{align*}
\frac{d}{dt}\left(\int_{\Omega(t)}u(t) + \int_{\Gamma(t)}z(t)\right) = 0\qquad \text{and} \qquad
\frac{d}{dt}\left(\int_{\Gamma(t)} w(t) + \int_{\Gamma(t)} z(t)\right) = 0
\end{align*}
so that the following quantities are conserved for all $t \in [0,T]$:
\begin{equation}\label{eq:conservationMass}
\begin{aligned}
\int_{\Omega(t)}u(t) + \int_{\Gamma(t)}z(t) &= \norm{u_0}{L^1(\Omega_0)} + \norm{z_0}{L^1(\Gamma_0)} =: M_1\\
\int_{\Gamma(t)} w(t) + \int_{\Gamma(t)}z(t) &=  \norm{w_0}{L^1(\Gamma_0)} + \norm{z_0}{L^1(\Gamma_0)} =: M_2.
\end{aligned}
\end{equation}
Note that these results hold independently of the choice of $r$.
\subsection{Main results}\label{sec:mainResults}
All of the following results are valid for dimensions in the physically relevant situations where $\Omega \subset \mathbb{R}^{d+1}$ for $d+1 \leq 3$; some results hold even for $d+1 \leq 4$. 
These constraints on the dimension are due to the various Sobolev-type functional inequalities that we shall use. We begin with the existence result, which we will prove in \S \ref{sec:existence} (the restriction on the dimension is a technical one, due only to the De Giorgi result of Lemma \ref{lem:deGiorgi} that is needed to prove the theorem). 

 Below, the velocity fields $\mathbf{V}, \mathbf{V}_\Gamma, \mathbf{V}_\Omega$ and $\mathbf{V}_0$ are continuously differentiable and satisfy the assumptions detailed in \S \ref{sec:intro} and \S \ref{sec:weakFormulation}.  

\begin{theorem}[Global existence]\label{thm:existence}
For dimensions $d \leq 3$, for every non-negative initial data $(u_0, w_0, z_0) \in L^\infty(\Omega_0)\times L^\infty(\Gamma_0)^2$, there exists a non-negative weak solution $(u,w,z) \in H^1_{H^{1}(\Omega)^*}\cap L^2_{H^1(\Omega)} \times (H^1_{L^2(\Gamma)}\cap L^2_{H^1(\Gamma)})^2$  to the system \eqref{eq:modelALE}. Furthermore, $w$ satisfies
\[\norm{w}{L^\infty_{L^\infty(\Gamma)}} 
\leq e^{(\norm{\sgrad \cdot \mathbf V_\Gamma}{\infty}+\delta_\Gamma)T}\left(\norm{w_0}{L^\infty(\Gamma_0)} + \frac{C}{\delta_{\Gamma}\delta_{k'}}e^{\left(\frac{A }{2\delta_{k'}} + \frac 12\norm{\sgrad \cdot \mathbf V_\Gamma}{\infty}\right)T}\left(\sqrt{2+A}\norm{w_0}{L^2(\Gamma_0)} + \sqrt 2 \norm{z_0}{L^2(\Gamma_0)}\right)\right)\]
where $A \geq \max\left(1, \frac{C_0}{2\delta_\Gamma} + \frac{(\delta_\Gamma - \delta_{\Gamma'})^2}{2\delta_\Gamma\delta_{\Gamma'}}\right)$ for $C_0 \geq 0$ chosen arbitrarily.
\end{theorem}
One expects not only $w$ to be a bounded function but also $u$ and $z$, and indeed we show that this is the case for dimensions $d \leq 2$. This restriction on $d$ is again technical in nature due to another De Giorgi method (Lemma \ref{lem:deGiorgiSecond}) we employ. Further details of this, and the proof of the following theorem will be given in \S \ref{sec:boundednessOfUandZ}.
\begin{theorem}[Boundedness for $z$ and $u$]\label{thm:LinftyBounds}
For dimensions $d \leq 2$, the solutions $u$ and $z$ of the system \eqref{eq:modelALE} are bounded in space and time. In particular,
\begin{align*}
\norm{z}{L^\infty_{L^\infty(\Gamma)}}&\leq e^{(\norm{\sgrad \cdot \mathbf V_\Gamma}{\infty}+\delta_{\Gamma'})T}\left(\norm{z_0}{L^\infty(\Gamma_0)}+ C\min(1\slash 2, \delta_{\Gamma'})^{-1}\norm{w}{L^\infty_{L^\infty(\Gamma)}}\norm{u}{L^2_{H^{1}(\Omega)}}\right)\\
\norm{u}{L^\infty_{L^\infty(\Omega)}} &\leq e^{(\norm{\nabla \cdot \mathbf V_\Omega}{\infty} + C_1\delta_\Omega^{-1} (\alpha + \frac 52\norm{j}{\infty})^2)T}\left(\norm{u_0}{L^\infty(\Omega_0)} + \max\left(1, \alpha C_1\right)(C_5\min(1, \delta_\Omega)^{-(\frac{1}{2\kappa} + \frac 12)}+ C_4)\right)
\end{align*}
where 
$\alpha = \delta_{k'}^{-1}\norm{z}{L^\infty_{L^\infty(\Gamma)}} + \norm{j}{\infty}\norm{u_0}{L^\infty(\Omega_0)}$ and $\kappa$ is as in \S \ref{sec:uBound}.
\end{theorem}
\begin{remark}It is important to emphasise that the $L^\infty$ bounds do not depend on the chosen parametrisation $\mathbf{V}_p$ of the evolution of the spaces but only on the material velocity fields $\mathbf V_\Omega$ and $\mathbf{V}_\Gamma$. 
\end{remark}
\begin{remark}Even in the non-moving setting the proofs of the above $L^\infty$ bounds do not simplify as the technical difficulties are mainly due to the nonlinear coupling. We will explain more during the course of this paper.
\end{remark}
Using these $L^\infty$ bounds it is then easy to prove continuous dependence and uniqueness.
\begin{theorem}[Continuous dependence]\label{thm:uniqueness} Suppose there are two triples of solutions $(u_i, w_i, z_i)$ corresponding to two triples of initial data $(u_{i0}, w_{i0}, z_{i0})$ for $i=1$, $2$. Then for dimensions $d\leq 2$, the following continuous dependence result holds:
\begin{align*}
&\norm{u_1(t)-u_2(t)}{L^2(\Omega(t))} + \norm{w_1(t)-w_2(t)}{L^2(\Gamma(t))} + \norm{z_1(t)-z_2(t)}{L^2(\Gamma(t))}\\
&\quad\leq C\left(\norm{u_{10}-u_{20}}{L^2(\Omega_0)} + \norm{w_{10}-w_{20}(t)}{L^2(\Gamma_0)} + \norm{z_{10}-z_{20}(t)}{L^2(\Gamma_0)}\right),
\end{align*}
where $C=C(\delta_k,\delta_{k'}, \norm{j}{\infty},\norm{\nabla\cdot\mathbf{V}_\Omega}{\infty},\norm{\nabla_\Gamma\cdot\mathbf{V}_\Gamma}{\infty},\norm{u_2}{L^{\infty}_{L^{\infty}}},\norm{w_1}{L^{\infty}_{L^{\infty}}})$.
\end{theorem}
\begin{proof}
Denote the differences by $d^u = u_1-u_2$ and so on. It is easy to write the weak formulation satisfied by $d^u, d^w$ and $d^z$; then taking $d^u, d^w$ and $d^z$ as test functions respectively, writing
$u_1w_1 - u_2w_2 = (u_1-u_2)w_1 + u_2(w_1-w_2) = w_1d^u + u_2d^w$ and using the $L^{\infty}$ bounds on the right hand side, we obtain 
\begin{align*}
\frac{1}{2}\frac{d}{dt}\int_{\Omega(t)}|d^u|^2  + \frac{1}{2}\int_{\Omega(t)}|d^u|^2\grad \cdot \mathbf V_\Omega + \delta_\Omega \int_{\Omega(t)}|\nabla d^u|^2  
&\leq C\int_{\Gamma(t)}|d^u|^2 + |d^w|^2 + |d^z|^2
\end{align*}
\begin{align*}
\frac{1}{2}\frac{d}{dt}\int_{\Gamma(t)}|d^w|^2 + \frac 12\int_{\Gamma(t)}|d^w|^2\sgrad \cdot \mathbf V_\Gamma  + \delta_\Gamma \int_{\Gamma(t)}|\sgrad d^w|^2 
&\leq C\int_{\Gamma(t)}|d^u|^2 + |d^w|^2 + |d^z|^2
\end{align*}
\begin{align*}
\frac{1}{2}\frac{d}{dt}\int_{\Gamma(t)}|d^z|^2 + \frac 12\int_{\Gamma(t)}|d^z|^2\sgrad \cdot \mathbf V_\Gamma  + \delta_{\Gamma'} \int_{\Gamma(t)}|\sgrad d^z|^2
&\leq C\int_{\Gamma(t)}|d^u|^2 + |d^w|^2 + |d^z|^2.
\end{align*}
Combining these three inequalities and moving the velocity terms onto the right hand side yields 
\begin{align*}
&\frac{d}{dt}\left(\int_{\Omega(t)}|d^u|^2  + \int_{\Gamma(t)}|d^w|^2  + \int_{\Gamma(t)}|d^z|^2\right) + 2\delta_\Omega \int_{\Omega(t)}|\nabla d^u|^2 \\
&\quad\leq C_1\int_{\Gamma(t)} |d^z|^2 + |d^w|^2 + \int_{\Omega(t)}C_2(\epsilon)|d^u|^2 + \epsilon |\nabla d^u|^2
\end{align*}
using the interpolated trace inequality \eqref{eq:interpolatedTrace} from Appendix \ref{sec:prelim}. We conclude by choosing $\epsilon< 2\delta_\Omega$ and applying Gronwall's lemma.
\end{proof}
Note the regularity for $w$ and $z$ given by Theorem \ref{thm:existence}: they are strong solutions with the corresponding equations holding pointwise almost everywhere in time.  Regularity for $u$ is a much more delicate matter. The inhomogeneous Robin boundary condition is the source of the problem; surprisingly maximal regularity for such types of boundary conditions for parabolic problems with time-dependent elliptic operators have only somewhat recently been comprehensively answered, see the work of Denk, Hieber and Pr\"uss \cite{Pruss}. See also \cite{Weidemaier2002, Nittka}. Unfortunately, it is not clear that the coefficients in our boundary condition have the desired smoothness to apply the theorems in these papers. Therefore, in \S \ref{sec:strongSolutionsForu}, we shall argue in a different way, making use of the $L^\infty$ bound on $u$ obtained in Theorem \ref{thm:LinftyBounds}.
\begin{theorem}[Regularity for $u$]\label{thm:uStrongSolution}
For dimensions $d \leq 2,$ 
the time derivative of $u$ satisfies $\dot u \in L^2_{L^2(\Omega)}(\tau, T)$ for every $\tau > 0$. That is, $u$ is a strong solution. In the non-moving case, this means that $u' \in L^2(\tau, T; L^2(\Omega_0))$.
\end{theorem}
We shall also consider exponential convergence of the solution to equilibrium in the stationary setting (see system \eqref{eq:modelStationary}) when the reaction rates are equal and when the diffusivity for the $u$ equation is $\delta_\Omega=1$. 
Associated to \eqref{eq:modelStationary} one can define the natural \emph{entropy functional}
\begin{equation}\label{eq:entropy}
E(u,w,z) := \int_\Omega u(\log u -1) + \int_\Gamma w(\log w - 1) +\int_\Gamma z(\log z-1)
\end{equation}
and its non-negative \emph{dissipation}
\begin{align}
D(u,w,z) &:= -\frac{d}{dt}E(u,w,z) =  \int_\Omega \frac{|\nabla u|^2}{u} + \delta_\Gamma \int_\Gamma \frac{|\sgrad w|^2}{w} + \delta_{\Gamma'}\int_\Gamma \frac{|\sgrad z|^2}{z} + \int_{\Gamma}(uw-z)\log\left(\frac{u w}{z}\right).\label{eq:dissipation}
\end{align}
Let $u_\infty$, $w_\infty$ and $z_\infty$ be the unique non-negative constants determined by the system \eqref{eq:algebraic} on page~\pageref{eq:algebraic}. In \S \ref{sec:equil}, following the methodology of Fellner and Laamri \cite{FellnerLaamri}, we will prove the exponential convergence of solutions of \eqref{eq:modelStationary} to $(u_\infty, w_\infty, z_\infty)$ using entropy methods. The main idea is to relate in a useful way the relative entropy $E(u,w,z)-E(u_\infty, w_\infty, z_\infty)$ to the entropy dissipation and then to (prove and) use a lower bound on the relative entropy in terms of $L^1$ distances to the equilibrium (in the sense of the Csiszar--Kullback--Pinsker inequality). We refer the reader to \cite{jungel} for a detailed introduction to entropy methods in PDEs and also the paper \cite{DesFellner} and references therein for an exposition of entropy methods for reaction-diffusion equations. Equilibrium convergence for systems of the form \eqref{eq:otherModel} where all equations are on the same domain have been studied in the literature eg. \cite{DesFellner, FellnerLaamri}, whilst a two-component system with a heat equation on a domain and a surface heat equation on its boundary coupled through a nonlinear Robin-type boundary condition was analysed in \cite{BaoFellnerLatos}.
\begin{theorem}[Convergence to steady state]\label{thm:equilibrium}
For dimensions $d\leq 2$ and if $\delta_{k}=\delta_{k'}=\delta_\Omega=1$, the solution $(u,w,z)$ of the system \eqref{eq:modelStationary} converges to $(u_\infty, w_\infty, z_\infty)$ as $t \to \infty$ in the following sense:
\begin{align*}
\norm{u(t)- u_\infty}{L^1(\Omega)}^2 + \norm{w(t)-w_\infty}{L^1(\Gamma)}^2 + \norm{z(t)-z_\infty}{L^1(\Gamma)}^2 
&\leq C e^{-Kt}(E(u_0,w_0,z_0) -E(u_\infty, w_\infty, z_\infty))
\end{align*}
where $C$ and $K$ are positive constants.
\end{theorem}
The restriction to dimensions $d\leq 2$ above is not because we use $L^\infty$ bounds for $u$ and $z$ in the course of the proof but because we need the $u$ equation to hold pointwise almost everywhere, which is guaranteed by Theorem \ref{thm:uStrongSolution}.

\subsection{Application to biology}\label{sec:bio}
Systems of equations of the form \eqref{eq:modelStationary} arise in the mathematical modelling of biological and chemical processes in cells (in particular those with interactions between processes on the cell membrane and processes inside or outside the cell) such as cell motility and chemotaxis \cite{Marth2014, MacKenzie} and cell signalling processes \cite{Ratz2012, RatzRogerSymmetry, BioPaper}. The modelling of ligand-receptor dynamics is particularly pertinent for our model: $u$ may represent the concentration of ligands in the extracellular volume $\Omega$, whilst $w$ and $z$ may be respectively the concentrations of surface receptors on $\Gamma$ and of ligand-receptor complexes formed by the binding of $u$ with $w$. This justifies the choice of the reaction term $r$ in \eqref{reaction}: we expect non-negative solutions $u$, $w,$ and $z$ so formally, when $u$ and $w$ combine, the complex $z$ is increased whilst the receptor $w$ and ligand $u$ are decreased. The non-trivial boundary condition in \eqref{eq:model}, a combination of a diffusive flux and an advective flux, models the so-called \emph{windshield effect} well known in mathematical biology. The idea is that the front of the moving cell $\Gamma$ will come into contact with ligands in $\Omega$ and thus there will be a greater binding onto $\Gamma$, whilst at the back of the cell complexes tend to dissociate into receptors and ligands. This effect was also taken into account in \cite{MacKenzie} in a similar geometric setting.

With particular choices of the various velocity fields in the model, we end up with special cases that are worth mentioning.
\begin{itemize}
\item We have already mentioned the case $\mathbf V_\Omega =\mathbf V_\Gamma = \mathbf V = \mathbf V_o \equiv 0$ where there is no evolution and no material processes, which results in the system \eqref{eq:modelStationary}. This is the usual geometric setting in which many related models (referenced in the introduction) are treated.
\item Choosing $\mathbf V_\Omega \equiv 0$ we end up in the physical setting of \cite{MacKenzie} where there is only a material velocity on $\Gamma$.
\item If $\mathbf V_\Gamma \cdot \nu = \mathbf V_\Omega \cdot \nu$ on $\Gamma$, then there is no windshield effect. This make sense because, referring to the explanation of the effect given above, if the material points in $\Omega$ are also moving with the normal velocity of $\Gamma$ near the surface $\Gamma$, one clearly cannot expect advection on or off the surface since the ligands are also moving with the same velocity.
\end{itemize}
Regarding particular applications already modelled in literature, in \cite{BioPaper}, the authors consider models similar to 
\begin{equation}\label{eq:bioModel}
\begin{aligned}
\dot u + u\grad \cdot \mathbf V_\Omega - D_L\Delta u &= 0 &&\text{on $\Omega$}\\
D_L\nabla u \cdot \nu + u(\mathbf V_\Gamma \cdot \nu - \mathbf V_\Omega \cdot \nu) &= -k_{on}uw + k_{off}z &&\text{on $\Gamma$}\\
D_L\nabla u \cdot \nu &=0 &&\text{on $\partial_0 \Omega$}\\
\dot w+ w\sgrad \cdot \mathbf V_\Gamma - D_\Gamma \slap w &= -k_{on}uw + k_{off}z &&\text{on $\Gamma$}\\
\dot z+ z\sgrad \cdot \mathbf V_\Gamma - D_{\Gamma'} \slap z &= k_{on}uw - k_{off}z &&\text{on $\Gamma$}
\end{aligned}
\end{equation}
but with no velocities and on a stationary spherical domain in two and three dimensions. This system describes ligand-receptor dynamics in cells where the interaction between extracellular ligands $u$ and cell surface receptors $w$ leads (reversibly) to the creation of ligand-receptor complexes $z$. Here the various parameters appearing above are dimensional constants obtain from experimental data relating to a particular application. In Appendix \ref{sec:nondim} we will prove that non-dimensionalising \eqref{eq:bioModel} gives rise to the system \eqref{eq:modelALE} with the parametrisation velocity $\mathbf V_p$ chosen to coincide with the material velocities.


\subsection{Outline of the paper} 
In \S \ref{sec:existence} we prove existence of solutions (Theorem \ref{thm:existence}). The proof is divided in several steps which will be dealt with in separate subsections. We establish in \S \ref{sec:boundednessOfUandZ} the $L^{\infty}$ bounds on both $z$ and $u$, in this order (Theorem \ref{thm:LinftyBounds}) after proving a lemma that is used to show boundedness for $w$. The proofs are rather technical and are based on the De Giorgi $L^2$-$L^{\infty}$ method. In \S \ref{sec:strongSolutionsForu} we prove that the solution $u$ whose existence was established in \S \ref{sec:existence} is actually a strong solution (Theorem \ref{thm:uStrongSolution}). The proof is divided in several steps, each devoted to bounding a different term on the weak pulled-back formulation. In \S \ref{sec:equil} we deal with the stationary setting, when all the velocity fields are zero, and prove the exponential convergence to equilibrium (Theorem \ref{thm:equilibrium}). We conclude in \S \ref{sec:conclusion} with some open issues.

Appendix \ref{sec:derivation} contains the derivation of the system \eqref{eq:model}. In Appendix \ref{sec:prelim} we state some classical results in the form they will be used later on. We prove some auxiliary lemmas regarding interpolation inequalities and we also establish some basic but useful calculus identities that will be used throughout the paper. Finally, in Appendix \ref{sec:nondim} we give the details of the non-dimensionalisation of the system \eqref{eq:bioModel}.

\section{Existence}\label{sec:existence}

In this section we prove existence of solution to problem \eqref{eq:modelALE}. This will be established by following these steps:
\begin{itemize}
\item \textsc{Step 1.} Prove existence for the doubly truncated problem
\begin{equation}\label{eq:fullTruncatedProblem}
\begin{aligned}
\dot u + u\grad \cdot \mathbf V_p - \delta_\Omega \Delta u + \grad \cdot (\mathbf{J}_\Omega u)&=0\\
\delta_\Omega\nabla u \cdot \nu -  ju &=  \frac{1}{\delta_{k'}}T_n(z)-\frac{1}{\delta_k}uT_m(w^+)\\
\dot w + w\sgrad \cdot \mathbf V_p - \delta_\Gamma \slap w + \sgrad \cdot (\mathbf{J}_{\Gamma} w) &=  \frac{1}{\delta_{k'}}T_n(z)-\frac{1}{\delta_k}uT_m(w^+)\\
\dot z + z\sgrad \cdot \mathbf V_p - \delta_{\Gamma'}\slap z + \sgrad \cdot (\mathbf{J}_{\Gamma} z) &= \frac{1}{\delta_k}uT_m(w^+) - \frac{1}{\delta_{k'}}T_n(z)
\end{aligned}
\end{equation}
where $T_n(x) = \min(n, x) + \max(-n,x) -x $ is the usual truncation at height $n$
\item \textsc{Step 2.} Prove non-negativity of solutions to \eqref{eq:fullTruncatedProblem}
\item \textsc{Step 3.} Pass to the limit $n \to \infty$ in \eqref{eq:fullTruncatedProblem} and so removing the truncation on $z$
\item \textsc{Step 4.} Find an $L^\infty$ bound on $w$ so that we can set $m:= \norm{w}{\infty}$ and remove the truncation on $w$. This concludes the existence.
\end{itemize}
\subsection{Existence for the truncated problem}
We will now show that the truncated problem \eqref{eq:fullTruncatedProblem} has a solution $(u,w,z) \in H^1_{H^{1}(\Omega)^*}\cap L^2_{H^1(\Omega)} \times (H^1_{L^2(\Gamma)}\cap L^2_{H^1(\Gamma)})^2.$ Before we proceed let us just introduce some notation to ease readability. 
Define 
\[W_\Omega = H^1_{H^1(\Omega)^*}\cap L^2_{H^1(\Omega)}\qquad\text{with}\qquad \norm{u}{W_\Omega}^2 := \norm{u}{L^2_{H^1(\Omega)}}^2 + \norm{\dot u}{L^2_{H^1(\Omega)^*}}^2\]
and
\[W_\Gamma = H^1_{L^2(\Gamma)}\cap L^2_{H^1(\Gamma)}\qquad\text{with}\qquad \norm{w}{W_\Gamma}^2 := \norm{w}{L^2_{H^1(\Gamma)}}^2 + \norm{\dot w}{L^2_{L^2(\Gamma)}}^2.\] 
Let now $f, g \in L^2_{L^2(\Gamma)}$ be arbitrary and consider
\begin{align*}
\dot u + u\grad \cdot \mathbf V_p - \delta_\Omega \Delta u + \grad \cdot (\mathbf{J}_\Omega u) &=0\\
\delta_\Omega\nabla u \cdot \nu &=  \delta_{k'}^{-1}T_n(g) - \delta_k^{-1}uT_m(f^+) + ju\\
u(0) &= u_0.
\end{align*}
This is a linear problem so by \cite[Theorem 3.6]{AESAbstract} it has a unique solution $u \in W_\Omega$ which we will write as $u=U(f,g)$, where $U \colon (L^2_{L^2(\Gamma)})^2 \to W_\Omega$. Indeed, one can apply a Galerkin argument to prove the existence: pushing forward an orthogonal basis $\{b_j\}_{j \in \mathbb{N}}$ of $H^1(\Omega_0)$ with the inverse of $\phi_t$ (defined in section \ref{sec:weakFormulation}) will result in a basis $\{b_j(t)\}_{j \in \mathbb{N}}$ for $H^1(\Omega(t))$ that satisfies the useful \emph{transport property} $\dot b_j(t) = 0$. Then one uses the coercivity and boundedness of the elliptic bilinear form that arises in the weak formulation of the equation to obtain energy estimates, and a modification of the standard Galerkin argument yields the result. Alternatively, one could also make use of the Banach--Ne\v{c}as--Babu\v{s}ka theorem which can be thought of as a generalisation to mixed bilinear forms of the Lax--Milgram theory. Testing with $u$, we easily obtain a bound on $u$ in $L^2_{H^1(\Omega)}$ independent of $f$ and $g$. Taking the supremum of the duality pairing of $\dot u$ with a test function, we also obtain a bound on the weak time derivative. Combining, we find
\begin{equation}\label{eq:4}
\norm{U(f,g)}{W_\Omega}
\leq C_1
\end{equation}
with the constant independent of $f$ and $g$. Then we define $w=w(f,g)$ and $z=z(f,g)$ as the solutions in $W_\Gamma$ of 
\begin{align*}
\dot w + w\sgrad \cdot \mathbf V_p  - \delta_\Gamma \slap w + \sgrad \cdot (\mathbf{J}_{\Gamma} w) &= \delta_{k'}^{-1}T_n (g) - \delta_k^{-1}U(f,g)T_m(f^+)\\
\dot z + z\sgrad \cdot \mathbf V_p  - \delta_{\Gamma'} \slap z +\sgrad \cdot (\mathbf{J}_{\Gamma} z) &= \delta_k^{-1}U(f,g)T_m(f^+) - \delta_{k'}^{-1}T_n (g) \\
(w(0), z(0)) &= (w_0,z_0)
\end{align*}
which exist now by \cite[Theorem 3.13]{AESAbstract}. Define a mapping $\Theta \colon (L^2_{L^2(\Gamma)})^2 \to W_\Gamma^2$ taking the functions $f$ and $g$ to the solutions by $\Theta(f,g) = (w,z)$. We seek a fixed point of $\Theta$. Firstly, note that testing with $w$ and $\dot w$ and $z$ and $\dot z$ in the respective equations and using the bound on $U(f,g)$, we again obtain
\[\norm{w}{W_\Gamma} + \norm{z}{W_\Gamma} \leq C_2\]
with $C_2$ independent of $f$ and $g$. Therefore, defining the set 
\[E := \{(f,g) \in W_\Gamma^2 :  \norm{f}{W_\Gamma} + \norm{g}{W_\Gamma} \leq C_2\},\]
we have that $\Theta\colon E \to E$. Let us check that $\Theta$ is weakly continuous. For this purpose, let $(f_k, g_k) \weaklyto (f,g)$ in $W_\Gamma^2$ with $(f_k,g_k) \in E$. Defining $(w_k, z_k) = \Theta(f_k, g_k)$, we see that $(w_k, z_k)$ lies in $E$ and therefore $(w_k, z_k) \weaklyto (w,z)$ in $W_\Gamma^2$ for some $w$ and $z$, and we need to show that $\Theta(f,g) = (w,z)$. Define also $u_k = U(f_k, g_k)$ which again by \eqref{eq:4} is bounded independent of $k$ so $u_k \weaklyto u$ in $W_\Omega$. We need to pass to the limit in the weak formulation of following system
\begin{align*}
\dot u_k + u_k\grad \cdot \mathbf V_p - \delta_\Omega \Delta u_k + \grad \cdot (\mathbf{J}_\Omega u_k) &=0\\
\delta_\Omega\nabla u_k \cdot \nu &= \frac{1}{\delta_{k'}}T_n(g_k) -\frac{1}{\delta_k}u_kT_m(f_k^+) +  ju_k\\
\dot w_k + w_k\sgrad \cdot \mathbf V_p  - \delta_\Gamma \slap w_k + \sgrad \cdot (\mathbf{J}_{\Gamma} w_k) &= \frac{1}{\delta_{k'}}T_n(g_k) -\frac{1}{\delta_k}u_kT_m(f_k^+)\\
\dot z_k + z_k\sgrad \cdot \mathbf V_p  - \delta_{\Gamma'} \slap z_k +\sgrad \cdot (\mathbf{J}_{\Gamma} z_k) &= \frac{1}{\delta_k}u_kT_m(f_k^+) - \frac{1}{\delta_{k'}}T_n(g_k).
\end{align*}
It follows that for a subsequence $(f_k, g_k) \to (f,g)$ in $(L^2_{L^2(\Gamma)})^2$ by compact embedding  (see \ref{Lem:AubinLions}). By 
using the (Lipschitz) continuity of the truncation  in $L^2$ and the continuity of the trace map, we may send $k \to \infty$ to find
\begin{align*}
\dot u + u\grad \cdot \mathbf V_p - \delta_\Omega \Delta u + \grad \cdot (\mathbf{J}_\Omega u) &=0\\
\delta_\Omega\nabla u \cdot \nu &= \frac{1}{\delta_{k'}}T_n(g) -\frac{1}{\delta_k}uT_m(f^+) + ju\\
\dot w + w\sgrad \cdot \mathbf V_p  - \delta_\Gamma \slap w + \sgrad \cdot (\mathbf{J}_{\Gamma} w) &= \frac{1}{\delta_{k'}}T_n(g) -\frac{1}{\delta_k}uT_m(f^+)\\
\dot z + z\sgrad \cdot \mathbf V_p  - \delta_{\Gamma'} \slap z +\sgrad \cdot (\mathbf{J}_{\Gamma} z) &= \frac{1}{\delta_k}uT_m(f^+) - \frac{1}{\delta_{k'}}T_n(g)
\end{align*}
and this shows that $U(f,g)=u$ and $\Theta(f,g) = (w,z)$ as required. We have shown that this holds for a subsequence but it is also true for the full sequence by a standard argument (for example see \cite{AEStefan}). Therefore, $\Theta$ has a fixed point and we have a solution of the truncated problem \eqref{eq:fullTruncatedProblem}.
\subsection{Non-negativity of the truncated problem}
In this section we prove that the solution $(u,w,z)$ found in the previous section is non-negative; this is expected since the initial data are non-negative. Testing the $z$ and $u$ equations with $z^-$ and $u^-$ respectively with the help of \eqref{eq:surfaceIBPidentity} and \eqref{eq:bulkIBPidentity}, and adding yields 
\begin{align*}
&\frac{d}{dt}\left(\int_{\Gamma(t)}| z^-|^2 + \int_{\Omega(t)}| u^-|^2\right) + \int_{\Gamma(t)} |z^-|^2{\sgrad \cdot \mathbf V_\Gamma}  + \int_{\Omega(t)}| u^-|^2{\grad \cdot \mathbf V_\Omega} + 2\delta_{\Gamma'}\int_{\Gamma(t)}|\sgrad  z^-|^2 +   2\delta_\Omega \int_{\Omega(t)}|\nabla  u^-|^2\\
&\leq 2\int_{\Gamma(t)} \delta_k^{-1}uT_m(w^+) z^{-} + \delta_{k'}^{-1}T_n (z) u^- + \frac{j}{2}| u^-|^2,
\end{align*}
We have used the fact that the first two integrals on the left hand side are differentiable (since the inner product $(u(t), v(t))_{L^2(\Gamma(t))}$ is differentiable in time for functions in $L^2_{H^1} \cap H^1_{H^{-1}}$ and the claim follows by density, see Lemma 2.11 of \cite{AEStefan} for example).
The first two terms on the right hand side can be estimated as follows by splitting the domain of integration:
\begin{align*}
\int_{\Gamma(t)} \delta_k^{-1} uT_m(w^+) z^{-} + \delta_{k'}^{-1}T_n (z) u^- &\leq 
\int_{\{z \leq 0\}} (\delta_k^{-1} uT_m(w^+) z^{-} +  \delta_{k'}^{-1}T_n (z^-) u^-)\\
&\leq \int_{\{z \leq 0\}} (\delta_k^{-1} u^-T_m(w^+) z^{-} + \delta_{k'}^{-1}T_n (z^-) u^-)\tag{writing $u=u^+ + u^-$}\\
&\leq C\left(\norm{ u^-}{L^2(\Gamma(t))}^2 + \norm{ z^-}{L^2(\Gamma(t))}^2\right).
\end{align*}
Moving the terms involving the velocity to the right hand side, via the interpolated trace inequality we estimate
\begin{align*}
\int_{\Gamma(t)} &\delta_k^{-1} uT_m(w^+) z^{-} + \delta_{k'}^{-1}T_n (z) u^- + \frac{j}{2}| u^-|^2 + \frac{\norm{\sgrad \cdot \mathbf V_\Gamma}{\infty}}{2}\int_{\Gamma(t)} |z^-|^2 + \frac{\norm{\grad \cdot \mathbf V_\Omega}{\infty}}{2}\int_{\Omega(t)} |u^-|^2\\
&\leq \epsilon\norm{\nabla  u^-}{L^2(\Omega(t))}^2 + C_\epsilon\norm{ u^-}{L^2(\Omega(t))}^2 + C\norm{ z^-}{L^2(\Gamma(t))}^2.
\end{align*}
This implies that if $\epsilon$ is small enough, 
\begin{align*}
\frac{d}{dt}\left(\int_{\Gamma(t)}| z^-|^2  +  \int_{\Omega(t)}| u^-|^2\right) \leq C\left(\norm{ u^-}{L^2(\Omega(t))}^2 + \norm{ z^-}{L^2(\Gamma(t))}^2\right)
\end{align*}
then Gronwall's inequality gives $ z^-$ and $ u^-$ are zero, so that $z, u \geq 0$. This easily implies that $w \geq 0$ after testing the equation for $ w$ with $ w^{-}$ 
and using Gronwall again.
\subsection{Passing to the limit in $n$}
Now we wish to drop the truncation in $z$. To achieve this, test the $u$, $w$ and $z$ equations with $u$, $w$ and $z$ respectively, use on the right hand sides the estimates
\begin{align*}
2\delta_{k'}^{-1}T_n (z)u + ju^2 &\leq \frac{1}{\delta_{k'}}z^2 + \left(\norm{j}{\infty} + \frac{1}{\delta_{k'}}\right)u^2\\
T_n (z)w &\leq \frac{1}{2}(w^2 + z^2)\\
uT_m(w)z &\leq \frac{m}{2}(u^2 + z^2)
\end{align*}
and combine to find
\begin{align*}
&\frac{1}{2}\frac{d}{dt}\left(\int_{\Omega(t)}u^2+\int_{\Gamma(t)}w^2 +\int_{\Gamma(t)}z^2\right) +   \delta_\Omega \int_{\Omega(t)}|\nabla u|^2 + \delta_{\Gamma} \int_{\Gamma(t)}|\sgrad w|^2 + \delta_{\Gamma'} \int_{\Gamma(t)}|\sgrad z|^2\\
&\leq \left(\frac{1}{2\delta_{k'}} + \frac{\norm{\sgrad \cdot \mathbf V_\Gamma}{\infty}}{2}\right)\int_{\Gamma(t)}w^2 + \left(\frac{1}{2\delta_{k'}} +\frac{1}{2\delta_{k'}} + \frac{m}{2\delta_k} + \frac{\norm{\sgrad \cdot \mathbf V_{\Gamma}}{\infty}}{2}\right)\int_{\Gamma(t)}z^2\\
&\quad + \left(\frac 12\norm{j}{\infty} + \frac{1}{2\delta_{k'}} + \frac{m}{2\delta_k}\right)\int_{\Gamma(t)}u^2 + \left( \frac{\norm{\grad \cdot \mathbf V_\Omega}{\infty}}{2}\right)\int_{\Omega(t)}u^2.
\end{align*}
Now using the interpolated trace inequality with $\epsilon = \delta_\Omega\slash(\norm{j}{\infty} + \frac{1}{\delta_{k'}} + \frac{m}{\delta_k}),$ we obtain
\begin{equation}\label{eq:boundOnn}
\norm{u}{L^\infty_{L^2(\Omega)}} + \norm{w}{L^\infty_{L^2(\Gamma)}} + \norm{z}{L^\infty_{L^2(\Gamma)}} + \norm{\grad u}{L^2_{L^2(\Omega)}} + \norm{\grad w}{L^2_{L^2(\Gamma)}}  + \norm{\grad z}{L^2_{L^2(\Gamma)}} \leq C(\delta_\Omega, \delta_k, \delta_{k'}, \delta_\Gamma, \delta_{\Gamma'}, m)
\end{equation}
independent of $n$ (recall that the triple $(u,w,z)$ depends on $n$, the truncation height of $z$). We also have, using the trace inequality,
\begin{align*}
\langle \dot u, \eta \rangle 
&\leq (\norm{\grad \cdot \mathbf V_\Omega}{\infty} + \norm{\grad \cdot \mathbf{J}_\Omega}{\infty})\norm{u}{L^2(\Omega(t))}\norm{\eta}{L^2(\Omega(t))}  + (\delta_\Omega+C_J) \norm{u}{H^1(\Omega(t))} \norm{\eta}{H^1(\Omega(t))}\\
&\quad +
C(\delta_{k'}^{-1}\norm{z}{L^2(\Gamma(t))} + m\delta_k^{-1}\norm{u}{L^2(\Gamma(t))} + \norm{j}{\infty}\norm{u}{L^2(\Gamma(t))})\norm{\eta}{H^1(\Omega(t))}
\end{align*}
where $C_J$ is a constant depending on $\mathbf{J}_\Omega$. Integrating and taking the supremum over $\eta \in L^2_{H^1(\Omega)}$, we find
\[\norm{\dot u}{L^2_{H^{-1}(\Omega)}} \leq C(\delta_\Omega, \delta_{k}, \delta_{k'}, m).\]
Testing the $w$ equation with $\dot w$ and using $T_n(z) \leq z$, and likewise with the $z$ equation, we obtain, after some manipulation,
\[\norm{\dot w}{L^2_{L^2(\Gamma)}} + \norm{\dot z}{L^2_{L^2(\Gamma)}} \leq C(\delta_\Omega, \delta_{k}, \delta_{k'}, m).\]
Using these uniform estimates and the Aubin--Lions-type compactness theorem in appendix \ref{Lem:AubinLions}, we may pass to the limit in $n$ and we shall find existence of
\begin{equation*}\label{eq:truncatedProblem}
\begin{aligned}
\dot u + u\grad \cdot \mathbf V_p - \delta_\Omega \Delta u + \grad \cdot (\mathbf{J}_\Omega u)&=0\\
\delta_\Omega\nabla u \cdot \nu - ju &= \frac{1}{\delta_{k'}}z-\frac{1}{\delta_k}uT_m(w)\\
\dot w + w\sgrad \cdot \mathbf V_p - \delta_\Gamma \slap w + \sgrad \cdot (\mathbf{J}_{\Gamma} w) &=  \frac{1}{\delta_{k'}}z-\frac{1}{\delta_k}uT_m(w) \\
\dot z + z\sgrad \cdot \mathbf V_p - \delta_{\Gamma'}\slap z + \sgrad \cdot (\mathbf{J}_{\Gamma} z) &= \frac{1}{\delta_k}uT_m(w) - \frac{1}{\delta_{k'}}z.
\end{aligned}
\end{equation*}
It remains for us to remove the truncation on $w$ by showing that $w$ is bounded.
\subsection{Proof of existence}\label{sec:subBoundForw}
Before we prove Theorem \ref{thm:existence}, let us remark that the $L^\infty$ bounds for $w$ (and $z$) are easy to obtain if the diffusion constants are the same for the two surface equations: one simply tests the equation satisfied by $v$ by $(v-M)^+$ for $M:=\norm{w_0}{L^\infty(\Gamma_0)} + \norm{z_0}{L^\infty(\Gamma_0)}$ and one finds a bound on $v$ (and by non-negativity, on $w$ and $z$) which depends only on $T$ and the initial data. In this simple case, we would not have needed to truncate $z$ either earlier. Clearly, in the general $\delta_\Gamma \neq \delta_{\Gamma'}$ case a weak maximum principle technique cannot work in trying to procure an $L^\infty$ bound on $w$ since we do not yet know if $z$ is bounded or not. However, we know by iteration arguments that the $L^\infty$ bound on parabolic equations should depend only on $L^r_{L^q}$ bounds on the (positive part of the) right hand side. Indeed, classical results on stationary domains \cite[\S III.7]{Lady}, \cite{aronson1963, Aronson1967, LU} lead us to expect an $L^\infty$ bound on solutions to linear parabolic equations with the right hand side $f \in L^{r}_{L^q(\Gamma)}$ if the condition
\begin{equation}\label{eq:maximalRegularity}
\frac 1r + \frac{d}{2q} < 1
\end{equation}
is satisfied. In our case, the corresponding right hand side is the variable $z$. But such bounds on $z$ depend on $w$ (due to the $uT_m(w)$ term in the $z$ equation) so this leads to a circular argument. However, if we rewrite the $w$ equation so as to eliminate $z$ by using the substitution $v = w+z$ (see \eqref{eq:crossDiff}) we can eventually obtain what we want. Let us proceed by recalling that $v$ solves
\begin{align*}
\dot v + v\sgrad \cdot \mathbf V_p + \sgrad \cdot (\mathbf{J}_{\Gamma}v) &= \delta_{\Gamma'}\slap v + (\delta_\Gamma - \delta_{\Gamma'})\slap w.
\end{align*}
\begin{lem}\label{lem:energyEstimates}The following energy estimate holds independent of $\delta_k$:
\[\norm{v}{L^\infty_{L^2(\Gamma)}} + \norm{w}{L^\infty_{L^2(\Gamma)}} + \frac{A}{\delta_k}\norm{\sqrt{u}w}{L^2_{L^2(\Gamma)}} + \delta_{\Gamma'}\norm{\sgrad v}{L^2_{L^2(\Gamma)}} + C_0\norm{\sgrad w}{L^2_{L^2(\Gamma)}} \leq C(T, \delta_{k'}, A(\delta_{\Gamma}, \delta_{\Gamma'}))\]
where $A \geq \max\left(1, \frac{C_0}{2\delta_\Gamma} + \frac{(\delta_\Gamma - \delta_{\Gamma'})^2}{2\delta_\Gamma\delta_{\Gamma'}}\right)$ and $C_0 \geq 0$ is chosen arbitrarily. 
\end{lem}
\begin{proof}
Test the equation for $v$ with $v$ itself to find
\begin{align*}
\frac{d}{dt}\int_{\Gamma(t)} v^2 + \int_{\Gamma(t)} v^2\sgrad \cdot \mathbf V_\Gamma + 2\delta_{\Gamma'}\int_{\Gamma(t)} |\sgrad v|^2 \leq 2|\delta_\Gamma - \delta_{\Gamma'}| \int_{\Gamma(t)} \epsilon |\sgrad v|^2 + C_\epsilon|\sgrad w|^2.
\end{align*}
Test also the $w$ equation with $w$ and multiply by a number $A>0$ to be determined to obtain
\begin{align*}
A\frac{d}{dt}\int_{\Gamma(t)} w^2 + A\int_{\Gamma(t)} w^2\sgrad \cdot \mathbf V_\Gamma + 2A\delta_{\Gamma}\int_{\Gamma(t)} |\sgrad w|^2 
= \frac{2A}{\delta_{k'}}\int_{\Gamma(t)} (v - w)w -\frac{2A}{\delta_k}\int_{\Gamma(t)} uw^2.
\end{align*}
Now pick $\epsilon = \delta_{\Gamma'}\slash 2|\delta_{\Gamma}-\delta_{\Gamma'}|$ and add the two inequalities together:
\begin{align*}
\frac{d}{dt}\int_{\Gamma(t)} (v^2 + Aw^2) + \int_{\Gamma(t)} (v^2+ Aw^2)\sgrad \cdot \mathbf V_\Gamma + \frac{2A}{\delta_k}\int_{\Gamma(t)} uw^2 + \delta_{\Gamma'}\int_{\Gamma(t)} |\sgrad v|^2  + 2A\delta_{\Gamma}\int_{\Gamma(t)} |\sgrad w|^2\\
 \leq 2|\delta_\Gamma - \delta_{\Gamma'}| \int_{\Gamma(t)} C_\epsilon|\sgrad w|^2 + \frac{2A}{\delta_{k'}}\int_{\Gamma(t)} vw.
\end{align*}
Recall that $C_\epsilon = (4\epsilon)^{-1}$. Now, for any $C_0 \geq 0$, pick $A$ such that $A \geq \max\left(1, \frac{C_0}{2\delta_\Gamma} + \frac{(\delta_\Gamma - \delta_{\Gamma'})^2}{2\delta_\Gamma\delta_{\Gamma'}}\right)$, then 
\begin{align*}
&\frac{d}{dt}\int_{\Gamma(t)} (v^2 + Aw^2) + \frac{2A}{\delta_k}\int_{\Gamma(t)} uw^2  + \delta_{\Gamma'}\int_{\Gamma(t)} |\sgrad v|^2  + C_0\int_{\Gamma(t)} |\sgrad w|^2\\
&\leq \frac{A }{\delta_{k'}}\int_{\Gamma(t)} v^2 + w^2 + \norm{\sgrad \cdot \mathbf V_\Gamma}{\infty}\int_{\Gamma(t)} (v^2+ Aw^2)\\
&\leq \left(\frac{A }{\delta_{k'}} + \norm{\sgrad \cdot \mathbf V_\Gamma}{\infty}\right)\int_{\Gamma(t)} (v^2+ Aw^2)\tag{using $1 \leq A$},
\end{align*}
which, with $\alpha := \left(A \slash \delta_{k'} + \norm{\sgrad \cdot \mathbf V_\Gamma}{\infty}\right)$
yields
\[\norm{v(t)}{L^2(\Gamma(t))}^2 + A\norm{w(t)}{L^2(\Gamma(t))}^2 \leq e^{\alpha T}\int_{\Gamma_0} (v_0^2 + Aw_0^2)\]
and
\[\frac{2A}{\delta_k}\norm{\sqrt{u}w}{L^2_{L^2(\Gamma)}}^2 + \delta_{\Gamma'}\norm{\sgrad v}{L^2_{L^2(\Gamma)}}^2 + C_0\norm{\sgrad w}{L^2_{L^2(\Gamma)}}^2 \leq \left(T\alpha e^{\alpha T} + 1\right)\int_{\Gamma_0} (v_0^2 + Aw_0^2).\]
\end{proof}
We are ready to conclude the existence.
\begin{proof}[Proof of Theorem \ref{thm:existence}]
The equation for $w$, with the right hand side written in terms of $v$ reads
\[\dot w + w\sgrad \cdot \mathbf V_p - \delta_\Gamma \slap w + \sgrad \cdot (\mathbf{J}_{\Gamma}w) =  \left(\frac{1}{\delta_{k'}}v- \frac{1}{\delta_{k'}}w-\frac{1}{\delta_k}uT_m(w) \right).
\]
Lemma \ref{lem:deGiorgi} below guarantees an $L^\infty$ bound on $w$  for dimensions $d\leq 3$ whose dependence on the right hand side data is only on the $L^\infty_{L^2(\Gamma)}$ norm of the positive part of the right hand side. Since we proved in Lemma \ref{lem:energyEstimates} that the $L^\infty_{L^2(\Gamma)}$ norm of $\delta_{k'}^{-1}v$ is bounded independently of $z$ and $u$, it follows that if we pick the truncation level of $w$ to be $m:=\norm{w}{L^\infty_{L^\infty(\Gamma)}}$ we obtain existence for \eqref{eq:model}.
\end{proof}
Actually, we know from Lemma \ref{lem:energyEstimates} that $v$ is bounded in the space $Q(\Gamma)$, not just in $L^\infty_{L^2(\Gamma)}$, so one might expect to profit using this extra information somehow in the existence or boundedness result. We address this question in Remark \ref{rem:gain}.
\section{$L^\infty$ bounds}\label{sec:boundednessOfUandZ}
In this section, we will prove essential boundedness for $z$ and $u$, as well as proving a lemma that we used to prove boundedness of $w$ (and existence for the system) in the previous section. We start with $L^\infty$ bounds for solutions of an abstract linear parabolic PDE on an evolving surface with $L^p$ right hand side data that we can apply to the $w$ and $z$ equations. Clearly these results can also be used for many other problems on evolving surfaces.
\medskip

\noindent\textbf{De Giorgi method.} Since we will be using variations of the De Giorgi $L^2$--$L^{\infty}$ scheme to prove the $L^{\infty}$ bounds and the proofs are rather technical, it is worth emphasising the key steps. This method was introduced by De Giorgi \cite{DeGiorgiEnnio} in his paper regarding the regularity of solutions to nonlinear elliptic problems. His technique allowed him to obtain boundedness and H\"older regularity of solutions with only $L^2$ \emph{a priori} estimates. This method has since then been applied to several other problems, including parabolic equations. The starting point is to define a sequence of positive numbers, say $U_k$. This is usually related to the $L^2$ norm of $u_k=(u-c_k)^+$, if $u$ is the function that we wish to bound. Then one has to establish a nonlinear recurrence estimate of the form
\begin{equation}\label{L-U}
U_k\leq C \beta^k U_{k-1}^{1+s}, \, \text{ where } \, C, \beta, s>0.
\end{equation}
Usual tools to achieve such estimate include Sobolev inequalities, energy estimates and the Chebyshev inequality. Then, if $U_0$ is small enough, namely if $U_0\leq \min(1, (2C)^{-1/s}\beta^{-1/s^2}),$ then $U_k$ converges to zero as $k$ tends to infinity. If also $c_k$ converges to some number $M$ we can say that $u\leq M$ in some appropriate set, depending on the definition of $U_k$. 

There is another formulation to this method, which relies on a different recurrence estimate. We now define a function $U\geq 0$ to be nonincreasing in $[\overline{x},\infty)$ and such that, for $y\geq x\geq \overline{x}$,
\begin{equation}\label{Stamp}
(y-x)^pU(y)\leq CU(x)^{1+s} \, \text{ where } \, C, p, s>0.
\end{equation}
Then $U(y)=0$ for $y\geq  \overline{x} + B$ where $B^p=CU(0)^s2^{\frac{p(s+1)}{s}}$. This result is due to Stampacchia \cite{Stampacchia1965}. To suit our purpose $U(y)$ will be related to the measure of the set $\{u>y\}$, where again $u$ is the function we wish to bound.

\subsection{Boundedness for parabolic problems on evolving surfaces}
We consider the following initial value problem set on $\Gamma(t)$
\begin{equation}\label{eq:linearIVP}
\begin{aligned}
\dot \labelForVar + \labelForVar\sgrad \cdot \mathbf V_p - D\Delta_\Gamma \labelForVar + \sgrad \cdot (\mathbf J_\Gamma \labelForVar) &= g\\
\labelForVar(0) &= \labelForVar_0
\end{aligned}
\end{equation}
for a given constant $D > 0$ and data $g \in L^2_{L^2(\Gamma)}$ and $a_0 \in L^2(\Gamma_0)$.

\begin{lem}[De Giorgi I]\label{lem:deGiorgi}For dimensions $d \leq 3$, the weak solution of the equation \eqref{eq:linearIVP}
given $g \in L^2_{L^2(\Gamma)}$ with $g^+ \in L^\infty_{L^2(\Gamma)}$  and $\labelForVar_0 \in L^\infty(\Gamma_0)$ satisfies
\[ 
\norm{\labelForVar}{L^\infty_{L^\infty(\Gamma)}} \leq e^{(\norm{\sgrad \cdot \mathbf V_\Gamma}{\infty} + D)T}\left(\norm{\labelForVar_0}{L^\infty(\Gamma_0)} + CD^{-1}\norm{g^+}{L^\infty_{L^{2}}}\right).
\]
\end{lem}
\begin{proof}
Let us consider the transformed problem
\begin{align*}
\dot{\tilde \labelForVar} + \tilde \labelForVar(\sgrad \cdot \mathbf V_p + \lambda)  - D\Delta_\Gamma \tilde \labelForVar + \sgrad \cdot (\mathbf J_\Gamma\tilde \labelForVar)  &= e^{-\lambda t}g =: f\\
\tilde \labelForVar(0) &= \labelForVar_0,
\end{align*}
where $\tilde{a}=e^{-\lambda t}a$. For convenience, we do not write the tildes from now on. Testing with $a_k:=(\labelForVar-k)^+$ for $k \geq k_0 := \norm{a_0}{L^\infty(\Gamma_0)}$ and integrating by parts using \eqref{eq:surfacePosId}, we find
\begin{align*}
\frac{1}{2}\frac{d}{dt}\int_{\Gamma(t)} |a_k|^2 + \int_{\Gamma(t)} \labelForVar a_k(\sgrad \cdot \mathbf V_\Gamma + \lambda) -\frac 12 |a_k|^2 \sgrad \cdot \mathbf V_\Gamma  +  D|\sgrad a_k|^2 
&\leq \int_{\Gamma(t)} f^+a_k.
\end{align*}
Let $\lambda = D+ \norm{\sgrad \cdot \mathbf V_\Gamma}{\infty}$. Then $\sgrad \cdot \mathbf V_\Gamma +\lambda \geq D$. Using $aa_k = a_k^2 + ka_k \geq a_k^2$ and integrating in time we obtain
\begin{align}
\frac{1}{2}\int_{\Gamma(t)} |a_k(t)|^2 - \frac{1}{2}\int_{\Gamma(s)} |a_k(s)|^2  + D\int_s^t\int_{\Gamma(\tau)} \left(|a_k|^2+ |\sgrad a_k|^2\right)
\leq \int_s^t \int_{\Gamma(\tau)} f^+a_k\label{eq:2}.
\end{align}
Since the function $I_k(t) := \int_{\Gamma(t)}|a_k(t)|^2$ is an inner product of functions in $L^2_{H^1(\Gamma)} \cap H^1_{L^2(\Gamma)}$, it is continuous,  and hence it has a maximum, say at $t=\sigma > 0.$ Now, let us pick a sequence $\delta_n \to 0$ such that, for $\sigma_n := \sigma - \delta_n$, the following hold:
\begin{enumerate}
\item $\norm{f^+(\sigma_n)}{L^2(\Gamma(\sigma_n))} \leq \norm{f^+}{L^\infty_{L^2(\Gamma)}}$ 
\item $\sigma_n$ is a Lebesgue point for the function 
\[G(\tau):=\int_{\Gamma(\tau)} f^+(\tau)a_k(\tau)-D \left(|\sgrad a_k(\tau)|^2 + |a_k(\tau)|^2\right).\]
(If $\sigma$ is a Lebesgue point for the above function then we can just set $\delta_n \equiv 0$; otherwise we know the set of Lebesgue points is dense and therefore we can choose $\sigma_n$)
\item $I_k(\sigma_n) - I_k(\sigma_n-\epsilon) \geq 0$
(this can be done if $\delta_n$ and $\epsilon$ are small enough, since $I_k$ is continuous, starts at zero and has a maximum on $(0,T]$).
\end{enumerate} Choosing in \eqref{eq:2} $t=\sigma_n$ and $s=\sigma_n - \epsilon$, we obtain
\begin{align*}
D\int_{\sigma_n - \epsilon}^{\sigma_n}\int_{\Gamma(\tau)} \left(|a_k|^2 +  |\sgrad a_k|^2\right)
\leq \int_{\sigma_n - \epsilon}^{\sigma_n} \int_{\Gamma(\tau)} f^+a_k
\end{align*}
and now divide by $\epsilon$ and send $\epsilon \to 0$ with the Lebesgue differentiation theorem to find
\begin{align}
D\norm{ a_k}{H^1(\Gamma(\sigma_n))}^2
\leq \int_{\Gamma(\sigma_n)} f^+a_k.\label{eq:usedEq}
\end{align}
Define $A_k(t) := \{ x \in \Gamma(t) \mid \labelForVar(t,x) > k\}$ and $\mu_k := \esssup_{0 \leq t \leq T} |A_k(t)|$. Clearly $\mu_k\geq 0$ is non-increasing for $k\geq \norm{a_0}{L^{\infty}(\Gamma_0)}$.  Using the Sobolev inequality \eqref{eq:sobolevInequality}, the above yields
\begin{align*}
DC_I^2\norm{a_k}{L^p(\Gamma(\sigma_n))}^p 
\leq \left(\int_{\Gamma(\sigma_n)\cap A_k(\sigma_n)}|f^+|^{p'}\right)^{\frac{1}{p'}}\norm{a_k}{L^p(\Gamma(\sigma_n))}
\end{align*}
where $p'$ is the conjugate to $p$ and $p$ satisfies the constraint in \eqref{eq:sobolevInequality}.
 Therefore,
\begin{align*}
DC_I^2\norm{a_k}{L^p(\Gamma(\sigma_n))}
&\leq |A_k(\sigma_n)|^{\frac{1}{rp'}}\norm{f^+(\sigma_n)}{L^{r'p'}(\Gamma(\sigma_n))}\\
& \leq \mu_k^{\frac{1}{rp'}}\norm{f^+}{L^\infty_{L^{r'p'}}}
\end{align*}
where $r \geq 1$ is arbitrary for now. Now we wish to establish a similar inequality for $I_k(\sigma_n)$. Using Holder's inequality with $(p\slash 2, p\slash (p-2))$ and the previous estimate,
\[I_k(\sigma_n) = \int_{\Gamma(\sigma_n)}|a_k|^2\\
\leq \left(\int_{A_k(\sigma_n)}1^{p\slash (p-2)}\right)^{\frac{p-2}{p}}\left(\int_{\Gamma(\sigma_n)}|a_k|^{p}\right)^{\frac{2}{p}}\\
\leq \frac{1}{D^2C_I^4}\mu_k^{\frac{2}{rp'} + \frac{p-2}{p}}\norm{f^+}{L^\infty_{L^{r'p'}}}^2.\]
Sending $\delta_n \to 0$ and using continuity of $I_k$ we obtain
\begin{align*}
I_k(\sigma)
&\leq \frac{1}{D^2C_I^4}\mu_k^{\frac{2}{rp'} + \frac{p-2}{p}}\norm{f^+}{L^\infty_{L^{r'p'}}}^2.
\end{align*}
Observe that for any $h \geq k$,
\[I_k(t) = \int_{A_k(t)}|a_k|^2 \geq \int_{A_h(t)}|a_k|^2 \geq (h-k)|^2|A_h(t)|\]
which implies that 
\begin{align*}
\mu_h \leq \frac{\norm{f^+}{L^\infty_{L^{p'r'}}}^2}{D^2C_I^4(h-k)^2}\mu_k^{\frac{2}{rp'} + \frac{p-2}{p}}
\end{align*}
which can be written as $(h-k)^2\mu_h\leq C\mu_k^{\gamma}$ where $C= \norm{f^+}{L^\infty_{L^{p'r'}}}^2\slash (D^2C_I^4)$ and $\gamma= {2}\slash (rp') + (p-2)\slash {p}$. This is the setting of the Stampacchia lemma \cite{Stampacchia1965}, see \eqref{Stamp}. We need $\gamma>1$ and this is satisfied if $r < p-1$. Moreover, given the data to the problem we must impose $p'r'\leq 2$. This poses no problem in dimensions $d\leq 2$ since in that case $p$ can be arbitrary as given by the Sobolev embedding. However, for higher dimensions, $p=2d/(d-2)$, hence, in order to satisfy the conditions above we must restrict ourselves to $d<4$. 
Therefore, $\mu_h =0$ for all $h \geq \norm{\labelForVar_0}{L^\infty(\Gamma_0)} + B$ where
\[B =\frac{\norm{f^+}{L^\infty_{L^{r'p'}}}}{DC_I^2}2^{\gamma/(\gamma-1)}|\Gamma|^{(\gamma-1)/2} 
\]
which directly implies that (back to the tilde notation) $\tilde \labelForVar \leq \norm{\labelForVar_0}{L^\infty(\Gamma_0)} + B.$ Transforming back, we find actually that
\[\labelForVar(t) \leq e^{\lambda t}\left(\norm{\labelForVar_0}{L^\infty(\Gamma_0)} + \frac{\norm{e^{-\lambda t}g^+}{L^\infty_{L^{r'p'}}}}{DC_I^2}2^{(\frac{2}{rp'}+\frac{p-2}{p})\slash (\frac{2}{rp'}-\frac 2p)}|\Gamma|^{\frac{1}{rp'}-\frac 1p}\right).\]
\end{proof}

\begin{remark}\label{rem:gain}As we wrote at the end of \S \ref{sec:existence}, Lemma \ref{lem:energyEstimates} shows that $v$ is bounded in the space $Q(\Gamma)$ and not just in $L^\infty_{L^2(\Gamma)}$, so we might expect a less strict restriction on the dimension for Theorem \ref{thm:existence} were we to use this in the De Giorgi method of the previous lemma to bound $w$ (see \S \ref{sec:subBoundForw}). Were we to utilise the embedding \eqref{eq:gn} instead of the Sobolev inequality \eqref{eq:sobolevInequality} on \eqref{eq:usedEq}, we would lose the $L^\infty$ requirement in time on the right hand side data and instead require a bound in $L^r_{L^q}$ for $r, q > 2$ (the downside is the increased spatial regularity), with $r$ and $q$ related through the condition in Lemma \ref{lem:gn}. However, the condition \eqref{eq:maximalRegularity} for such a right hand side $f$ in $L^r_{L^q(\Gamma)}$ again just translates to requiring $d \leq 3$, so we do not gain anything by using \eqref{eq:gn}.
\end{remark}
One might think that we could prove the bound on $z$ using the previous lemma just like it was used to prove the bound on $w$. The positive part of the right hand side of the $z$ equation is precisely $\delta_{k}^{-1}uw$, which satisfies
\[\frac{1}{\delta_{k}}uw \leq \frac{\norm{w}{L^\infty_{L^\infty(\Gamma)}}}{\delta_{k}}u.\]
However, $u$ is known only to be in $L^\infty_{L^2(\Omega)}$ whereas we need it bounded in $L^\infty_{H^1(\Omega)}$ (by the trace theorem) in order to apply Lemma \ref{lem:deGiorgi}. Therefore, we need a different stipulation answered by the next lemma (but note the condition on the dimension $d$ cf. the previous lemma). Again we will employ a De Giorgi iteration scheme of the type discussed earlier (see \eqref{L-U}).

\begin{lem}[De Giorgi II]\label{lem:deGiorgiSecond}For dimensions $d \leq 2$, the weak solution of the equation \eqref{eq:linearIVP}
given $g \in L^2_{L^2(\Gamma)}$ with $g^+ \in L^2_{L^{2+\epsilon}(\Gamma)}$  and $\labelForVar_0 \in L^\infty(\Gamma_0)$ satisfies
\[\norm{a}{L^\infty_{L^\infty(\Gamma)}} \leq e^{(\norm{\sgrad \cdot \mathbf V_\Gamma}{\infty}+D)T}\left(\norm{a_0}{L^\infty(\Gamma_0)}+ C(T)\min(1\slash 2, D)^{-1}\norm{g^+}{L^2_{L^{2+\epsilon}(\Gamma)}}\right)\]
for almost every $t \in [0,T]$, where $C$ is independent of $t$.
\end{lem}
\begin{proof}
We will adapt \cite[\S III.7]{Lady} for our setting of an evolving surface and our specific choice of exponents on the right hand side data. Let us suppose that $a_0=0$ for now. Like the proof of Lemma \ref{lem:deGiorgi}, we may transform, relabel and manipulate the equation so that from \eqref{eq:2}, we have for $k \geq k_0 := \norm{a_0}{L^\infty(\Gamma_0)}$,
%
%
\[ 
\min\left(\frac{1}{2},D\right)\norm{a_k}{Q(\Gamma)}^2\leq \norm{f^+}{L^2_{L^{2+\epsilon}}}\norm{a_k}{L^{2}_{L^{\frac{2+\epsilon}{1+\epsilon}}}}.
\]
The spaces involved in the $a_k$ norm on the right hand side do not satisfy the conditions of Lemma \ref{lem:gn} so we will apply first H\"older's inequality with suitable exponents. If we define $A_k(t)=\{x\in\Gamma(t)\mid a(t,x)>k\}$ (as before),
\[
\norm{a_k}{L^{2}_{L^{\frac{2+\epsilon}{1+\epsilon}}}}\leq \left(\int |A_k(t)|^{\frac{2(1+\epsilon)}{2+\epsilon}}\right)^\frac{1}{2\lambda'}\norm{a_k}{L^{2\lambda}_{L^{\frac{2+\epsilon}{1+\epsilon}\lambda}}},
\]
where $\lambda=2(1+\frac 1d)-\frac{2}{2+\epsilon}$. By Lemma \ref{lem:gn}, the norm on the right hand side becomes
\[
\norm{a_k}{L^{2\lambda}_{L^{\frac{2+\epsilon}{1+\epsilon}\lambda}}}\leq C_1(T)\norm{a_k}{Q(\Gamma)}.
\]
Putting it all together we have, denoting $m=\min(1/2,D)$
\[
\norm{a_k}{Q(\Gamma)}\leq \frac{C_1(T)}{m}\norm{f^+}{L^2_{L^{2+\epsilon}}}\left(\int |A_k(t)|^{\frac{2(1+\epsilon)}{2+\epsilon}}\right)^\frac{1}{2\lambda'}.
\]
Set now $k_n=(2-2^{-n})N$ for some large $N$ to be defined later. Note that $k_0=N$ and that $k_n\uparrow 2N$ as $n$ tends to infinity. Finally, define
\[
z_n=\left(\int_0^T |A_{k_n}(t)|^{\frac{2(1+\epsilon)}{2+\epsilon}}\right)^\frac{1}{\lambda}.
\]
This sequence of positive numbers will play the role of $U_k$ as presented in \eqref{L-U}. From the previous inequality we obtain
\[
\norm{a_k}{Q(\Gamma)} \leq \frac{C_1(T)}{m}\norm{f^+}{L^2_{L^{2+\epsilon}}}z_n^{\frac{\lambda-1}{2}}.
\]
Now, from the definition of $k_n$ we see that if $a>k_n$ then $a>k_{n-1} + 2^{-n}N$ and this implies
$\{ a> k_n\} \subset \{a_{k_{n-1}} > 2^{-n}N\}$
So, using Chebyshev's inequality
for some $r$ to be defined later,
\[|\{ a> k_n\} | \leq | \{a_{k_{n-1}} > 2^{-n}N\}| \leq \frac{1}{(2^{-n}N)^r} \int_{\{a_{k_{n-1}} > 2^{-n}N\}} |a_{k_{n-1}}|^r\]
and thus
\begin{align*}
z_n 
   &\leq \left(\int_0^T\left(\frac{1}{2^{-n}N}\right)^\frac{2r(1+\epsilon)}{2+\epsilon}\left(\int_{\Gamma(t)}|a_{k_{n-1}}|^r\right)^\frac{2(1+\epsilon)}{2+\epsilon}   \right)^\frac{1}{\lambda} = \left(\frac{1}{2^{-n}N}\right)^\frac{2r(1+\epsilon)}{\lambda(2+\epsilon)}\left(\int_0^T\left(\int_{\Gamma(t)}|a_{k_{n-1}}|^r\right)^\frac{2(1+\epsilon)}{2+\epsilon}   \right)^\frac{1}{\lambda}.
\end{align*}
Given our choice of $\lambda$, if we pick $r=\frac{2+\epsilon}{1+\epsilon}\lambda$ then the norm on the right hand side satisfies the hypothesis of Lemma \ref{lem:gn} and therefore
\begin{align*}
z_n &\leq\frac{2^{2n}}{N^2}C_1(T)\norm{a_{k_{n-1}}}{Q(\Gamma)}^2\\
    & \leq \frac{2^{2n}C_1(T)^3}{N^2m^2}\norm{f^+}{L^2_{L^{2+\epsilon}}}^2 z_{n-1}^{\lambda-1}\\
    & = C_2 \beta^{n}z_{n-1}^{1+s},
\end{align*}
where $C_2 := \frac{C_1(T)^3}{N^2m^2}\norm{f^+}{L^2_{L^{2+\epsilon}}}^2 =: \frac{C_3}{N^2},$ $\beta=4$ and $s=\left(\frac 2d -\frac{2}{2+\epsilon} \right)$ and hence we are in the conditions exposed in \eqref{L-U} if $d< 2+\epsilon$ (needed to ensure $s>0$). It remains only to show the bound on $z_0$. It is straighforward to see that
$$z_0\leq T^\frac{1}{\lambda}|\Gamma|^\frac{2(1+\epsilon)}{\lambda(2+\epsilon)} =: C_4$$
and so if we now choose $N^2 \geq 2C_4^s C_3\beta ^{1/s},$ then
$z_0\leq (2C_2)^{-1/s}\beta^{-1/s^2}$
and we conclude that $z_n$ converges to zero as $n$ tends to infinity. This implies that $a(t)\leq 2N$ almost everywhere on $\Gamma(t)$.
\end{proof}
\subsection{The bound on $z$}\label{sec:boundOnZ}
Recalling that the positive part of the right hand side of the $z$ equation is $\delta_{k}^{-1}uw$, we know by \eqref{eq:boundOnn} that $u$ is bounded in $L^\infty_{L^2(\Omega)} \cap L^2_{H^1(\Omega)}$ which, since $H^{1\slash 2}(\Gamma) \cts L^{2+\epsilon}(\Gamma)$ \cite[Theorem 3.81]{Demengel} for all dimensions (where $\epsilon$ depends on the dimension), implies
\[\norm{u}{L^{2+\epsilon}(\Gamma(t))} \leq C_1\norm{u}{H^{1\slash 2}(\Gamma(t))} \leq C_2\norm{u}{H^1(\Omega(t))},\]
so that (the trace of) $u$ is bounded in $L^2_{L^{2+\epsilon}(\Gamma)}$. This gain in the spatial regularity for $u$, combined with the $L^\infty$ bound for $w$, allows us to apply the previous lemma and conclude the boundedness of $z$ since $z$ satisfies an equation of the form \eqref{eq:linearIVP} with $g^+=\delta_k^{-1}uw$.
\subsection{The bound on $u$}\label{sec:uBound}
Boundedness for $u$ is more complicated because we are dealing with two domains ($\Omega$ and its boundary) due to the Robin boundary condition. Let us use the notation $\tilde u = ue^{-\lambda_u t}$ so that $\dot u = e^{\lambda_u t}\md \tilde u + \lambda_u u.$
The equation for $\tilde u$ is
\begin{equation*}\label{eq:foru}
\begin{aligned}
\md{ \tilde u} + \tilde u(\grad \cdot \mathbf V_\Omega  + \lambda_u) - \delta_\Omega \Delta \tilde u + \grad \cdot (\mathbf{J}_{\Omega}u) &=0\\
\delta_\Omega\nabla \tilde u \cdot \nu &= \delta_{k'}^{-1}e^{- \lambda_u t} z -\delta_k^{-1} \tilde u  w +  j\tilde u\\
\tilde u(0) &= u_0.
\end{aligned}
\end{equation*}
Consider the following two problems:
\begin{equation}\label{eq:3}
\begin{aligned}
\dot{a} + a(\grad \cdot \mathbf V_p  + \lambda_u) - \delta_\Omega \Delta a + \grad \cdot (\mathbf{J}_{\Omega}a)&=0\\
\delta_\Omega\nabla a \cdot \nu &= \delta_{k'}^{-1}e^{- \lambda_u t} z -\delta_k^{-1} a w + j(a+b)\\
a(0) &= 0
\end{aligned}
\end{equation}
and
\begin{align*}
\dot{b} + b(\grad \cdot \mathbf V_p  + \lambda_u) - \delta_\Omega \Delta b + \grad \cdot (\mathbf{J}_{\Omega}b)&=0\\
\delta_\Omega \nabla b \cdot \nu &= -\delta_k^{-1} b  w\\
b(0) &= u_0.
\end{align*}
It is clear that $(a+b)$ is a solution of the problem satisfied by $\tilde u$ above. The $L^\infty$ bound $b(t) \leq \norm{u_0}{L^\infty(\Gamma_0)}$ follows after testing the $b$ equation with $(b-M_0)^+$ where $M_0:= \norm{u_0}{L^\infty(\Gamma_0)}$, taking $\lambda_u \geq \norm{\grad \cdot \mathbf V_\Omega}{\infty}$, using \eqref{eq:bulkPosId}, the interpolated trace inequality and Gronwall's lemma. So it remains for us to show that the solution of  \eqref{eq:3}
is bounded given $b \in L^\infty_{L^\infty(\Omega)}$. 

We will again apply a De Giorgi method. Of course there is a lot of similarity to Lemmas \ref{lem:deGiorgi} and \ref{lem:deGiorgiSecond} but we follow the work of Nittka \cite{Nittka} here, which offers several improvements and corrections over the similar material presented in \cite[\S III.7--8]{Lady}.  Let $B_k(t) := \{ x \in \Gamma(t) \mid a(t,x) \geq k\}.$ For $k \geq 0$ let us test with $a_k = (a-k)^+$ and use \eqref{eq:bulkPosId} to get
\begin{align*}
\frac{1}{2}&\int_{\Omega(t)}|a_k(t)|^2  + \int_0^t \int_{\Omega(s)}\delta_\Omega|\nabla a_k|^2 + (\nabla \cdot \mathbf V_\Omega + \lambda_u)aa_k - \frac{1}{2}|a_k|^2\nabla \cdot \mathbf V_\Omega\\
&+ \int_0^t\int_{\Gamma(s)} \frac 12 ja_k^2-(\delta_{k'}^{-1}e^{- \lambda_u t} z-\delta_k^{-1} a  w)a_k- j(a+b)a_k  =0.
\end{align*}
The two terms involving the velocity can be combined if we assume that $\lambda_u \geq \norm{\grad \cdot \mathbf V_\Omega}{\infty}$ (we shall specify $\lambda_u$ later on) and use $a=(a-k) + k$. Now, since $(a+b)a_k = |a_k|^2 + ka_k + ba_k$ and $a_k = \frac{1}{k} k(a-k)^+ \leq \frac{1}{k} (k^2 + |(a-k)^+|^2)$, we have
\begin{align*}
(a+b)a_k 
\leq (2+\norm{b}{L^\infty_{L^\infty(\Omega)}})|a_k|^2 + \frac{\norm{b}{L^\infty_{L^\infty(\Omega)}}}{k}k^2 + k^2
\end{align*}
for $k \geq 1$. Using this, the last term in the weak formulation above can be manipulated as
\begin{align*}
&-\int_0^t\int_{\Gamma(s)}\frac 12 ja_k^2 - (\delta_{k'}^{-1}e^{- \lambda_u t} z-\delta_k^{-1} a w)a_k - j(a+b)a_k\\
&\quad\leq \int_0^t\int_{\Gamma(s)}\delta_{k'}^{-1}\norm{z}{L^\infty_{L^\infty(\Gamma)}}a_k + \norm{j}{\infty}\left(\left(\frac 52+\norm{b}{L^\infty_{L^\infty(\Omega)}}\right)|a_k|^2 + \frac{\norm{b}{L^\infty_{L^\infty(\Omega)}}}{k}k^2 + k^2\right) \\
&\quad\leq \int_0^t\int_{\Gamma(s)}\frac{\delta_{k'}^{-1}\norm{z}{L^\infty_{L^\infty(\Gamma)}}}{k}(k^2+|a_k|^2)+ \norm{j}{\infty}\left(\left(\frac 52+\norm{b}{L^\infty_{L^\infty(\Omega)}}\right)|a_k|^2 + \frac{\norm{b}{L^\infty_{L^\infty(\Omega)}}}{k}k^2 + k^2\right)\\
&\quad=\int_0^t\int_{\Gamma(s)}\frac{\alpha}{k}k^2  + \left(\alpha+\frac 52\norm{j}{\infty}\right)|a_k|^2 + \norm{j}{\infty}k^2
\end{align*}
where we defined $\alpha := \delta_{k'}^{-1}\norm{z}{L^\infty_{L^\infty(\Gamma)}} + \norm{j}{\infty}\norm{b}{L^\infty_{L^\infty(\Omega)}}$.  Now we use the interpolated trace inequality on the (integral of the) second term above with \[\epsilon = \frac{\delta_\Omega}{2(\alpha + \frac 52\norm{j}{\infty})} \quad\text{and}\quad C_\epsilon = \frac{C}{\epsilon} = 2\delta_\Omega^{-1}C\left(\alpha + \frac 52\norm{j}{\infty}\right)\]
to find, taking $k \geq 1$,
\begin{align*}
\frac{1}{2}\int_{\Omega(t)}|a_k(t)|^2 &\leq
-\int_0^t \int_{\Omega(s)}\frac{\delta_\Omega}{2}|\nabla a_k(t)|^2  + (\lambda_u + \frac{\nabla \cdot \mathbf V_\Omega}{2})|a_k(t)|^2 + \left(\frac{\alpha}{k} + \norm{j}{\infty}\right)\int_0^t\int_{B_k(s)}k^2\\
&\quad + 2C\delta_\Omega^{-1}(\alpha + \frac 52\norm{j}{\infty})^2\int_0^t\int_{\Omega(s)}|a_k(t)|^2\\
&\leq -\frac{\delta_\Omega}{2}\int_0^t \int_{\Omega(s)}|\nabla a_k(t)|^2 + \left(\frac{\alpha}{k}+ \norm{j}{\infty}\right)\int_0^t\int_{B_k(s)}k^2 
\end{align*}
where for the last inequality we picked $\lambda_u = \norm{\nabla \cdot \mathbf V_\Omega}{\infty} + 2C\delta_\Omega^{-1}(\alpha + \frac 52\norm{j}{\infty})^2$. This is
\begin{align*}
\frac{1}{2}\int_{\Omega(t)}|a_k(t)|^2 + \frac{\delta_\Omega}{2}\int_0^t \int_{\Omega(s)}|\nabla a_k(t)|^2
&\leq  \frac{\alpha}{k}\int_0^t\int_{B_k(s)}k^2 + \norm{j}{\infty}\int_0^t\int_{B_k(s)}k^2,
\end{align*}
and now take essential supremums over $t$, setting $m=\min(1, \delta_\Omega)$:
\begin{align}
 \frac{1}{2}m\norm{a_k}{Q(\Omega)}^2 
&\leq \left(\frac{\alpha}{k}k^2+ \norm{j}{\infty}k^2\right)\int_0^T\int_{B_k(s)}.\label{eq:prelim2}
\end{align}
By Holder's inequality,
\begin{align}
\left(\frac{\alpha}{k}+\norm{j}{\infty}\right)\int_0^T\int_{B_k(s)}
&\leq \left(\frac{\alpha}{k}+ \norm{j}{\infty}\right) |\Gamma|^{1\slash q}T^{1\slash r}\norm{\chi_{B_k}}{L^{r\slash (r-1)}_{L^{q\slash (q-1)}}}  \leq (1 + C_1\norm{j}{\infty})\norm{\chi_{B_k}}{L^{r\slash (r-1)}_{L^{q\slash (q-1)}}} \label{eq:pre1}
\end{align}
where we have taken $k \geq \alpha|\Gamma|^{1\slash q}T^{1\slash r}$ and set $C_1:=|\Gamma|^{1\slash q}T^{1\slash r}$ and also we constrain $r$ and $q$ so that
\begin{equation}\label{eq:conditionOnrq}
\frac{1}{r} + \frac{d}{2q} < \frac{1}{2}.
\end{equation}
 Now define $\kappa$, $r_*$ and $q_*$ by
\[\frac{1}{r} + \frac{d}{2q} = \frac{1}{2}-\frac{\kappa (d+1)}{2}, \qquad  r_* = \frac{2(1+\kappa)r}{r-1}, \qquad \text{and} \qquad q_* = \frac{2(1+\kappa)q}{q-1}.\]
The condition \eqref{eq:conditionOnrq} implies that $\kappa > 0$. Then \eqref{eq:pre1} can be written like
\begin{align*}
\left(\frac{\alpha}{k}+ \norm{j}{\infty}\right)\int_0^T\int_{B_k(s)} \leq (1 + C_1\norm{j}{\infty})\norm{\chi_{B_k}}{L^{r_*}_{L^{q_*}}}^{2(1+\kappa)}.
\end{align*}
So we find that \eqref{eq:prelim2} becomes
\begin{align*}
\frac{1}{2}m\norm{a_k}{Q(\Omega)}^2 &\leq (1+C_1\norm{j}{\infty})k^2\norm{\chi_{B_k}}{L^{r_*}_{L^{q_*}}}^{2(1+\kappa)} = C_2k^2 \left(\int_0^T |B_k|^{r_*\slash q_*}\right)^{2(1+\kappa)\slash r_{*}}\label{eq:A3}
\end{align*}
where we set $C_2 := 1+C_1\norm{j}{\infty}$.
This, as we said before, holds whenever
\begin{equation}\label{eq:largenessOnk}
k \geq \max\left(1, \alpha C_1\right).
\end{equation}
Now the proof is analogous to what is done in the proof of Lemma \ref{lem:deGiorgiSecond} where instead of the inequality of Lemma \ref{lem:gn} we use Lemma \ref{lem:interpolatedSobolev}. Define $k_n := (2-2^{-n})N$ for a large $N$ and
\[z_n := \left(\int_0^T |B_{k_n}(t)|^{r_*\slash q_*}\right)^{2\slash r_*}.\]
Then we in fact find
\[z_{n+1} \leq 2^6m^{-1}C_2C_I4^nz_n^{1+\kappa}\]
(here $C_I$ is the constant from Lemma \ref{lem:interpolatedSobolev}) and if we take $\hat k$ such that it satisfies \eqref{eq:largenessOnk}, then 
\[(N-\hat k)^2z_0  \leq 2m^{-1}C_2C_I\hat{k}^2 C_0.\]
Defining $C_3:=|\Gamma|^{2(1+\kappa)\slash q_*}T^{2(1+\kappa)\slash r_{*}}$
and picking $N = \hat k(\sqrt{C_3}2^{\frac 12  + \frac 3\kappa + \frac{1}{\kappa^2}}(m^{-1}C_2C_I)^{\frac{1}{2\kappa} + \frac 12} + 1)$, we have
\begin{align*}
z_0 &\leq 
 (2^6m^{-1}C_2C_I)^{-1\slash \kappa}4^{-1\slash \kappa^2}.
\end{align*}
By \cite[II, Lemma 5.6]{Lady}, we have $z_n \to 0$ as $n \to \infty.$ This suggests, since $k_n \to 2N$, that $a(t) \leq 2N$ almost everywhere on $\Gamma(t)$.
 Putting everything together, we find
\[u(t) \leq e^{\norm{\nabla \cdot \mathbf V_\Omega}{\infty}T + C_1\delta_\Omega^{-1} (\alpha + \frac 52\norm{j}{\infty})^2T}\left(\max\left(1, \alpha C_1\right)(C_4m^{-(\frac{1}{2\kappa} + \frac 12)}+ C_5)+\norm{u_0}{L^\infty(\Omega_0)}\right).\]
\section{Strong solutions for $u$}\label{sec:strongSolutionsForu}
In order to show that $\dot u$ is in fact a function and not just an element of a dual space, we would like to test with $\dot u$ but obviously this is not possible. We may try a smoothing technique like using the Galerkin approximation for the $u$ equation and testing with the finite-dimensional time derivative $\dot u_n$ but this does not help either since we would actually need a uniform bound in $L^\infty$ on the $u_n$ which does not necessarily follow from the $L^\infty$ bound on $u$. 
Instead we mollify in time the $u$ equation and try to make use of the $L^\infty$ bound obtained for $u$ in Theorem \ref{thm:LinftyBounds}. Essentially we will test the $u$ equation with an approximation of $\dot u$ by Steklov averaging \cite{DiBenedetto, DGV, Lady}. Actually, we will do this for the equation pulled back onto a reference domain. Before we proceed, let us discuss the Steklov averaging technique and some properties. For a function $v \in L^p(0,T;X)$ where $X$ is a Banach space, we define its Steklov average $v_h$ by
\[v_h(t) = \begin{cases}
\frac 1h \int_t^{t+h} v(s)\;\mathrm{d}s &: 0 < t \leq T-h\\
0 &: T-h < t \leq T
\end{cases}\]
for $0< h < T$. For $p \neq \infty$, it is well known that for any $\epsilon \in (0,T)$, we have $v_h \in L^p(0,T-\epsilon;X)$ with $\norm{v_h}{L^p(0,T-\epsilon;X)} \leq \norm{v}{L^p(0,T;X)}$ for all $h \in (0,\epsilon)$, and $v_h \to v$ in $L^p(0,T-\epsilon;X)$ as $h \to 0$. Also, $v_h$ has the (strong) a.e. derivative
\[\partial_t v_h(t) = \frac{v(t+h)-v(t)}{h} =: D_hv(t).\] 
Now, to obtain the weak form corresponding to the equation on the reference domain, we introduce some notation. We have defined in Section \ref{sec:weakFormulation} the flow $\Phi$ such that 
$$\Phi_t:\overline{\Omega_0}\rightarrow\overline{\Omega(t)} \,\text{ with }\, \frac{d}{dt}\Phi_t(\cdot)=\mathbf{V}_p(t,\Phi_t(\cdot)).$$
Let $\mathbf D\Phi_t$ be the Jacobian matrix of $\Phi_t$ with inverse $\mathbf{M}(t) := (\mathbf D\Phi_t)^{-1}$ and let $J_t=\det{\mathbf D\Phi_t}$ be its determinant. Define $\mathbf{A}(t) := \delta_\Omega J_t\mathbf{M}(t)\mathbf{M}(t)^T$ and $\omega_t := |\mathbf{M}(t)^T \nu_0|$, where $\nu_0$ is the outward  normal to $\Gamma_0$.

Below, $\tilde{g}\colon \Omega_0\times[0,T]\rightarrow \R$ denotes the pullback of a function $g \in L^2_X$ by $\tilde{g}(\xi,t)=g(\Phi_t(\xi),t)$; we use the same notation for the pullback of functions defined on $\Gamma(t)$. 
 Pulling back the weak form of the $u$ equation gives (see \cite[Chapter 9, \S 4.2]{Delfour} for the boundary integral)
\[\langle \tilde u_t, J_t\tilde \varphi \rangle + \int_{\Omega_0}\mathbf{A}(t)\grad \tilde u \grad \tilde \varphi +  \tilde u \tilde \varphi V_0J_t + 
J_t\widetilde{\mathbf{J}_\Omega} \cdot \mathbf{M}(t)^T \grad \tilde u \tilde \varphi = \int_{\Gamma_0}\left( \frac{\tilde z}{\delta_{k'}} - \frac{\tilde u \tilde w}{\delta_k}\right)\tilde \varphi J_t \omega_t + j_0\tilde{u}\tilde{\varphi}J_t\omega_t,\]
where $V_0=\widetilde{\nabla\cdot\mathbf{V}_\Omega}$ is the pullback of $\nabla\cdot\mathbf{V}_\Omega$, and similarly for $j_0=\widetilde{(\mathbf{V}_\Omega-\mathbf{V}_\Gamma)}\cdot\tilde{\nu}$.
Now setting $\psi = J_t\tilde \varphi$, this becomes
\[\langle \tilde u_t, \psi \rangle + \int_{\Omega_0}\mathbf{A}(t)\grad \tilde u \grad (J_t^{-1}\psi) + \tilde u  \psi V_0 + 
\widetilde{\mathbf{J}_\Omega} \cdot \mathbf{M}(t)^T \grad \tilde u \psi = \int_{\Gamma_0}\left( \frac{\tilde z}{\delta_{k'}} - \frac{\tilde u \tilde w}{\delta_k}\right)\psi \omega_t+  j_0\tilde{u}\psi\omega_t.\]
Since $\mathbf{A}(t)\grad \tilde u \grad (J_t^{-1}\psi)  = J_t^{-1}\mathbf{A}(t)\grad \tilde u \grad \psi + \psi \mathbf{A}(t)\grad \tilde u \grad J_t^{-1},$ setting $\mathbf{B}(t):=J_t^{-1}\mathbf{A}(t)$, 
 integrating by parts in time and relabelling to remove all the tildes from $\tilde u$, we obtain
\begin{equation}\label{eq:pulledBackWeakForm}
 \frac{d}{dt}\int_{\Omega_0}u\psi  + \int_{\Omega_0}\mathbf{B}(t)\grad u \grad \psi + \psi \mathbf{A}(t)\grad  u \grad J_t^{-1}   +  u \psi V_0 + 
\widetilde{\mathbf{J}_\Omega} \cdot \mathbf{M}(t)^T \grad u \psi = \int_{\Gamma_0}\left( \frac{ z}{\delta_{k'}} - \frac{uw}{\delta_k}\right)\psi \omega_t+ j_0u\psi\omega_t
\end{equation}
if $\psi$ is independent of time. To write the weak form associated to the function $u_h$ (see \cite[Chapter II]{DiBenedetto}), in \eqref{eq:pulledBackWeakForm}, divide by $h$ and integrate over $(t,t+h)$:
\begin{align}
\nonumber \int_{\Omega_0}\partial_t u_h(t)\psi +  \frac 1h\int_{\Omega_0}\int_t^{t+h}\mathbf{B}(s)\grad u \grad \psi &+ \frac 1h\int_{\Omega_0}\int_t^{t+h}\psi \mathbf{A}(s)\grad  u \grad J_s^{-1}  + \frac 1h\int_{\Omega_0}  \int_t^{t+h}u \psi V_0\\
&+ \frac 1h \int_t^{t+h}\int_{\Omega_0}
\widetilde{\mathbf{J}_\Omega} \cdot \mathbf{M}(s)^T \grad u \psi = \frac 1h\int_{\Gamma_0}\int_t^{t+h}\left( \frac{z}{\delta_{k'}} - \frac{uw}{\delta_k}+j_0u\right)\psi \omega_s.\label{eq:pbWeakForm}
\end{align}
The idea is to test with $\partial_t u_h(t)$ and integrate over $t$ and try to find a bound on $u_h'$ in $L^2(0,T;L^2(\Omega_0))$ independent of $h$. This requires us to handle the various terms in the equality above, which is not at all straightforward since the coefficients are time-dependent. For example, in the non-moving setting where the elliptic operator such as the Laplacian is independent of time, we have
\[\left(\langle -Au, v \rangle\right)_h = \langle -Au_h, v \rangle,\]
i.e., the Steklov average commutes with elliptic operator. In our case of a moving domain this equality is no longer true (in fact some extra terms appear) because the coefficients of the operator $A$ depend on time. For this purpose, we need the following auxiliary results.
\subsection{Preliminary results}
We begin with the following fundamental lemma, which follows by a simple integration by parts argument. 
\begin{lem}\label{lem:basicIBP}
Let $\varphi\colon [0,T] \to \mathbb{R}$ be absolutely continuous and $f \in L^1(0,T)$. We have
\begin{align*}
\frac 1h \int_t^{t+h} \varphi(s)f(s) 
&= \varphi(t+h) f_h(t) + D_h\varphi(t)\int_0^t f(s) - \frac 1h\int_t^{t+h}\varphi'(s)\int_0^s f(r).
\end{align*}
\end{lem}
Given functions $g$, $\varphi$ and $k$ defined on $[0,T]\times \Omega_0$, it is convenient to define a map $\mathcal{L}_h$ by the expression
\[\mathcal L_h(g, \varphi, k)(t):= g(t)\left(D_h\varphi(t)K(t) - \frac 1h \int_t^{t+h}\varphi'(s)K(s)\right)\qquad\text{where $K(t) = \int_0^t k(s)$}.\]
Thus the equality of Lemma \ref{lem:basicIBP} can be rewritten as
\begin{align*}
\frac 1h \int_t^{t+h} \varphi(s)f(s) 
&= \varphi(t+h) f_h(t) + \mathcal{L}_h(1, \varphi, f)(t).
\end{align*}
A form of Lemma \ref{lem:basicIBP} for matrix-vector products is given by the following corollary, which can be proved by writing the products componentwise and using the formula for $\mathcal{L}_h$.
\begin{cor}\label{coro:matrix}
Let $\psi \in H^1(\Omega_0)$, $f\in L^2(0,T;H^1(\Omega_0))$ and let $\mathbf{N}(t)=[n_{ij}(t)]$ be a matrix with $n_{ij}\colon [0,T] \to \Omega$ absolutely continuous and $n_{ij}' \in L^1(0,T;L^\infty(\Omega))$. Then
\begin{equation*}  
\frac 1h\int_{\Omega_0}\int_t^{t+h} \mathbf{N}(s)\grad f(s) \grad \psi = \int_{\Omega_0} \mathbf{N}(t+h)\grad f_h(t) \grad \psi + \sum_{ij}\int_{\Omega_0}\mathcal{L}_h(\psi_{x_j}, n_{ij}, f_{x_i})(t).
\end{equation*}
\end{cor}
Since $\mathcal{L}_h$ appears a number of times, it is useful to bound it in terms of its arguments. First, define the space
\[X := \{ \varphi \in \operatorname{Lip}(0,T;L^\infty(\Omega_0)) \mid \varphi \text{ is differentiable a.e. $t \in [0,T]$}\}.\]
This is indeed a proper subset (that is, the constraint is not redundant) because $L^\infty(\Omega_0)$ does not have the {Radon--Nikodym} property \cite[Example 1.2.8]{Arendt}.
\begin{lem}\label{lem:boundednessOfL}
The map $\mathcal L_h\colon L^2(0,T;L^2(\Omega_0))\times X \times L^2(0,T;L^2(\Omega_0)) \to L^1(0,T;L^1(\Omega_0))$ and for $g$, $k \in L^2(0,T;L^2(\Omega_0))$ and $\varphi \in X$,
\begin{align*}
&\int_0^T \int_{\Omega_0} \mathcal L_h(g, \varphi, k)(t)
\leq 4\operatorname{Lip(\varphi)}\sqrt{T-h}\norm{g}{L^2(0,T-h;L^2(\Omega_0))} \norm{k}{L^2(0,T;L^2(\Omega_0))}.
\end{align*}
\end{lem}
\begin{proof}
By adding and subtracting the same term in the definition of $\mathcal{L}_h$,
\begin{align*} 
\mathcal L_h(g,\varphi,k)(t) 
  &= g(t)K(t) \left(D_h\varphi(t)- \varphi'(t)\right) + g(t)\left(\varphi'(t)K(t) - \frac 1h \int_t^{t+h}\varphi'(s)K(s)\right).
\end{align*}
The assumptions on $\varphi$ mean that it is Lipschitz continuous in time with a global Lipschitz constant, and therefore $D_h\varphi(t) \leq \text{Lip}(\varphi).$ After integrating the above expression for $\mathcal{L}_h$ in space and over $t \in (0,T-h)$, the first term on the right hand side is
\begin{align*} 
\int_0^{T-h}\int_{\Omega_0}g(t)K(t) \left(D_h\varphi(t)- \varphi'(t)\right)
 &\leq 2\operatorname{Lip}(\varphi)\norm{g}{L^2(0,T-h;L^2(\Omega_0))}\norm{K}{L^2(0,T-h;L^2(\Omega_0))}
\end{align*}
and the second term can be dealt with as follows:
\begin{align*}
\int_0^{T-h}\int_{\Omega_0}g(t)\left(\varphi'(t)K(t) - \frac 1h \int_t^{t+h}\varphi'(s)K(s)\right) &\leq \int_0^{T-h}\int_{\Omega_0}|g(t)|\norm{\varphi'}{\infty}\left(|K(t)|   + \frac 1h \int_t^{t+h}|K(s)|\right)\\
&= \operatorname{Lip(\varphi)}\int_0^{T-h}\int_{\Gamma_0}|g(t)||K(t)| + |g(t)|(|K(t)|)_h\\
&\leq 2\operatorname{Lip(\varphi)}\norm{g}{L^2(0,T-h;L^2(\Omega_0))}\norm{K}{L^2(0,T;L^2(\Omega_0))}.
\end{align*}
The claim follows once we note that
\[\norm{K}{L^2(0,T-h;L^2(\Omega_0))}^2 = \int_0^{T-h}\int_{\Omega_0} \left|\int_0^ t k(s)\right|^2 \leq \int_0^{T-h}\int_{\Omega_0} \int_0^ t t|k(s)|^2 \leq (T-h)\norm{k}{L^2(0,T-h;L^2(\Omega_0))}^2.\]
\end{proof}

\subsection{Obtaining the bound}

As we mentioned before, we want to establish a bound on $u_h'$ in $L^2(0,T;L^2(\Omega_0))$ and to do so we will integrate \eqref{eq:pbWeakForm} in time for a specific test function $\psi$. Let $\xi \in C^\infty_c((0,T])$ be a smooth function vanishing near $t=0$ and equal to one near $t=T$ with $0 \leq \xi \leq 1$ and pick $\psi = \xi(t)\partial_t u_h(t)$ in \eqref{eq:pbWeakForm} (this cutoff function is necessary to deal with the Laplacian term: we cannot bound $\nabla u_h(0)$ in $L^2(\Omega_0)$ independent of $h$). Our aim is to prove the following lemma, which we will do by addressing each term in the next subsections.
\begin{lem}\label{lem:ssInequalities}
With $\psi = \xi(t)\partial_t u_h(t)$, the following lower bound
\begin{align}
 \frac 1h\int_0^{T-h}\int_{\Omega_0}\int_t^{t+h} \mathbf{B}(s)\grad u \grad \psi &\geq (\lambda_T-(d+1)\rho)\norm{\grad u_h(T-h)}{L^2(\Omega_0)}^2 - C_\rho\label{eq:step1}
 \end{align}
 and the following upper bounds
 \begin{align}
\frac 1h\int_0^{T-h} \int_{\Omega_0}\int_t^{t+h}\psi \mathbf{A}(s)\grad  u \grad J_s^{-1} +\widetilde{\mathbf{J}_\Omega} \cdot \mathbf{M}(s)^T \grad u \psi + u \psi  V_0 &\leq \epsilon\lVert{\sqrt{\xi}u_h'}\rVert_{L^2(0,T-h;L^2(\Omega_0))}^2 + C_\epsilon\label{eq:step2}\\
\frac 1h\int_0^{T-h}\int_{\Gamma_0}\int_t^{t+h}( z - u w+ j_0u)\psi \omega_s &\leq C\label{eq:step3}
\end{align}
hold, where $\lambda_T >0$ is as in \eqref{eq:ss1}, $\rho$, $\epsilon > 0$ can be chosen arbitrarily and all constants on the right hand sides are independent of $h$.
\end{lem}

\subsubsection{The Laplacian term}
Clearly, the most troublesome term in \eqref{eq:pbWeakForm} is the one involving the gradient of the test function $\psi = \xi \partial_t u_h$. We must manipulate it in such a way as to extract a positive contribution of $\int_{\Omega_0}|\nabla u_h(T-h)|^2$ (a term that we cannot bound from above since it would require pointwise control on the gradient of $u$). This is trivial when $B$ is independent of time but in our case the extra terms arising from the time dependency generate unwanted \emph{negative} contributions of (the square root of) the above integral, which we will overcome by Young's inequality. Let us begin by using Corollary \ref{coro:matrix} to write
\begin{align*}
\nonumber \frac 1h\int_{\Omega_0}\int_t^{t+h} \mathbf{B}(s)\grad u \grad \psi &= 
\int_{\Omega_0} \xi(t) \mathbf{B}(t+h)\grad u_h \grad \partial_t u_h(t) \\
\nonumber &\quad + \sum_{ij}\int_{\Omega_0}\xi(t)(\partial_t u_h(t))_{x_j}\left(D_hb_{ij}(t)\int_0^t u_{x_i}(s) - \frac 1h\int_t^{t+h}b_{ij}'(s)\int_0^s u_{x_i}(r)\right)\\
\nonumber &= \frac{1}{2}\frac{d}{dt}\int_{\Omega_0}\xi(t) \mathbf{B}(t+h)\grad u_h(t) \grad u_h(t)\\
\nonumber &\quad - \frac 12 \int_{\Omega_0}(\xi'(t)\mathbf{B}(t+h)+\xi(t)B'(t+h))\grad u_h(t) \grad u_h(t)\\
&\quad + \sum_{ij}\underbrace{\int_{\Omega_0}\xi(t)(\partial_t u_h(t))_{x_j}\left(D_hb_{ij}(t)\int_0^t u_{x_i}(s) - \frac 1h\int_t^{t+h}b_{ij}'(s)\int_0^s u_{x_i}(r)\right)}_{=:I_{ij}}
\end{align*}
Above, we wrote the $\mathcal{L}_h$ term out explicitly as Lemma \ref{lem:boundednessOfL} does not produce a useful estimate in this instance and it requires some preparation. Before that, consider the first term on the right hand side. We have, since $\xi$ vanishes near zero and equals one near $T$ and $\mathbf{B}(T)$ is a positive-definite matrix,
\begin{align}
    \nonumber \frac{1}{2}\int_0^{T-h}\frac{d}{dt}\int_{\Omega_0}\xi(t) \mathbf{B}(t+h)\grad u_h(t) \grad u_h(t) & = \frac{1}{2}\int_{\Omega_0}\xi(T-h) \mathbf{B}(T)\grad u_h(T-h) \grad u_h(T-h)\\
    & \geq \lambda_T \int_{\Omega_0}|\nabla u_h(T-h)|^2,\label{eq:ss1}
\end{align}
where the constant $\lambda_T$ is independent of $x$. This term, which is a positive contribution of the gradient evaluated at a point, will be used to absorb negative terms that arise below. Regarding the second term on the right hand side, it is straighforward to bound
\begin{equation}\label{eq:ss2}
\frac 12 \int_0^{T-h}\int_{\Omega_0}(\xi'(t)\mathbf{B}(t+h)+\xi(t)B'(t+h))\grad u_h(t) \grad u_h(t)\leq C_1\norm{\nabla u_h}{L^2(0,T;L^2(\Omega_0))}^2.
\end{equation}
This leaves us with the integral $I_{ij}$ which in its current form is not helpful since we cannot bound the term $(\partial_t u_h(t))_{x_j}$. However, we can use the convenient properties of the Steklov averaging allowing us to swap the order of spatial and temporal derivatives:
\[(\partial_t u_h(t))_{x_j} = \left(\frac{u(t+h)-u(t)}{h}\right)_{x_j} = \frac{u_{x_j}(t+h)-u_{x_j}(t)}{h} = \partial_t ((u_{x_j})_h(t)),\]
thus we may integrate by parts in time to yield
\begin{align}
\nonumber I_{ij} &= \frac{d}{dt} \int_{\Omega_0}\xi(t)(u_{x_j})_h(t)\left(D_hb_{ij}(t)\int_0^t u_{x_i}(s) - \frac 1h\int_t^{t+h}b_{ij}'(s)\int_0^s u_{x_i}(r)\right)\\
\nonumber &- \int_{\Omega_0}\xi(t)(u_{x_j})_h(t)\frac{d}{dt}\left(D_hb_{ij}(t)\int_0^t u_{x_i}(s) - \frac 1h\int_t^{t+h}b_{ij}'(s)\int_0^s u_{x_i}(r)\right)\\
&- \int_{\Omega_0}\xi'(t)(u_{x_j})_h(t)\left(D_hb_{ij}(t)\int_0^t u_{x_i}(s) - \frac 1h\int_t^{t+h}b_{ij}'(s)\int_0^s u_{x_i}(r)\right)\label{eq:bigline}.
\end{align}
Expanding the brackets above leaves us with six terms to control. We do this in a few steps.

\noindent\textsc{Step A} (\emph{the first, third and fifth term}). After expanding the brackets, we see that the sum of the first term and the third term is
\begin{align*}
&\frac{d}{dt} \int_{\Omega_0}\xi(t)(u_{x_j})_h(t)D_hb_{ij}(t)\int_0^t u_{x_i}(s)   - \int_{\Omega_0}\xi(t)(u_{x_j})_h(t)\frac{d}{dt}\left(D_hb_{ij}(t)\int_0^t u_{x_i}(s)\right)\\
&=\frac{d}{dt} \int_{\Omega_0}\xi(t)(u_{x_j})_h(t)D_hb_{ij}(t)\int_0^t u_{x_i}(s)   - \int_{\Omega_0}\xi(t)(u_{x_j})_h(t)\left(D_hb_{ij}'(t)\int_0^t u_{x_i}(s) + D_hb_{ij}(t)u_{x_i}(t)\right).
\end{align*}
Let us integrate in time and handle each of these terms on the right hand side.

\noindent \textsc{Step A.1.} Since $\xi$ vanishes near zero, we have
\begin{align*}
\int_0^{T-h}\frac{d}{dt} \int_{\Omega_0}\xi(t)(u_{x_j})_h(t)D_hb_{ij}(t)\int_0^t u_{x_i}(s)
&=  \int_{\Omega_0}\xi(T-h)(u_{x_j})_h(T-h)D_hb_{ij}(T-h)\int_0^{T-h} u_{x_i}(s)\\
&\leq  \norm{b_{ij}}{\text{Lip}}\int_{\Omega_0}|(u_{x_j})_h(T-h)|\int_0^{T-h} |u_{x_i}(s)|\\
&\leq C_1\norm{(u_{x_j})_h(T-h)}{L^2(\Omega_0)}\int_0^{T-h}\norm{u_{x_i}(s)}{L^2(\Omega_0)}\\
&= C_1\norm{(u_{x_j})_h(T-h)}{L^2(\Omega_0)}\norm{u_{x_i}}{L^1(0,T-h;L^2(\Omega_0))}\\
&\leq \frac{\rho}{2}\norm{(u_{h})_{x_j}(T-h)}{L^2(\Omega_0)}^2 + C_\rho\norm{u_{x_i}}{L^1(0,T;L^2(\Omega_0))}^2
\end{align*}
since derivatives and the Steklov averaging commute.

\noindent \textsc{Step A.2.} We also have
\begin{align*}
\int_0^{T-h}\int_{\Omega_0}\xi(t)(u_{x_j})_h(t)D_hb_{ij}'(t)\int_0^t u_{x_i}(s) &\leq \norm{b_{ij}'}{\text{Lip}}\int_0^{T-h}\int_{\Omega_0}|(u_{x_j})_h(t)|\int_0^t |u_{x_i}(s)|\\
&\leq C_1\int_0^{T-h}\int_0^t \int_{\Omega_0}|(u_{x_j})_h(t)||u_{x_i}(s)|\\
&\leq C_2\norm{u_{x_i}}{L^1(0,T;L^2(\Omega_0))}\int_0^{T-h}\norm{(u_{x_j})_h(t)}{L^2(\Omega_0)}\\
&\leq C_3\norm{u_{x_i}}{L^1(0,T;L^2(\Omega_0))}\norm{u_{x_j}}{L^1(0,T-h;L^2(\Omega_0))}.
\end{align*}
The fifth term in \eqref{eq:bigline} can be dealt with in exactly the same manner.

\noindent \textsc{Step A.3.} Finally, we have
\begin{align*}
\int_0^{T-h}\int_{\Omega_0}\xi(t)(u_{x_j})_h(t)D_hb_{ij}(t)u_{x_i}(t) &\leq \norm{b_{ij}(t)}{\text{Lip}}\int_0^{T-h}\int_{\Omega_0}|(u_{x_j})_h(t)||u_{x_i}(t)|\\
&\leq C_1\norm{u_{x_j}}{L^2(0,T-h;L^2(\Omega_0))}\norm{u_{x_i}}{L^2(0,T;L^2(\Omega_0))}.
\end{align*}
\textsc{Step B} (\emph{The second and fourth terms}). Now let us handle (the negative of) the second and fourth terms in \eqref{eq:bigline}.

\noindent\textsc{Step B.1.} We have
\begin{align*}
\int_0^{T-h}\frac{d}{dt} \int_{\Omega_0}\xi(t)(u_{x_j})_h(t)\left(\frac 1h\int_t^{t+h}b_{ij}'(s)\int_0^s u_{x_i}(r)\right)&= \int_{\Omega_0}\xi(T-h)(u_{x_j})_h(T-h)\left(\frac 1h\int_{T-h}^{T}b_{ij}'(s)\int_0^s u_{x_i}(r)\right)\\ 
&\leq \frac{\norm{b'_{ij}}{\infty}}{h}\int_{T-h}^T\int_0^s\int_{\Omega_0}|(u_{x_j})_h(T-h)||u_{x_i}(r)|\\
&\leq \frac{C_1}{h} \norm{(u_{x_j})_h(T-h)}{L^2(\Omega_0)}\int_{T-h}^T \norm{u_{x_i}}{L^1(0,T;L^2(\Omega_0))}\\
&= C_1 \norm{(u_{x_j})_h(T-h)}{L^2(\Omega_0)} \norm{u_{x_i}}{L^1(0,T;L^2(\Omega_0))}\\
&\leq \frac{\rho}{2}\norm{(u_{x_j})_h(T-h)}{L^2(\Omega_0)}^2  + C_\rho
\end{align*}
\textsc{Step B.2.} Since
\begin{align*}
\frac{d}{dt}\left(\frac 1h\int_t^{t+h}b_{ij}'(s)\int_0^s u_{x_i}(r)\right) &=\frac 1h\left(b_{ij}'(t+h)\int_0^{t+h} u_{x_i} - b_{ij}'(t)\int_0^t u_{x_i}(r)\right),
\end{align*}
we find the fourth term to be
\begin{align*}
&\int_0^{T-h}\int_{\Omega_0}\xi(t)(u_{x_j})_h(t)\frac 1h\left(b_{ij}'(t+h)\int_0^{t+h} u_{x_i} - b_{ij}'(t)\int_0^t u_{x_i}(r)\right)\\
&= \frac 1h\int_0^{T-h}\int_{\Omega_0}\xi(t)(u_{x_j})_h(t)\left(b_{ij}'(t+h)\left[\int_0^{t+h} u_{x_i}-\int_0^t u_{x_i}\right] + \left[b_{ij}'(t+h)- b_{ij}'(t)\right]\int_0^t u_{x_i}(r)\right)\\
&= \int_0^{T-h}\int_{\Omega_0}\xi(t)(u_{x_j})_h(t)b_{ij}'(t+h)\frac 1h\int_t^{t+h} u_{x_i}(r) + \frac 1h\int_0^{T-h}\int_0^t\int_{\Omega_0}\xi(t)(u_{x_j})_h(t)\left[b_{ij}'(t+h)- b_{ij}'(t)\right] u_{x_i}(r)\\
&\leq \norm{b_{ij}'}{\infty}\int_0^{T-h}\int_{\Omega_0}|(u_{x_j})_h(t)||(u_{x_i})_h(t)| + \norm{b'_{ij}}{\text{Lip}}\int_0^{T-h}\int_0^t\int_{\Omega_0}|(u_{x_j})_h(t)| |u_{x_i}(r)|\\
&\leq C_1\norm{(u_{x_j})_h}{L^2(0,T-h;L^2(\Omega_0))}\norm{(u_{x_i})_h}{L^2(0,T-h;L^2(\Omega_0))} + C_2\norm{(u_{x_j})_h}{L^1(0,T-h;L^2(\Omega_0))}\norm{u_{x_i}}{L^1(0,T;L^2(\Omega_0))}
\end{align*}
\textsc{Step C} (\emph{The final term}). It remains for us to bound the last term in \eqref{eq:bigline}:
\begin{align*}
\int_0^{T-h}\int_{\Omega_0}\xi'(t)(u_{x_j})_h(t)\left(\frac 1h\int_t^{t+h}b_{ij}'(s)\int_0^s u_{x_i}(r)\right) &= \int_0^{T-h}\int_{\Omega_0}\xi'(t)(u_{x_j})_h(t)[b_{ij}'U_{x_i}]_h(t)\\
&\leq \norm{\xi'}{\infty}\norm{(u_{x_j})_h}{L^2(0,T;L^2(\Omega_0))}\norm{[b_{ij}'U_{x_i}]_h}{L^2(0,T;L^2(\Omega_0))}\\
&\leq \norm{\xi'}{\infty}\norm{b_{ij}'}{\infty}\norm{u_{x_j}}{L^2(0,T;L^2(\Omega_0))}\norm{u_{x_i}}{L^2(0,T;L^2(\Omega_0))}
\end{align*}
where we defined $U_{x_i}(t) = \int_0^t u_{x_i}$.

\noindent \textsc{Conclusion.} Combining all of the steps and summing \eqref{eq:bigline} over $i$ and $j$, we find 
\begin{align*}
\sum_{ij}\int_0^{T-h}I_{ij}
\leq (d+1)\rho\norm{\grad u_h(T-h)}{L^2(\Omega_0)}^2 + C_1
\end{align*}
for a constant $C_1$ not depending on $h$. 
Taking this into account with \eqref{eq:ss1} and \eqref{eq:ss2}, we end up with \eqref{eq:step1}.
\subsubsection{The remaining terms on the left hand side}
The remaining integrals over $\Omega_0$ are lower order in $\partial_t u_h$ so we can be a bit more crude in how we deal with them but since we already have the more precise machinery of Lemmas \ref{lem:basicIBP} and \ref{lem:boundednessOfL} let us take more care.
Denoting the components of $\widetilde{\mathbf{J}_\Omega}$ as $k_i$, using
\[\mathbf{A}(s)\grad  u \grad J_s^{-1} +
\widetilde{\mathbf{J}_\Omega} \cdot \mathbf{M}(s)^T \grad u  =\sum_{ij}a_{ij}(s)j_{x_j}u_{x_i} +  \sum_{ij} k_jm_{ij}u_{x_i} =\sum_{ij}\underbrace{(a_{ij}(s)j_{x_j}+k_jm_{ij})}_{=:n_{ij}}u_{x_i}\]
 we see that, using Lemma \ref{lem:basicIBP}, and the definition of $\psi$,
\begin{align*}
&\frac 1h \int_{\Omega_0}\int_t^{t+h}\psi \mathbf{A}(s)\grad  u(s) \grad J_s^{-1} + 
\widetilde{\mathbf{J}_\Omega} \cdot \mathbf{M}(s)^T \grad u(s) \psi\\ 
&= \int_{\Omega_0}\xi(t)\partial_t u_h(t) \mathbf{A}(t+h)\grad u_h(t) \grad J_{t+h}^{-1} + \sum_{ij}\mathcal L_h(t;\xi \partial_t u_h,n_{ij}, u_{xi})\\
&\leq \int_{\Omega_0}\frac{\epsilon}{4}\xi(t)|\partial_tu_h(t)|^2 + C_\epsilon |\mathbf{A}(t+h)\grad u_h(t) \grad J_{t+h}^{-1}|^2 + \sum_{ij}\mathcal L_h(t;\xi \partial_t u_h,n_{ij}, u_{xi})
\end{align*}
and
\begin{align*}
\frac 1h\int_{\Omega_0}  \int_t^{t+h}u \psi V_0 &= 
 \int_{\Omega_0}\xi(t)\partial_t u_h(t)\left(V_0(t+h) u_h(t) + D_hV_0(t)\int_0^t u(s) - \frac 1h\int_t^{t+h}V_0'(s)\int_0^s u(r)\right)\\
&\leq \int_{\Omega_0}\frac{\epsilon}{4}\xi(t)|\partial_t u_h(t)|^2 + C_\epsilon|V_0(t+h) u_h(t)|^2 + \mathcal L_h(t;\xi \partial_t u_h, V_0, u).
\end{align*}
Integrating both of these and combining, we find with the aid of Lemma \ref{lem:boundednessOfL},
\begin{align*}
\frac 1h\int_0^{T-h}  \int_{\Omega_0}&\int_t^{t+h}\psi \mathbf{A}(s)\grad  u(s) \grad J_s^{-1} + 
\widetilde{\mathbf{J}_\Omega} \cdot \mathbf{M}(s)^T \grad u(s) \psi +  u(s)\psi  V_0\nonumber \\
&\leq  \frac{\epsilon}{2}\norm{\sqrt{\xi}\partial_t u_h}{L^2(0,T-h;L^2(\Omega_0))}^2 + C_\epsilon \int_0^{T-h}\int_{\Omega_0}|\grad u_h(t)|^2|\mathbf{A}(t+h)^T\grad J_t^{-1}|^2 \\ & +C\norm{\sqrt{\xi}\partial_t u_h}{L^2(0,T-h;L^2(\Omega_0))}\norm{u_{x_i}}{L^2(0,T;L^2(\Omega_0))}\nonumber \nonumber 
C + \norm{\sqrt{\xi}\partial_t u_h}{L^2(0,T-h;L^2(\Omega_0))}\norm{u}{L^2(0,T;L^2(\Omega_0))}
\end{align*}
and this leads to \eqref{eq:step2} after using the boundedness of $|\mathbf{A}(t+h)^T\grad J_t|$ and Young's inequality with $\epsilon$.

\subsubsection{Dealing with the boundary quantities}
We have an issue with $\partial_t u_h$ on the boundary, since the trace theorem would insist on gradient control on $u_t$. The workaround is to integrate by parts and to use a trick involving the chain rule to deal with the $u\partial_t u_h$ term.

First, to ease the notation, let $\hat z = \delta_{k'}^{-1}\xi\omega z$, $\hat w = \delta_k^{-1}\xi \omega w$, $\hat{j_0}=\xi\omega j_0$ and define $f=\hat{w}-\hat{j_0}$. The right hand side of \eqref{eq:pbWeakForm} is then
\begin{align}
\nonumber \frac 1h\int_{\Gamma_0}\int_t^{t+h}\left( \frac{z}{\delta_{k'}} - \frac{u w}{\delta_k}+ j_0u\right)\psi \omega_s &=\int_{\Gamma_0}\partial_t u_h(t)\frac 1h\int_t^{t+h}\hat z - u(\hat w-\hat{j_0}) \\
\nonumber &=\int_{\Gamma_0}\partial_t u_h(t)[\hat z_h - (u f)_h] \\
&=\frac{d}{dt}\int_{\Gamma_0} \hat{z}_h u_h - \int_{\Gamma_0} \partial_t \hat{z}_h u_h - \frac{d}{dt}\int_{\Gamma_0} [u f]_h u_h + \int_{\Gamma_0} \partial_t [uf]_h u_h  \label{eq:ste1}.
\end{align}
The final term above still contains explicitly a derivative of $u$, so let us rewrite it by first using
\[\partial_t [uf]_h = \frac{u(t+h)-u(t)}{h}f(t+h) + u(t)\frac{f(t+h)-f(t)}{h} = f(t+h)\partial_t u_h  + u\partial_t f_h\]
so that
\begin{align*}
\int_{\Gamma_0} \partial_t [uf]_h u_h 
&= \frac{1}{2}\int_{\Gamma_0} f(t+h)\partial_t (u_h^2) + \int_{\Gamma_0}uu_h\partial_t f_h\\
&= \frac 12\frac{d}{dt}\int_{\Gamma_0} f(t+h)u_h^2 - \frac 12 \int_{\Gamma_0} f'(t+h)u_h^2 + \int_{\Gamma_0} uu_h\partial_t f_h.
\end{align*}
(On $\Gamma_0$ by $\partial_t u_h$ we mean $\partial_t [(u|_{\Gamma_0})_h]$ where of course $(u|_{\Gamma_0})_h(t) = \frac 1h \int_t^{t+h}u|_{\Gamma_0}$.) 
Plugging this into \eqref{eq:ste1} above and integrating over $t \in (0,T-h)$, we find the right hand side of \eqref{eq:ste1} to be
\begin{align*}
&\int_{\Gamma_0} (\hat{z}_h(T-h) u_h(T-h)- \hat{z}_h(0)u_h(0)) - \int_0^{T-h}\int_{\Gamma_0} \partial_t \hat{z}_h u_h - \int_{\Gamma_0} ([uf]_h(T-h) u_h(T-h) - [uf]_h(0)u_h(0))\\
&+ \frac 12\int_{\Gamma_0} (f(T)u_h(T-h)^2 - f(0)u_h(0)^2) - \frac 12 \int_0^{T-h}\int_{\Gamma_0} f'(t+h)u_h^2 + \int_0^{T-h}\int_{\Gamma_0} uu_h\partial_t f_h.
\end{align*}
We discuss each term in turn. Writing
\[u_h(s)^2 = \left(\frac 1h \int_s^{s+h}u\right)^2 \leq \frac 1h \int_s^{s+h}u^2\]
hence
$\int_{\Gamma_0} f(T)u_h(T-h)^2 \leq \frac 1h \int_{T-h}^T \int_{\Gamma_0} |f(T)|u^2 \leq h\norm{f(T)}{L^1(\Gamma_0)}\norm{u}{L^\infty(0,T;L^\infty(\Gamma_0))}^2$, so that the fourth integral is bounded. Similarly to this we may deal with the first integral. The third term we deal with as follows:
\begin{align*}
\int_{\Gamma_0} [uf]_h(s)u_h(s) &= \frac{1}{h^2}\int_{\Gamma_0} \left(\int_s^{s+h} u(r)f(r)\right)\left( \int_s^{s+h}u(p)\right)\\
&=\frac{1}{h^2}\int_s^{s+h} \int_s^{s+h}\int_{\Gamma_0} u(r)f(r)u(p)\\
&\leq \frac{1}{h^2}\norm{f}{L^\infty(0,T;L^\infty)}\int_{\Gamma_0} \left(\int_s^{s+h}u(r)\right)^2\\
&\leq |\Gamma|^2\norm{f}{L^\infty(0,T;L^\infty(\Gamma_0))}\norm{u}{L^\infty(0,T;L^\infty(\Gamma_0))}^2
\end{align*}
thanks to Jensen's inequality. The second and last terms become
\begin{align*}
&\int_0^{T-h}\int_{\Gamma_0} uu_h\partial_t f_h - u_h \partial_t \hat{z}_h\\
&\leq \norm{u}{L^\infty(0,T;L^\infty(\Gamma_0))}\norm{\partial_t f_h}{L^2(0,T-h;L^2(\Gamma_0))}\norm{u_h}{L^2(0,T-h;L^2(\Gamma_0))} + \norm{\partial_t \hat{z}_h}{L^2(0,T-h;L^2(\Gamma_0))}\norm{u_h}{L^2(0,T-h;L^2(\Gamma_0))}.
\end{align*}
The time derivatives of $f_h$ and $\hat{z}_h$ are bounded because the Steklov average commutes with time derivative and we know that $w_t$ and $z_t$ are bounded in $L^2(0,T;L^2(\Gamma_0))$. The same holds for $j_0$ since $\Phi$ is a $C^2$-diffeomorphism. Finally, like above,
\begin{align*}
\int_0^{T-h}\int_{\Gamma_0} f'(t+h)u_h^2 &\leq \frac 1h\int_0^{T-h}\int_{\Gamma_0} |f'(t+h)|\int_t^{t+h}u^2(r)\\
&\leq \norm{u}{L^\infty(0,T;L^\infty(\Gamma_0))}^2 \norm{f'(\cdot+h)}{L^1(0,T-h;L^1(\Gamma_0))}.
\end{align*}
\subsection{Conclusion}
Integrating \eqref{eq:pbWeakForm} (tested with $\xi u_h'$ of course) and combining \eqref{eq:step1}, \eqref{eq:step2} and \eqref{eq:step3} from Lemma \ref{lem:ssInequalities}, we find
\begin{align*}
(1-\epsilon)\int_0^{T-h}\int_{\Omega_0}\xi(t)|\partial_t u_h(t)|^2 +  (\lambda_T-(d+1)\rho)\norm{\grad u_h(T-h)}{L^2(\Omega_0)}^2   \leq  C
\end{align*}
Then if we just pick $\epsilon$ and $\rho$ sufficiently small, we obtain
\[\int_0^{T-h}\int_{\Omega_0}\xi(t)|\partial_t u_h(t)|^2 \leq C\]
independent of $h$. Since $\xi$ is compactly supported around zero, we have (for a subsequence) that $\partial_t u_h \to v$ for some $v$ in $L^2(\tau, T;L^2(\Omega_0))$ for every $\tau$. Since $u_h \to u$ in $L^2(0,T-\delta; H^1(\Omega_0))$ for every $\delta$, we conclude that on $(\tau, T - \delta)$, $u_t = v$. Therefore it follows for $\tau > 0$ that $u_t \in L^2(\tau, T; L^2(\Omega_0))$ and we have proved Theorem \ref{thm:uStrongSolution}.

\section{Exponential convergence to equilibrium}\label{sec:equil}
In this section we prove that the solution of the system \eqref{eq:modelStationary} with $\delta_\Omega=\delta_k=\delta_{k'} =1$ (repeated here for convenience)
\begin{equation*}
\begin{aligned}
 u_t - \Delta u &=0 &&\text{in }\Omega\\
\nabla u \cdot \nu &= z - uw &&\text{on }\Gamma\\
\grad u \cdot \nu &=0 &&\text{on $\partial D$}\\
w_t -\delta_\Gamma \Delta_\Gamma w&=z - uw&&\text{on }\Gamma\\
z_t-\delta_{\Gamma'}\Delta_\Gamma z&=uw-z &&\text{on }\Gamma
\end{aligned}
\end{equation*}
converges to an equilibrium $(u_\infty, w_\infty, z_\infty)$ 
consisting of the non-negative constants uniquely determined by the well-posed system
\begin{equation}\label{eq:algebraic}
\begin{aligned}
|\Omega|u_\infty + |\Gamma|z_\infty &= M_1\\
|\Gamma|(w_\infty + z_\infty) &= M_2\\
u_\infty w_\infty &= z_\infty,
\end{aligned}
\end{equation}
where $M_1$ and $M_2$ were defined in \eqref{eq:conservationMass}. 
It is instructive to give the outline for the proof of Theorem \ref{thm:equilibrium} now and leave the details to be filled in below.
\begin{proof}[Proof of Theorem \ref{thm:equilibrium}]
Recall from \S \ref{sec:mainResults} the definition of the entropy and dissipation functionals $E$ and $D$ given by \eqref{eq:entropy} and \eqref{eq:dissipation}. Theorem \ref{thm:equilibriumDandE} gives the differential inequality
\[\frac{d}{dt}(E(u,w,z) - E(u_\infty, w_\infty, z_\infty) = -D(u,w,z) \leq -K(E(u,w,z)-E(u_\infty, w_\infty, z_\infty))\]
where $K>0$, and Gronwall's inequality along with the lower bound on the relative entropy from Theorem \ref{thm:equilibriumRelativeEntropy} leads to
\begin{align*}
C\left(\norm{u(t)- u_\infty(t)}{L^1(\Omega)}^2 + \norm{w(t)-w_\infty(t)}{L^1(\Gamma)}^2 + \norm{z(t)-z_\infty(t)}{L^1(\Gamma)}^2\right) 
&\leq e^{-Kt}(E(u_0,w_0,z_0) -E(u_\infty, w_\infty, z_\infty)).
\end{align*}
\end{proof}
Let us now introduce some functional inequalities on $M=\Omega$ or $\Gamma$ that we will use in proving some of the steps outlined above. Recall the standard Poincar\'e inequality:
\[\norm{u-\bar u}{L^2(M)} \leq C_P(M)\norm{\grad_M u}{L^2(M)}.\]
We need also the \emph{logarithmic Sobolev inequality} \cite{Stroock1993}: for $\bar u \neq 0$,
\begin{equation}\label{eq:logSobolev}
\int_M u\log\left(\frac{u}{\bar u}\right) \leq 4C_L(M)\int_M |\grad_M \sqrt{u}|^2.
\end{equation}
See \cite{Desvillettes2014} where we learn that this is in fact a consequence of the Poincar\'e inequality and the Sobolev inequality \eqref{eq:sobolevInequality}. 
Another estimate we require is  a lower bound for the entropy called the \emph{Csiszar--Kullback--Pinsker inequality} \cite{Cs}: for $\bar u \neq 0$,
\begin{equation}\label{eq:CKP}
\int_M u\log \left(\frac{u}{\bar u}\right) \geq \frac{1}{2|M|\bar u}\norm{u-\bar u}{L^1(M)}^2.
\end{equation}
\subsection{The Csiszar--Kullback--Pinsker inequality}
The following theorem is a type of Csiszar--Kullback--Pinsker inequality for functions satisfying the conservation laws \eqref{eq:conservationMass}.
\begin{theorem}\label{thm:equilibriumRelativeEntropy}
Let $u\colon \Omega \to \mathbb{R}$ and $w, z \colon \Gamma \to \mathbb{R}$ be non-negative measurable functions satisfying the conservation laws \eqref{eq:conservationMass} with $M_1, M_2 >0$. There exists a constant $C$ depending on $M_1$ and $M_2$ such that
\begin{align*}
E(u,w,z) - E(u_\infty, w_\infty, z_\infty) 
\geq C\left(\norm{u- u_\infty}{L^1(\Omega)}^2 + \norm{w-w_\infty}{L^1(\Gamma)}^2 + \norm{z- z_\infty}{L^1(\Gamma)}^2\right).
\end{align*}
\end{theorem}
\begin{proof}
We follow the proof of Lemma 2.4 in \cite{Fellner2015}. 
Since $u\log u - u_\infty\log u_\infty = u\log\left(u\slash u_\infty\right) + (u-u_\infty)\log u_\infty$ and (using $z_\infty = u_\infty w_\infty$)
\begin{align*}
\int_\Omega &(u-u_\infty)\log u_\infty + \int_\Gamma(w-w_\infty)\log w_\infty + \int_\Gamma (z-z_\infty)\log z_\infty\\
&=\log u_\infty \left(\int_\Omega (u-u_\infty)+ \int_{\Gamma}(z-z_\infty)\right) + \log w_\infty \int_\Gamma(w-w_\infty + z-z_\infty)\\
&=0,
\end{align*}
we find
\begin{align*}
&E(u,w,z) - E(u_\infty, w_\infty, z_\infty)\\
&= \int_\Omega u\log \left(\frac{u}{\bar u}\right) + \int_\Gamma w\log \left(\frac{w}{\bar w}\right) + \int_\Gamma z\log\left(\frac{z}{\bar z}\right)+ |\Omega|\left(\bar u\log\left(\frac{\bar u}{u_\infty}\right) - (\bar u-u_\infty)\right)\\
&\quad + |\Gamma|\left(\bar w\log\left(\frac{\bar w}{w_\infty}\right) - (\bar w-w_\infty)\right) + |\Gamma|\left(\bar z\log\left(\frac{\bar z}{z_\infty}\right) - (\bar z-z_\infty)\right)\tag{using $u\slash u_\infty = (u\slash \bar u) \times (\bar u \slash u_\infty)$}\\
&\geq \frac{1}{2}\left(\frac{1}{M_1}\norm{u-\bar u}{L^1(\Omega)}^2 + \frac{1}{M_2}\norm{w-\bar w}{L^1(\Gamma)}^2 + \frac{1}{M_2}\norm{z-\bar z}{L^1(\Gamma)}^2\right)+ \frac{|\Omega|^2}{4M_1}(\bar u-u_\infty)^2 + \frac{|\Gamma|^2}{4M_2}(\bar w-w_\infty)^2 \\
&\quad + \frac{|\Gamma|^2}{4M_2}(\bar z-z_\infty)^2\tag{by the Csiszar--Kullback--Pinsker inequality \eqref{eq:CKP} and \eqref{eq:xyineq} below}\\
&\geq C_1\left(\norm{u-\bar u}{L^1(\Omega)}^2 + \norm{w-\bar w}{L^1(\Gamma)}^2 + \norm{z-\bar z}{L^1(\Gamma)}^2+\norm{\bar u-u_\infty}{L^1(\Omega)}^2 + \norm{\bar w-w_\infty}{L^1(\Gamma)}^2 +\norm{\bar z-z_\infty}{L^1(\Gamma)}^2\right)
\end{align*}
where $C_1:=\min\left({1}\slash {4M_1}, {1}\slash {4M_2}\right)$ and we used the inequality
\begin{equation}\label{eq:xyineq}
x\log\left(\frac xy\right)-(x-y) 
\geq \frac{(x-y)^2}{2x+2y}
\end{equation}
with $x=\bar u$ and $y=u_\infty$ 
and the bounds $\bar u, u_\infty \leq M_1\slash |\Omega|$. Finally, we obtain the result by using on the right hand side of the above calculation 
\begin{align*}
\norm{u-\bar u}{L^1(\Omega)}^2 + \norm{\bar u-u_\infty}{L^1(\Omega)}^2 &\geq \frac{1}{2}\left(\norm{u-\bar u}{L^1(\Omega)} + \norm{\bar u-u_\infty}{L^1(\Omega)}\right)^2\\
&\geq \frac{1}{2}\left(\int_\Omega \sqrt{|u-\bar u|^2 + |\bar u-u_\infty|^2}\right)^2\\
&\geq \frac{1}{2}\norm{u-u_\infty}{L^1(\Omega)}^2
\end{align*}
(since $|u-u_\infty|^2 = |u-\bar u|^2 + |\bar u - u_\infty|^2$).
\end{proof}
\subsection{Entropy entropy-dissipation estimate}
The aim now is to relate the dissipiation with the relative entropy. More precisely, we want to show that
\[D(u,w,z) \geq K(E(u,w,z) - E(u_\infty, w_\infty, z_\infty))\]
for a positive constant $K$. This will require some technical results in the form of the next two lemmas. The above inequality will then be proved in Theorem  \ref{thm:equilibriumDandE}. We denote $U_\infty = \sqrt{u_\infty}$ and similarly for $w_\infty$ and $z_\infty$. The next lemma is established along the same method as Lemma 3.1 in \cite{FellnerLaamri} so we shall omit the proof.
\begin{lem}\label{lem:equilibriumConstants}
Let $a, b, c \geq 0$ be constants such that the equalities
\begin{align*}
|\Omega|a^2 + |\Gamma|c^2 &= M_1 = |\Omega|U_\infty^2 + |\Gamma|Z_\infty^2\\
|\Gamma|(b^2 + c^2) &= M_2 = |\Gamma|(W_\infty^2 + Z_\infty^2)
\end{align*}
hold. Then
\[(a-U_\infty)^2 + (b-W_\infty)^2 + (c-Z_\infty)^2 \leq C(ab-c)^2.\]
\end{lem}
The next lemma is a version of Lemma \ref{lem:equilibriumConstants} for functions.
\begin{lem}\label{lem:equilibriumFunctions}
Let $A \colon \Omega \to \mathbb{R}$ and $B, C\colon \Gamma \to \mathbb{R}$ be non-negative measurable functions with $A \in H^1(\Omega)$ satisfying the conservation laws
\begin{align*}
|\Omega|\bar{A^2} + |\Gamma|\bar{C^2} &= M_1 = |\Omega|U_\infty^2 + |\Gamma|Z_\infty^2\\
|\Gamma|(\bar{B^2} + \bar{C^2}) &= M_2 = |\Gamma|(W_\infty^2 + Z_\infty^2).
\end{align*}
Then there exist constants $L^1$, $L^2$ and $L^3$ such that
\begin{align*}
&\norm{A-U_\infty}{L^2(\Omega)}^2 + \norm{B-W_\infty}{L^2(\Gamma)}^2 + \norm{C-Z_\infty}{L^2(\Gamma)}^2\\
&\leq L_1\norm{AB-C}{L^2(\Gamma)}^2+ L_2(\rho)\left(\norm{A-\bar A}{L^2(\Omega)}^2 + \norm{B-\bar B}{L^2(\Gamma)}^2 + \norm{C-\bar C}{L^2(\Gamma)}^2\right) + \rho\norm{\grad A}{L^2(\Omega)}^2
\end{align*}
for $\rho > 0$ arbitrarily small.
\end{lem}
\begin{proof}We shall adapt techniques presented in the proofs of \cite[Lemma 3.5]{BaoFellnerLatos} and \cite[Lemma 3.2]{FellnerLaamri}; the latter paper covers our type of reaction term but all quantities are on a single (bulk) domain, whilst the former has a rather different reaction term to the one in our case but there is a bulk-surface coupling between the quantities. 

The proof is split into two steps. It is important to bear in mind that below all mean values of $A$ or $A^2$ are taken over the domain $\Omega$ and never on $\Gamma$. 

\noindent \textsc{Step 1.} In the first step, we shall prove
\begin{align*}
\norm{AB-C}{L^2(\Gamma)}^2 \geq \frac{|\Gamma|}{2}|\bar A \bar B-\bar C|^2 - K_1\left(\norm{A-\bar A}{L^2(\Gamma)}^2 + \norm{B-\bar B}{L^2(\Gamma)}^2+\norm{C-\bar C}{L^2(\Gamma)}^2\right).
\end{align*}
This inequality establishes a relationship between the $L^2$ norm of $A$ on $\Gamma$ and the mean value of $A$ on $\Omega$. Define $\delta_1 = A-\bar A$ on $\Omega$ and $\delta_2$ and $\delta_3$ on $\Gamma$ similarly. Define the set
\[S := \{x \in \Gamma \mid |\delta_1(x)|, |\delta_2(x)|, |\delta_3(x)| \leq K\}\]
which is  well defined due to the trace theorem. Since $AB=(\bar A + \delta_1)(\bar B + \delta_2) = \bar A \bar B + \bar A\delta_2 + \bar B  \delta_1 +  \delta_1 \delta_2$ and $C=\bar C +\delta_3$, we have with the aid of the identity $2ab \leq  a^2\slash 2 + 2 b^2$ that
\begin{align}
\nonumber \norm{ AB-C}{L^2(S)}^2 
 &\geq \frac{1}{2}\norm{\bar A \bar B - \bar C}{L^2(S)}^2 - \norm{\bar A \delta_2 + \bar B  \delta_1 +  \delta_1\delta_2 - \delta_3}{L^2(S)}^2\\
&\geq \frac{1}{2}\norm{\bar A \bar B - \bar C}{L^2(S)}^2 - \underbrace{2\max\left(\frac{M_1}{|\Omega|}, \frac{M_2}{|\Gamma|}, K^2, 1\right)}_{=:K_1}\int_S |\delta_2|^2 + | \delta_1|^2 + |\delta_3|^2.\label{eq:a}
\end{align}
where we used that $\bar A^2 \leq \bar{A^2} \leq M_1\slash |\Omega|$. Now we work on $S^\perp = \{x \in \Gamma \mid |\delta_1(x)| \text{ or } |\delta_2(x)| \text{ or } |\delta_3(x)| > K\}$. By Chebyshev's inequality,
\[|\{\delta_i^2 > K^2\}| 
\leq \frac{1}{K^2}\int_{\Gamma}|\delta_i|^2\]
and therefore
$|S^\perp| \leq (3\slash K^2)(\int_\Gamma | \delta_1|^2 + |\delta_2|^2 + |\delta_3|^2)$. Hence
\[\norm{\bar A \bar B-\bar C}{L^2(S^\perp)}^2 \leq \left(\frac{M_1}{|\Omega|}+1\right)\frac{M_2}{|\Gamma|}|S^\perp| \leq K_2\left(\int_\Gamma | \delta_1|^2 + |\delta_2|^2 + |\delta_3|^2\right)\] 
where $K_2= (3M_2\slash K^2|\Gamma|)\left(M_1\slash |\Omega|+1\right)$. This leads to 
\begin{align*}
\norm{ A B-C}{L^2(S^\perp)}^2 &\geq \norm{\bar A \bar B-\bar C}{L^2(S^\perp)} -K_2\left(\int_\Gamma | \delta_1|^2 + |\delta_2|^2 + |\delta_3|^2\right)\end{align*}
since the right hand side is non-positive, and finally, combining this with \eqref{eq:a},
\begin{align}
\norm{AB-C}{L^2(\Gamma)}^2 
&\geq \frac{|\Gamma|}{2}|\bar A \bar B- \bar C|^2 - \underbrace{\left(K_1+\frac{K_2}{2}\right)}_{=: K_3}\left(\norm{ A - \bar A}{L^2(\Gamma)}^2 + \norm{B-\bar B}{L^2(\Gamma)}^2 + \norm{C-\bar C}{L^2(\Gamma)}^2 \right).\label{eq:ineq1}
\end{align}
\textsc{Step 2.} Define $\mu_i$ by
\[\sqrt{\bar{A^2}}=U_\infty(1+\mu_1), \qquad \sqrt{\bar{B^2}}=W_\infty(1+\mu_2),\qquad \sqrt{\bar{C^2}}=Z_\infty(1+\mu_3)\]
and observe that $\mu_i \geq -1$. 
We see that
$|\Omega|\bar{\delta_1^2}  
= \norm{A-\bar A}{L^2(\Omega)}^2 
= |\Omega|(\bar{A^2}-\bar A^2)$
which implies
\begin{equation}\label{eq:meanValue1}
\bar A =  \sqrt{\bar{A^2}} - \frac{\bar{\delta_1^2}}{\sqrt{\bar{A^2}}+\bar A} = U_\infty(1+\mu_1) - \frac{\bar{\delta_1^2}}{\sqrt{\bar{A^2}}+\bar A}.
\end{equation}
Thus the terms on the left hand side of the statement to be proved can be written as
\begin{align*}
\norm{A-U_\infty}{L^2(\Omega)}^2 
&= |\Omega|\left(U_\infty^2(1+\mu_1)^2 - 2U_\infty^2(1+\mu_1) + \frac{2U_\infty\bar{\delta_1^2}}{\sqrt{\bar{A^2}}+\bar A} + {U_\infty^2}\right)\\
&=|\Omega|\left(\mu_1^2U_\infty^2 + \frac{2U_\infty \bar{\delta_1^2}}{\sqrt{\bar{A^2}}+\bar A}\right)
\end{align*}
which is unbounded for vanishing $\bar{A^2}$ (and $\bar{A^2} \geq \bar A^2$). So we consider two subcases:

\noindent \textsc{Case 1} ($\bar{A^2}, \bar{B^2}, \bar{C^2} \geq \epsilon^2$).
Now, observe that the left hand side of the statement to be proved is
\begin{align*}
&\norm{A-U_\infty}{L^2(\Omega)}^2 + \norm{B-W_\infty}{L^2(\Gamma)}^2 + \norm{C-Z_\infty}{L^2(\Gamma)}^2\\
&= |\Omega|\left(\mu_1^2U_\infty^2 + \frac{2U_\infty \bar{\delta_1^2}}{\sqrt{\bar{A^2}}+\bar A}\right) + |\Gamma|\left(\mu_2^2W_\infty^2 + \frac{2W_\infty \bar{\delta_2^2}}{\sqrt{\bar{B^2}}+\bar B}\right) +|\Gamma|\left(\mu_3^2Z_\infty^2 + \frac{2Z_\infty \bar{\delta_3^2}}{\sqrt{\bar{C^2}}+\bar C}\right)\\
&\leq |\Omega|\mu_1^2U_\infty^2 + |\Gamma|\mu_2^2W_\infty^2 + |\Gamma|\mu_3^2Z_\infty^2 + \underbrace{\frac{2}{\epsilon}\max\left(\sqrt{\frac{M_1}{|\Omega|}}, \sqrt{\frac{M_2}{|\Gamma|}}\right)}_{=: K_4}\left(\bar{\delta_1^2} + \bar{\delta_2^2}+\bar{\delta_3^2}\right).
\end{align*}
To deal with the right hand side of the statement to be shown, first defining
\[F_A := \frac{1}{\sqrt{\bar{A^2}}+\bar A} \leq \frac{1}{\sqrt{\bar{A^2}}} \leq \frac{1}{\epsilon}\qquad\text{and}\qquad  G_{U,B}:= U_\infty(1+\mu_1)F_B = \sqrt{\bar{U^2}}F_B \leq \frac{1}{\epsilon}\sqrt{\frac{M_1}{|\Omega|}},\]
note that, with \eqref{eq:meanValue1},
\begin{align}
\nonumber \bar A\bar B - \bar C
 &= \left(U_\infty(1+\mu_1) - {\bar{\delta_1^2}}F_A\right)\left(W_\infty(1+\mu_2) - {\bar{\delta_2^2}}F_B\right) - \left(Z_\infty(1+\mu_3) - {\bar{\delta_3^2}}F_C\right)\\
\nonumber &= Z_\infty((1+\mu_1)(1+\mu_2)-(1+\mu_3))  - {\bar{\delta_2^2}}G_{U,B} - {\bar{\delta_1^2}}G_{W,A} + {\bar{\delta_1^2}}{\bar{\delta_2^2}}F_AF_B + {\bar{\delta_3^2}}F_C\\
&\geq Z_\infty((1+\mu_1)(1+\mu_2)-(1+\mu_3)) - \frac{\sqrt{K_5}}{\sqrt{2C_M}}(\bar{\delta_1^2} + \bar{\delta_2^2} +\bar{\delta_3^2})\label{eq:inequal2}
\end{align}
where we used $
\bar{\delta_1^2} \leq \bar{A^2} + {\bar A^2}
\leq 2M_1\slash |\Omega|$ and $K_5=K_5(\epsilon, M_1, M_2)$ is defined so that the inequality holds, and $C_M$ is the upper bound $(\bar{\delta_1^2} + \bar{\delta_2^2} +\bar{\delta_3^2}) \leq C_M$.
So the right hand side of the claimed inequality becomes, recalling \eqref{eq:ineq1} and the interpolated trace inequality
\begin{align*}
&L_1\norm{AB-C}{L^2(\Gamma)}^2 + L_2\left(\norm{A-\bar A}{L^2(\Omega)}^2 + \norm{B-\bar B}{L^2(\Gamma)}^2 + \norm{C-\bar C}{L^2(\Gamma)}^2\right) + \rho\norm{\grad A}{L^2(\Omega)}^2 \\
&\geq L_1\left( \frac{|\Gamma|}{2}|\bar A \bar B-\bar C|^2 - K_3\left(C_\rho\norm{A-\bar A}{L^2(\Omega)}^2 + \frac{\rho}{K_3L_1}\norm{\grad A}{L^2(\Omega)}^2 + \norm{B-\bar B}{L^2(\Gamma)}^2+\norm{C-\bar C}{L^2(\Gamma)}^2\right)\right)\\
&\quad+  L_2\left(\norm{A-\bar A}{L^2(\Omega)}^2 + \norm{B-\bar B}{L^2(\Gamma)}^2 + \norm{C-\bar C}{L^2(\Gamma)}^2\right)+\rho\norm{\grad A}{L^2(\Omega)}^2\\
&\geq \frac{|\Gamma|L_1}{4}\left(Z_\infty^2((1+\mu_1)(1+\mu_2)-(1+\mu_3))^2 - K_5(\bar{\delta_1^2} + \bar{\delta_2^2} + \bar{\delta_3^2})\right)\\
&\quad+ (L_2-K_3L_1\max(C_\rho, 1))\left(|\Omega|\bar{\delta_1^2} + |\Gamma|(\bar{\delta_2^2} + \bar{\delta_3^2})\right)\tag{by \eqref{eq:inequal2}}\\
&\geq \frac{L_1|\Gamma|}{4}\left(Z_\infty^2((1+\mu_1)(1+\mu_2)-(1+\mu_3))^2\right)\\
&\quad+ \underbrace{\left((L_2-K_3L_1\max(C_\rho, 1))\min(|\Omega|, |\Gamma|)-\frac{|\Gamma|L_1K_5}{4}\right)}_{=:K_6}\left(\bar{\delta_1^2} + \bar{\delta_2^2} + \bar{\delta_3^2}\right)
\end{align*}
where when we recalled \eqref{eq:inequal2} we used that if $a \geq b -c$ for non-negative constants then $2a^2 \geq b^2 - 2c^2$ and the fact that the $\bar{\delta_i^2}$ are bounded. It suffices to prove that
\begin{align*}
|\Omega|\mu_1^2U_\infty^2 + |\Gamma|\mu_2^2W_\infty^2 + |\Gamma|\mu_3^2Z_\infty^2 &\leq \frac{L_1|\Gamma|}{4}\left(Z_\infty^2((1+\mu_1)(1+\mu_2)-(1+\mu_3))^2\right) + (K_6-K_4)\left(\bar{\delta_1^2} + \bar{\delta_2^2} + \bar{\delta_3^2}\right)
\end{align*}
which follows by Lemma \ref{lem:equilibriumConstants} provided $K_6\geq K_4$ and this is the case if $L_2$ is chosen sufficiently large.

\noindent \textsc{Case 2} ($\bar{A^2}$  or $\bar{B^2}$ or $\bar{C^2} \leq \epsilon^2$). 
In this case, we bound the $L^2$ norm of $A-U_\infty$ explicitly in terms of the $M_i$:
\[\norm{A-U_\infty}{L^2(\Omega)}^2 + \norm{B-W_\infty}{L^2(\Gamma)}^2 + \norm{C-Z_\infty}{L^2(\Gamma)}^2  \leq 2M_1 + 4M_2.\]
Consider the case that $\bar{A^2} \leq \epsilon^2.$ Just like at the start of Step 2, we have $\bar C^2 = \bar{C^2} - \bar{\delta_3^2}$ and thus
\begin{align*}
\bar C^2 &=\frac{M_1}{|\Gamma|} - \frac{|\Omega|}{|\Gamma|}\bar{A^2} - \bar{\delta_3^2}
\geq \frac{M_1}{|\Gamma|} - \frac{|\Omega|}{|\Gamma|}(\epsilon^2+\bar{\delta_1^2}) - \bar{\delta_3^2}.
\end{align*} 
With this, we may write, using $\bar A^2 \leq \bar{A^2} \leq \epsilon^2$,
\begin{align}
|\bar C-\bar A \bar B|^2 
\geq \frac{M_1}{|\Gamma|} - \frac{|\Omega|}{|\Gamma|}(\epsilon^2+\bar{\delta_1^2}) - \bar{\delta_3^2} - 2\bar B \bar C \epsilon\label{eq:toComp}
\end{align}
and so if $L_2$ is large enough and $\epsilon$ is small enough,
\begin{align*}
|\bar C-\bar A \bar B|^2 + L_2(\bar{\delta_1^2} + \bar{\delta_2^2} + \bar{\delta_3^2}) &\geq \frac{M_1}{|\Gamma|} - \frac{|\Omega|}{|\Gamma|}\epsilon^2 - 2\bar B \bar C \epsilon + \left(L_2 - \frac{|\Omega|}{|\Gamma|}\right)\bar{\delta_1^2} + L_2\bar{\delta_2^2} + (L_2-1) \bar{\delta_3^2}   \geq C^* > 0.
\end{align*}
From the top, we find
\begin{align*}
\norm{A-U_\infty}{L^2(\Omega)}^2 + \norm{B-W_\infty}{L^2(\Gamma)}^2 + \norm{C-Z_\infty}{L^2(\Gamma)}^2 
&\leq \frac{2M_1 + 4M_2}{C^*}\left(|\bar C-\bar A \bar B|^2 + L_2(\bar{\delta_1^2} + \bar{\delta_2^2} + \bar{\delta_3^2})\right).
\end{align*}
The $\bar{B^2} \leq \epsilon^2$ case can be dealt with similarly. For the $\bar{C^2} \leq \epsilon^2$ case, we see that, like above,
\begin{align*}
\bar A^2 
&\geq \frac{M_1}{|\Omega|} - \frac{|\Gamma|}{|\Omega|}(\epsilon^2+\bar{\delta_3^2}) - \bar{\delta_1^2} \qquad\text{and}\qquad
\bar B^2 
\geq \frac{M_2}{|\Gamma|}  - (\epsilon^2 + \delta_3^2)- \bar{\delta_2^2}
\end{align*}
and so in lieu of \eqref{eq:toComp}
\begin{align*}
(\bar C - \bar A \bar B)^2 &\geq (\bar A \bar B)^2 - 2\bar A \bar B \bar C
\geq \left(\frac{M_1}{|\Omega|} - \frac{|\Gamma|}{|\Omega|}(\epsilon^2+\bar{\delta_3^2}) - \bar{\delta_1^2} \right)\left(\frac{M_2}{|\Gamma|}  - (\epsilon^2 + \delta_3^2)- \bar{\delta_2^2}\right) - 2\bar A \bar B \epsilon 
\end{align*}
and the same argument as before gives the result.
\end{proof}

Finally we are able to state and prove the desired entropy-dissipation estimate.

\begin{theorem}\label{thm:equilibriumDandE}
There exists a constant $K > 0$ such that
\[D(u,w,z) \geq K(E(u,w,z) - E(u_\infty, w_\infty, z_\infty)).\]
\end{theorem}
\begin{proof}
This theorem can be proved like Proposition 3.1 in \cite{FellnerLaamri}. Define the continuous function $\varphi\colon (0,\infty)^2 \to \mathbb{R}$ by
\[\varphi(x,y) = \frac{x\log(\frac xy) -(x-y)}{(\sqrt{x}-\sqrt{y})^2} = \varphi\left(\frac xy, 1\right)\]
which is increasing in its first argument and satisfies $\varphi(\bar u, u_\infty) = \varphi(\frac{\bar u}{u_\infty},1)$ and $\bar u\slash u_\infty \leq \frac{M}{u_\infty}$ so that $\varphi(\bar u\slash u_\infty, 1) \leq C_M$ where $C_M$ depends on the conservation law constants and the equilibrium values.  Then
\begin{align*}
\bar u\log\left(\frac{\bar u}{u_\infty}\right) - (\bar u - u_\infty) &\leq C_M(\sqrt{\bar u}-\sqrt{u_\infty})^2\tag{by definition of $\varphi$ and the bound above}\\
&= C_M(U_\infty(1+\mu_1)-\sqrt{u_\infty})^2\\
&= C_MU_\infty^2\mu_1^2
\end{align*}
where we used $\bar u = \bar{U^2} = U_\infty^2(1+\mu_1)^2$ by definition of the $\mu_i$ in the proof of Lemma \ref{lem:equilibriumFunctions}. Using this and the log Sobolev inequality \eqref{eq:logSobolev}, we find
\begin{align*}
&E(u,w,z)-E(u_\infty, w_\infty, z_\infty)\\
&= \int_\Omega u\log \frac{u}{\bar u} + \int_\Gamma w\log \frac{w}{\bar w} + \int_\Gamma z\log \frac{z}{\bar z}\\
&\quad+ |\Omega|\left(\bar u \log \frac{\bar u}{u_\infty} - (\bar u-u_\infty)\right)+ |\Gamma|\left(\bar w \log \frac{\bar w}{w_\infty} - (\bar w-w_\infty)\right) + |\Gamma|\left(\bar z \log \frac{\bar z}{z_\infty} - (\bar z-z_\infty)\right)\\
&\leq 4\max(C_L(\Omega), L_\Gamma)\left(\int_{\Omega}|\nabla U|^2 + \int_{\Gamma}|\sgrad W|^2 + \int_{\Gamma}|\sgrad Z|^2\right) + |\Omega|C_MU_\infty^2\mu_1^2 + |\Gamma|C_M(W_\infty^2 \mu_2^2 + Z_\infty^2\mu_3^2).
\end{align*}
On the other hand, using $(uw-z)(\log (uw) - \log z) \geq 4(UW-Z)^2$, we find
\begin{align*}
D(u,w,z) &\geq 4\left(\int_\Omega |\grad U|^2 + \delta_\Gamma \int_\Gamma |\sgrad W|^2 + \delta_{\Gamma'}\int_\Gamma |\sgrad Z|^2 \right) + 4\int_\Gamma (UW-Z)^2\\
&\geq 4\norm{UW-Z}{L^2(\Gamma)}^2 + 4C_P\theta\min(1,\delta_\Gamma, \delta_{\Gamma'})\left(\norm{U-\bar U}{L^2(\Omega)}^2 + \norm{W-\bar W}{L^2(\Gamma)}^2 + \norm{Z-\bar Z}{L^2(\Gamma)}^2\right)\\
&\quad + 4(1-\theta)\min(1, \delta_\Gamma, \delta_{\Gamma'})\left(\int_\Omega |\nabla U|^2 + \int_\Gamma |\sgrad W|^2 + \int_\Gamma |\sgrad Z|^2 \right)
\end{align*}
where in the final inequality we simply wrote  $1=\theta + (1-\theta)$ and employed Poincar\'e's inequality on one part of this separation. To conclude, we need to prove that the right hand side of the above exceeds
\begin{align*}
K\left(4\max(C_L(\Omega), L_\Gamma)\left(\int_{\Omega}|\nabla U|^2 + \int_{\Gamma}|\sgrad W|^2 + \int_{\Gamma}|\sgrad Z|^2\right) + |\Omega|C_MU_\infty^2\mu_1^2 + |\Gamma|C_M(W_\infty^2 \mu_2^2 + Z_\infty^2\mu_3^2)\right)
\end{align*}
for a positive constant $K$. If we choose $K$ so that ${4(1-\theta)\min(1, \delta_\Gamma, \delta_{\Gamma'})}{} - 4K\max(C_L(\Omega), L_\Gamma) \geq \epsilon$ for $\epsilon$ sufficiently small (see below),  we are left to prove that
\begin{align*}
&4\norm{UW-Z}{L^2(\Gamma)}^2 + 4C_P\theta\min(1, \delta_\Gamma, \delta_{\Gamma'})\left(\norm{U-\bar U}{L^2(\Omega)}^2 + \norm{W-\bar W}{L^2(\Gamma)}^2 + \norm{Z-\bar Z}{L^2(\Gamma)}^2\right)\\
& + \epsilon\left(\int_\Omega |\nabla U|^2 + \int_\Gamma |\sgrad W|^2 + \int_\Gamma |\sgrad Z|^2 \right)\\
&\quad\geq KC_M\left(|\Omega|U_\infty^2\mu_1^2 + |\Gamma|(W_\infty^2 \mu_2^2 + Z_\infty^2\mu_3^2)\right).
\end{align*}
Indeed, setting $A(\epsilon) := \min(4,4C_P\theta\min(1, \delta_\Gamma, \delta_{\Gamma'}), \epsilon)\min(L_1^{-1}, L_2(\epsilon)^{-1}, \epsilon^{-1})$, the left hand side of the above is greater than
\begin{align*}
&A(\epsilon) \left(L_1\norm{UW-Z}{L^2(\Gamma)}^2 + L_2(\epsilon)\left(\norm{U-\bar U}{L^2(\Omega)}^2 + \norm{W-\bar W}{L^2(\Gamma)}^2 + \norm{Z-\bar Z}{L^2(\Gamma)}^2\right) + \epsilon\norm{\grad U}{L^2(\Omega)}^2 \right)\\
&\geq A(\epsilon)\left(\norm{U-U_\infty}{L^2(\Omega)}^2 + \norm{W-W_\infty}{L^2(\Gamma)}^2 + \norm{Z-Z_\infty}{L^2(\Gamma)}^2 \right)\tag{by Lemma \ref{lem:equilibriumFunctions}}\\
&\geq A(\epsilon)(|\Omega|U_\infty^2\mu_1^2 + |\Gamma|W_\infty^2\mu_2^2 + |\Gamma|Z_\infty^2\mu_3^2).
\end{align*}
Now, fix $\epsilon$ so that
\[0 < \hat K := \frac{4(1-\theta)\min(1, \delta_\Gamma, \delta_{\Gamma'})  -\epsilon}{4\max(C_L(\Omega), L_\Gamma)} 
\]
then we pick $K$ to satisfy $0 < K \leq \hat K$ and $A(\epsilon) \geq KC_M$.
\end{proof}
\section{Conclusion}\label{sec:conclusion}
The extension of these results to higher dimensions is an open issue and perhaps the stage where we use the De Giorgi method can be improved to utilise the fact that the equations are coupled and thus it may make sense to treat the full $3 \times 3$ system in a unified approach to derive the $L^\infty$ bounds. In terms of the model, we could also consider equations on the interior of the surface $\Gamma(t)$, i.e. on $I(t)$; including such equations may result in more realistic models for certain applications and there is biological motivation \cite{BioPaper} to do so.  
We did not have time to consider some fast reaction and diffusion limits for the system. The idea is to send some of the parameters appearing in \eqref{eq:modelALE} to zero and study the resulting problems, which in fact turn out to be free boundary problems. The rigorous justification of these limits needs the resolvement of some technical issues so we will address this limiting behaviour in separate paper.
\section*{Acknowledgements}
AA acknowledges support from the Warwick Mathematics Platform grant 3117 when this work was started and the Weierstrass Institute when this work was finished. AA would like to thank CME, JT and the Mathematics Institute at the University of Warwick for organising a productive research visit in November 2016. The research of CME was partially supported
by the Royal Society via a Wolfson Research Merit Award. This work was done while JT was a Visiting Academic at the Mathematics Institute of the University of Warwick. She is grateful for the warm hospitality.  The authors thank the referees for their useful comments and feedback.
\appendix
\section{Derivation of the model}\label{sec:derivation}
In this section we derive the model system \eqref{eq:model} from conservation laws and transport theorems applied to the bulk and to the surface. We begin by addressing the bulk equation. Let $M(t)\subset \Omega(t)$ be a portion of the bulk domain with boundary $\partial M(t)$ moving with a velocity field $\mathbf V_a$ (which has to be such that the normal component of $\mathbf V_a$ agrees with $\mathbf V$ on $\Gamma$ and $\mathbf V_o$ on $\partial D$). Consider the conservation law
\begin{equation*}\label{cons.law}
\frac{d}{dt}\int_{M(t)}u=-\int_{\partial M(t)}\vec{q}\cdot\nu_M
\end{equation*}
where $\nu_M=\nu_M(t)$ is the outward normal vector to $\partial M(t)$ and no reaction term inside the domain is considered. Then, by the divergence theorem 
and Reynolds transport theorem we find that
$$\int_{M(t)}u_t+\nabla\cdot(u\mathbf V_a)=-\int_{M(t)}\nabla\cdot\vec{q}.$$ 
Now we choose $\vec{q}=-D\nabla u+u(\mathbf V_\Omega-\mathbf V_a)$ where the advective term takes into account the fact that points in $\Omega(t)$ are subject to a material velocity field $\mathbf V_\Omega$, and we use the arbitrariness of $M(t)$ to obtain
\begin{equation}\label{eqbulk}
u_t+\nabla\cdot(u\mathbf V_\Omega)-D\Delta u=0.
\end{equation}
We derive now the equations on the surface, along the same lines as in \cite{Dziuk2006}. As before let $M(t) \subset \Gamma(t)$ be a portion of the surface with boundary $\partial M(t)$ which is moving with the normal velocity $\mathbf V$. The conservation law now admits a reaction term inside $M$ in addition to the flux along the boundary $\partial M$
\begin{equation*}\label{cons.law2}
\frac{d}{dt}\int_{M(t)}\vartheta=\int_{M(t)}f-\int_{\partial M(t)}\vec{q}\cdot\mu
\end{equation*}
where $f$ is the forcing term to be defined later and $\mu$ is a conormal vector, that is, it is an outward unitary vector normal to $\partial M$ and tangential to $\Gamma$. The flux is given by the vector $\vec{q}$. Now, using the integration by parts formula
$$\int_{\partial M(t)} \vec{q}\cdot\mu=\int_{M(t)}\left(\nabla_\Gamma\cdot\vec{q}+\vec{q}\cdot\nu H\right)$$
and the transport theorem on surfaces, we may write
$$\int_{M(t)}\partial^\circ \vartheta + \vartheta\sgrad \cdot \mathbf V +\nabla_\Gamma\cdot\vec{q} + \vec{q}\cdot \nu H=\int_{M(t)}f,$$
where $\partial^\circ \vartheta=\vartheta_t+ \nabla \vartheta \cdot \mathbf V$ is the normal time derivative \cite{CFG, Dziuk2013}, which is the material derivative with respect to a velocity field $\mathbf{V}$ that is purely normal.
Similar to before we choose $\vec{q}=-D_\vartheta\nabla_\Gamma \vartheta+\vartheta(\mathbf V_\Gamma-\mathbf V)$ which gives the pointwise equation
\begin{equation}\label{eqsurface}
\partial^\circ \vartheta + \vartheta\sgrad \cdot \mathbf V  + \nabla_\Gamma\cdot\left(\vartheta\mathbf V_\Gamma^\tau\right)-D_\vartheta\Delta_\Gamma \vartheta=f.
\end{equation}
Equations \eqref{eqbulk} and \eqref{eqsurface} correspond to those on the model problem \eqref{eq:model} by taking $f=r$ and $D_\vartheta = \delta_\Gamma$ for the equation for $w$ and $f=-r$ and $D_\vartheta = \delta_{\Gamma'}$ for the equation for $z$.   
\section{Preliminary results}\label{sec:prelim}
In this section we collect some technical facts that are used in the paper. Here and below, $\Omega \subset \mathbb{R}^{d+1}$ is a sufficiently smooth bounded domain with $\partial\Omega =: \Gamma$. 
\subsection{Calculus identities}Observe that for sufficiently regular functions $a$ and $b$ defined on  $\Omega$ and vector fields $\mathbf{J}$, 
\[\int_\Omega (\mathbf{J}\cdot\grad a)b + (\grad\cdot\mathbf{J})ab  = \int_\Omega \grad \cdot (a\mathbf{J})b = \int_\Omega\grad \cdot (a\mathbf{J}b) - a\mathbf{J}\cdot \grad b = \int_{\Gamma}ab \mathbf{J} \cdot \nu -\int_\Omega a(\mathbf{J}\cdot \grad b).\]
From this we can deduce several expressions that will be useful throughout the paper. 

\noindent \textbf{Bulk identities.}
\begin{align}
    \nonumber \int_\Omega  (\mathbf{J}\cdot \grad a)a &= \frac 12 \int_{\Gamma} a^2\mathbf{J}\cdot \nu - \frac 12\int_\Omega (\grad \cdot \mathbf{J})a^2 \\
    \int_\Omega \grad \cdot (\mathbf{J}a)a_k  &= \int_\Omega (\grad \cdot \mathbf{J}) aa_k +\frac 12 \int_\Gamma a_k^2(\mathbf{J}\cdot\nu) - \frac{1}{2}\int_\Omega a_k^2 \grad \cdot \mathbf{J}\label{eq:bulkPosId}\\
    \int_\Omega \grad \cdot (\mathbf{J}_\Omega a)a^+  &= \frac 12\int_\Omega |a^+|^2(\grad \cdot \mathbf{J}_\Omega )  +\frac 12 \int_\Gamma j|a^+|^2\label{eq:bulkIBPidentity}
\end{align}
Here, we used that $\grad a=\grad a_k$  in $\operatorname{supp} \{a_k\}$ to write
$\grad\cdot (\mathbf{J}a)a_k= (\grad\cdot\mathbf{J})aa_k+ (\mathbf{J}\cdot\grad a_k)a_k$, and we recalled that $j$ is the jump on the velocities defined before. In a similar way we can deduce formulae if $a$, $b$ and $\mathbf J$ are now defined on $\Gamma$.

\noindent \textbf{Surface identities}
\begin{align}
\nonumber \int_\Gamma (\mathbf{J}\cdot \sgrad a)a &= -\frac 12 \int_{\Gamma} a^2H\mathbf{J}\cdot \nu - \frac 12\int_\Gamma (\sgrad \cdot \mathbf{J})a^2 \\
\int_\Gamma \sgrad \cdot (\mathbf{J}a)a_k  &= \int_\Gamma (\sgrad \cdot \mathbf{J}) aa_k  -\frac 12 \int_\Gamma a_k^2H(\mathbf{J}\cdot\nu) - \frac{1}{2}\int_\Gamma a_k^2 \sgrad \cdot \mathbf{J}\label{eq:surfacePosId}\\
\int_\Gamma \sgrad \cdot (\mathbf{J}_\Gamma a)a^+  &= \frac 12\int_\Gamma |a^+|^2 \sgrad \cdot \mathbf{J}_\Gamma\label{eq:surfaceIBPidentity} 
\end{align}
Above, we used the divergence theorem $\int_\Gamma \sgrad \cdot \mathbf J = -\int_\Gamma H\mathbf J\cdot \nu$ on closed surfaces \cite[Theorem 2.10]{Dziuk2013}.

The final equalities in the two sets of identities also hold when $a^+$ is replaced with $a^-$. 
\subsection{Useful inequalities}
We frequently use the interpolated trace inequality \cite[Theorem 1.5.1.10]{Grisvard}: given $u \in H^1(\Omega)$, the following holds for any $\epsilon > 0$:
\begin{equation}\label{eq:interpolatedTrace}
\int_{\Gamma}|u|^2 \leq \epsilon\int_\Omega |\grad u|^2 + \frac{C}{\epsilon}\int_\Omega |u|^2.
\end{equation}
We also use the standard Sobolev inequality: for $v \in H^1(\Gamma)$, 
\begin{equation}\label{eq:sobolevInequality}
C_I\norm{v}{L^p(\Gamma)} \leq \norm{v}{H^1(\Gamma)} \quad\text{where}\quad \begin{cases}
1 \leq p < \infty &: d\leq 2\\
p = \frac{2d}{d-2}&: d> 2.
\end{cases}
\end{equation}
\textbf{Interpolation inequalities}. Let us record some interpolation inequalities related to the quantities 
\begin{equation*}\label{normQ}
\norm{u}{Q(\Gamma)} := \max_{t \in [0,T]} \norm{u(t)}{L^2(\Gamma(t))} + \norm{\sgrad u}{L^2_{L^2(\Gamma)}}
\end{equation*}
and
\[\norm{u}{Q(\Omega)} := \max_{t \in [0,T]}\norm{u(t)}{L^2(\Omega(t))} + \norm{\grad u}{L^2_{L^2(\Omega)}}.\]
\begin{lem}\label{lem:gn}For $r_* > 2$ and $q_*$ defined by \[\frac{1}{q_*} = \frac{d-2}{r_*d} + \frac{r_*-2}{2r_*},\] we have
\begin{equation}\label{eq:gn}
\norm{u}{L^{r_*}_{L^{q_*}(\Gamma)}} \leq C_1(T)\norm{u}{Q(\Gamma)}
\end{equation}
where $C_1(T)=C_1\sqrt T$ if $T > 1$, otherwise $C_1$ is independent of $T$.
\end{lem}
\begin{proof}
The Gagliardo--Nirenberg inequality (eg. \cite[\S 1.2]{Brouttelande2003}) states
\[\norm{u}{L^{q_*}(\Gamma)} \leq C\norm{u}{H^1(\Gamma)}^\theta \norm{u}{L^2(\Gamma)}^{1-\theta}\]
where ${1}\slash {q_*} = {\theta(d-2)}\slash {2d} + {(1-\theta)}\slash {2}$ for $\theta \in (0,1)$.
This implies that
\[\norm{u}{L^{r_*}_{L^{q_*}(\Gamma)}} = \left(\int_0^T \norm{u}{L^{q_*}(\Gamma(t))}^{r_*}\right)^{\frac{1}{r_*}} \leq C\left(\int_0^T \norm{u}{H^1(\Gamma(t))}^{r_*\theta}\right)^{\frac{1}{r_*}}\max_t \norm{u(t)}{L^2(\Gamma(t))}^{(1-\theta)}.\]
For $r_* > 2$, choosing $\theta = {2}\slash {r_*}$, this reads
\[\norm{u}{L^{r_*}_{L^{q_*}(\Gamma)}} \leq C\norm{u}{L^2_{H^1(\Gamma)}}^{\frac{2}{r_*}}\max_t \norm{u(t)}{L^2(\Gamma(t))}^{1 - \frac{2}{r_*}}.\]
An application of Young's inequality with exponents $({r_*}\slash{2}, ({r_*}\slash {2})')$ implies
\[\norm{u}{L^{r_*}_{L^{q_*}(\Gamma)}} \leq \frac{2C}{r_*}\norm{u}{L^2_{H^1(\Gamma)}} + C\left(1-\frac{2}{r_*}\right)\max \norm{u(t)}{L^2(\Gamma(t))},\]
and we conclude by using
\begin{align*}
\norm{u}{L^2_{H^1(\Gamma)}} &\leq \sqrt{T}\max_t\norm{u(t)}{L^2(\Gamma(t))}+ \norm{\sgrad u}{L^2_{L^2(\Gamma)}}.
\end{align*}
\end{proof}
The result of the next lemma is similar to the inequality in Lemma \ref{lem:gn} but it relates the left hand side to a norm on the bulk domain. For more details see \cite[(A.1)]{Nittka} and references therein.
\begin{lem}\label{lem:interpolatedSobolev}For $r_* \in [2,\infty]$ and $q_* \in [2, 2d\slash (d-1)]$ satisfying
\[\frac{1}{r_*} + \frac{d}{2q_*} = \frac{d+1}{4},\] we have
\[\norm{u}{L^{r_*}_{L^{q_*}(\Gamma)}} \leq \sqrt{C_I}\norm{u}{Q(\Omega)}.\]
\end{lem}

\subsection{Aubin-—Lions compactness} 
Suppose $\Omega(t)$ is a moving domain with boundary $\Gamma(t)$ endowed with the velocity fields and properties and regularity stated in \S \ref{sec:weakFormulation}. We use the following compactness result which follows from the standard Aubin-—Lions compactness theory on Bochner spaces (see \cite[Theorem II.5.16]{Boyer}) and the fact that the standard Bochner spaces are isomorphic to our time-evolving Bochner spaces with an equivalence of norms (see \cite[\S 4 and \S 5]{AESApplications}).
\begin{lem}\label{Lem:AubinLions}
The following holds:
\begin{enumerate}
\item The space $H^1_{H^1(\Gamma)^*}\cap L^2_{H^1(\Gamma)}$ is compactly embedded in $L^2_{L^2 (\Gamma)}.$
\item The space $H^1_{L^2(\Gamma)}\cap L^2_{H^1(\Gamma)}$ is compactly embedded in $L^2_{H^{1\slash 2}(\Gamma)}.$
\item The space $H^1_{H^1(\Omega)^*}\cap L^2_{H^1(\Omega)}$ is compactly embedded in $L^2_{L^2(\Omega)}$.
\end{enumerate}
\end{lem}

\begin{proof}
Let us begin by addressing the first case. Let $u_n$ be a bounded sequence in $H^1_{H^{1}(\Gamma)^*}\cap L^2_{H^1(\Gamma)} =: W_\Gamma$; we need to show that there is a subsequence that converges in $L^2_{L^2(\Gamma)}$. 

For readability, define $W_0 := H^1(0,T;H^1(\Gamma_0)^*) \cap L^2(0,T;H^1(\Gamma_0))$. The mapping $\phi_{-t} w = w \circ \Phi_t$ defined in section \ref{sec:weakFormulation} is such that 
\begin{equation}\label{eq:appendixIFF}
w \in W_\Gamma \quad\text{if and only if} \quad \phi_{-(\cdot)}w(\cdot) \in W_0
\end{equation}
 and the equivalence of norms
\begin{equation}\label{eq:appendixEON}
C_1\norm{\phi_{-(\cdot)}w(\cdot)}{W_0} \leq \norm{w}{W_\Gamma} \leq C_2\norm{\phi_{-(\cdot)}w(\cdot)}{W_0}\quad \forall w \in W_\Gamma
\end{equation}
holds. These two properties are indeed true due to the smoothness assumed on the velocity fields in the definition of $\Phi_t$ which then implies a certain smoothness on the Jacobian term $J_t$ that arises when transforming integrals on $\Gamma(t)$ onto the initial surface $\Gamma_0$ via the map $\Phi_t$ and/or its inverse (refer to section \ref{sec:strongSolutionsForu} where we defined these objects); thus it is not difficult to see that $w \in L^2_{H^1(\Gamma)}$ if and only if $\phi_{-(\cdot)}w(\cdot) \in L^2(0,T;H^1(\Gamma_0))$. The proof is then in two steps to deal with the time and material derivatives: one first proves that $\hat w$ belongs to $W_0$ if and only if $J_{(\cdot)}\hat w$ belongs to $W_0$, then one proves $\hat w \in W_0$ has a weak time derivative if and only if its pushforward $\phi_{(\cdot)}\hat w(\cdot)$ (here $\phi_t := (\phi_{-t})^{-1}$) has a weak material derivative which can be shown by some manipulations involving the formula defining the time and material derivatives (see the proof of Theorem 2.33 in \cite{AESAbstract} for the details of this in an abstract setting). The equivalence of norms \eqref{eq:appendixEON} is again a consequence of the smoothness of $J_{(\cdot)}$. Full details can be found in the aforementioned citation and also in \cite[\S 4.1]{AESApplications} for this particular case.


Due to \eqref{eq:appendixEON}, the sequence $\hat u_n := \phi_{-(\cdot)}u_n$ is bounded in $W_0$, and thanks to the standard Aubin--Lions compactness result \cite[Theorem II.5.16]{Boyer} on Bochner spaces, this gives rise to the existence of a subsequence $\hat u_{n_j}$ such that
\[\hat u_{n_j} \to \hat\eta \quad \text{in $L^2(0,T;L^2(\Gamma_0))$}\] for some $\hat\eta \in L^2(0,T;L^2(\Gamma_0))$. Because of our smoothness assumptions on $\Phi_t$, the map $\phi_{t}\colon L^2(\Gamma_0) \to L^2(\Gamma(t))$ is bounded \emph{uniformly} in $t$ (this is easy to see by transforming integrals using the Jacobian term $J_{(\cdot)}$, see also \cite[Lemma 3.2]{Vierling}) and this along with the measurability of $t \mapsto \norm{\phi_t \hat w_0}{L^2(\Gamma(t))}$ (for fixed $\hat w_0 \in L^2(\Gamma_0)$) ensures that $\phi_{(\cdot)}$ also carries $L^2(0,T;L^2(\Gamma_0))$ into $L^2_{L^2(\Gamma)}$ with
\begin{equation}\label{eq:appendixEONnonDiff}
C_1\norm{\hat w}{L^2(0,T;L^2(\Gamma_0))} \leq \norm{\phi_{(\cdot)}\hat w(\cdot)}{L^2_{L^2(\Gamma)}} \leq C_2\norm{\hat w}{L^2(0,T;L^2(\Gamma_0))}\quad \forall \hat w \in L^2(0,T;L^2(\Gamma_0))
\end{equation}
(see also the paragraph after Definition 3.3 in \cite{AESApplications}).  Using this inequality with $\hat w=\hat u_{n_j}-\hat\eta$ and recalling that $\phi_{(\cdot)}$ is linear, we see that the pushforward $\phi_{(\cdot)}\hat u_{n_j}(\cdot)=\phi_{(\cdot)}\phi_{-(\cdot)}u_{n_j}(\cdot)=u_{n_j}$ of this subsequence converges to $\phi_{(\cdot)}\hat\eta(\cdot)$ in $L^2_{L^2(\Gamma)}$, which proves the result.

This structure of the above proof is unchanged for the remaining two cases; we just need to verify \eqref{eq:appendixIFF}, \eqref{eq:appendixEON} and \eqref{eq:appendixEONnonDiff} with the right modifications for the particular case in consideration. The properties \eqref{eq:appendixIFF} and \eqref{eq:appendixEON} for the second case in the lemma can be justified with the simple technical adjustments needed for the better regularity present in the time/material derivative. The corresponding inequality \eqref{eq:appendixEONnonDiff} here is
\[C_1\norm{\hat w}{L^2(0,T;H^{1\slash 2}(\Gamma_0))} \leq \norm{\phi_{(\cdot)}\hat w(\cdot)}{L^2_{H^{1\slash 2}(\Gamma)}} \leq C_2\norm{\hat w}{L^2(0,T;H^{1\slash 2}(\Gamma_0))}\quad\forall \hat w \in L^2(0,T;H^{1\slash 2}(\Gamma_0))
\]
and the justification of this and related technical matters were studied in detail in \S 5.4.1 of \cite{AESApplications}. The third case is also very similar to the proof presented above; the only difference is that the spaces are on the domain $\Omega(t)$ instead of the boundary and indeed this has been dealt with in \cite[\S 4.2]{AESApplications}.
\end{proof}

\section{Non-dimensionalisation}\label{sec:nondim}
In order to non-dimensionalise the system \eqref{eq:bioModel}, we use the new variables
\[\bar x = x\slash L
\qquad \bar t = t\slash S
\qquad \bar u = u\slash U
\qquad \bar w = w\slash W
\qquad \bar z = z\slash Z
\qquad \bar{\mathbf V}_\Omega = S \mathbf{V_\Omega}\slash L
\qquad \bar{\mathbf V}_\Gamma = S\mathbf {V_\Gamma} \slash L\]
under the notation $\bar f(\bar t, \bar x)=f(t,x)\slash F$. Here $L$ is a length scale, $S$ is a time scale and  $U, W, Z$ are typical concentration values for $u,w,z$ respectively. Observe by the chain rule that
\[u_t = \frac US\bar u'
\qquad \grad u = \frac UL\grad \bar u
\qquad \grad \cdot \mathbf{V}_\Omega = \frac 1S \grad \cdot \bar{\mathbf{V}_\Omega}
\qquad (\mathbf{V}_\Gamma - \mathbf{V}_\Omega)\cdot \nu = \frac LS (\bar{\mathbf{V}_\Gamma}-\bar{\mathbf{V}_\Omega})\cdot \nu =: \frac LS\bar j
\]
and hence $\dot u = \frac US \md \bar u.$ The $u$ equation then reads
\[\frac US \md \bar u + \frac US \bar u \grad \cdot \bar{\mathbf{V}_\Omega} - \frac{D_LU}{L^2}\lap \bar u = 0,\]
which we can multiply through by $\frac SU$ to obtain
\[\md \bar u + \bar u \grad \cdot \bar{\mathbf{V}_\Omega} - \delta_\Omega \lap \bar u = 0\]
where we have set $\delta_\Omega := {SD_L}\slash {L^2}$. The boundary condition becomes
\[\frac{D_L U}{L}\grad \bar u \cdot \nu + \frac{UL}{S}\bar j\bar u  = Zk_{off}\bar z - UWk_{on}\bar u \bar w\]
which we can write as 
(multiplying by ${S}\slash {LU}$)
\[\delta_\Omega\grad \bar u \cdot \nu + \bar u \bar j = \frac{SZ}{LU}k_{off}\bar z - \frac{SW}{L}k_{on}\bar u \bar w.\]
In a similar fashion, we derive the equations for $w$ and $z$: 
\begin{align*}
\md \bar w + \bar w\sgrad \cdot \bar{\mathbf{V}_\Gamma} - \frac{SD_\Gamma}{L^2}\slap \bar w &= \frac{SZ}{W}k_{off}\bar z - USk_{on}\bar u \bar w\\
\md \bar z + \bar z\sgrad \cdot \bar{\mathbf{V}_\Gamma} - \frac{D_{\Gamma'} S}{L^2}\slap \bar z &= -\left(Sk_{off}\bar z - \frac{SUW}{Z}k_{on}\bar u \bar w\right)
\end{align*}
Defining
\[
\delta_\Gamma := \frac{D_\Gamma S}{L^2}
\qquad \delta_{\Gamma'} := \frac{SD_{\Gamma'}}{L^2}
\qquad \gamma := \frac{LU}{W}
\qquad d_{k} := \frac{L}{WSk_{on}}
\qquad \gamma' := \frac {LU}{Z}
\qquad d_{k'} := \frac{UL}{SZk_{off}},\] so that
\[\frac{\gamma}{d_k} = USk_{on}\qquad \frac{\gamma'}{d_k} = \frac{WUS}{Z}k_{on}
\qquad \frac{\gamma}{d_{k'}} = \frac{ZSk_{off}}{W}\qquad\text{and}\qquad \frac{\gamma'}{d_{k'}} = Sk_{off},\]
we can write
\begin{align*}
\delta_\Omega\grad \bar u \cdot \nu + \bar u \bar j &= \frac{1}{d_{k'}}\bar z - \frac{1}{d_k}\bar u \bar w\\
    \md \bar w + \bar w\sgrad \cdot \bar{\mathbf{V}_\Gamma} - d_{\Gamma}\slap \bar w &= \gamma\left(\frac{1}{d_{k'}}\bar z - \frac{1}{d_k}\bar u \bar w\right)\\
    \md \bar z + \bar z\sgrad \cdot \bar{\mathbf{V}_\Gamma} - d_{\Gamma'}\slap \bar z &= -\gamma'\left(\frac{1}{d_{k'}}\bar z - \frac{1}{d_k}\bar u \bar w\right).
\end{align*}
Now set $w = {\bar w}\slash {\gamma}$ and $z={\bar z}\slash {\gamma'}$ and the above then becomes
\begin{align*}
\delta_\Omega\grad \bar u \cdot \nu + \bar u \bar j &= \frac{\gamma'}{d_{k'}} z - \frac{\gamma}{d_k}\bar u w\\
    \md  w +  w\sgrad \cdot \bar{\mathbf{V}_\Gamma} - d_{\Gamma}\slap w &= \frac{\gamma'}{d_{k'}} z - \frac{\gamma}{d_k}\bar u w\\
    \md z +  z\sgrad \cdot \bar{\mathbf{V}_\Gamma} - d_{\Gamma'}\slap z &= -\left(\frac{\gamma'}{d_{k'}} z - \frac{\gamma}{d_k}\bar u w\right).
\end{align*}
Finally, set 
\[\frac{1}{\delta_k} := \frac{\gamma}{d_k}\qquad \frac{1}{\delta_{k'}} := \frac{\gamma'}{d_{k'}}\] and relabel all variables (and write $j:=-\bar j$) to obtain (recalling the equation for $\bar u$)
\begin{align*}
\md u +  u \grad \cdot {\mathbf{V}_\Omega} - \delta_\Omega \lap  u &= 0\\
\delta_\Omega\grad  u \cdot \nu - u j &= \frac{1}{\delta_{k'}} z - \frac{1}{\delta_k} u w\\
    \md  w +  w\sgrad \cdot {\mathbf{V}_\Gamma} - d_{\Gamma}\slap w &= \frac{1}{\delta_{k'}} z - \frac{1}{\delta_k} u w\\
    \md z +  z\sgrad \cdot {\mathbf{V}_\Gamma} - d_{\Gamma'}\slap z &= \frac{1}{\delta_k} u w -\frac{1}{\delta_{k'}} z .
\end{align*}
 This is exactly the model \eqref{eq:modelALE} with the parametrisation velocity $\mathbf V_p$ chosen to be the corresponding material velocities. 
\bibliographystyle{abbrv}
\bibliography{CBBSBib}

\def\cprime{$'$}
\begin{thebibliography}{10}

\bibitem{AEStefan}
A.~Alphonse and C.~M. Elliott.
\newblock A {S}tefan problem on an evolving surface.
\newblock {\em Philos. Trans. A}, 373(2050):20140279, 16, 2015.

\bibitem{AESAbstract}
A.~Alphonse, C.~M. Elliott, and B.~Stinner.
\newblock An abstract framework for parabolic {PDE}s on evolving spaces.
\newblock {\em Port. Math.}, 72(1):1--46, 2015.

\bibitem{AESApplications}
A.~Alphonse, C.~M. Elliott, and B.~Stinner.
\newblock On some linear parabolic {PDE}s on moving hypersurfaces.
\newblock {\em Interfaces Free Bound.}, 17(2):157--187, 2015.

\bibitem{Arendt}
W.~Arendt, C.~J.~K. Batty, M.~Hieber, and F.~Neubrander.
\newblock {\em Vector-valued {L}aplace transforms and {C}auchy problems},
  volume~96 of {\em Monographs in Mathematics}.
\newblock Birkh\"auser/Springer Basel AG, Basel, second edition, 2011.

\bibitem{aronson1963}
D.~G. Aronson.
\newblock {On the Green's function for second order parabolic differential
  equations with discontinuous coefficients}.
\newblock {\em Bull. Amer. Math. Soc.}, 69(6):841--847, 11 1963.

\bibitem{Aronson1967}
D.~G. Aronson and J.~Serrin.
\newblock Local behavior of solutions of quasilinear parabolic equations.
\newblock {\em Archive for Rational Mechanics and Analysis}, 25(2):81--122,
  1967.

\bibitem{BaoFellnerLatos}
T.~Q. {Bao}, K.~{Fellner}, and E.~{Latos}.
\newblock {Well-posedness and exponential equilibration of a volume-surface
  reaction-diffusion system with nonlinear boundary coupling}.
\newblock {\em ArXiv e-prints}, Apr. 2014.

\bibitem{Bothe2012}
D.~Bothe, M.~Pierre, and G.~Rolland.
\newblock {Cross-Diffusion Limit for a Reaction-Diffusion System with Fast
  Reversible Reaction}.
\newblock {\em Communications in Partial Differential Equations},
  37(11):1940--1966, 2012.

\bibitem{Boyer}
F.~Boyer and P.~Fabrie.
\newblock {\em Mathematical Tools for the Study of the Incompressible
  Navier-Stokes Equations and Related Models}.
\newblock Applied Mathematical Sciences. Springer, 2012.

\bibitem{Brouttelande2003}
C.~Brouttelande.
\newblock {The Best-Constant Problem for a Family of Gagliardo--Nirenberg
  Inequalities on a Compact Riemannian Manifold}.
\newblock {\em Proceedings of the Edinburgh Mathematical Society},
  46(01):117--146, 2003.

\bibitem{CFG}
P.~Cermelli, E.~Fried, and M.~E. Gurtin.
\newblock Transport relations for surface integrals arising in the formulation
  of balance laws for evolving fluid interfaces.
\newblock {\em J. Fluid Mech.}, 544:339--351, 2005.

\bibitem{Cs}
I.~Csisz\'{a}r.
\newblock {Eine informationstheoretische Ungleichung und ihre Anwendung auf den
  Beweis der Ergodizitat von Markoffschen Ketten}.
\newblock {\em Magyar. Tud. Akad. Mat. Kutato Int. Kozl.}, 8:85--108, 1963.

\bibitem{DeGiorgiEnnio}
E.~De~Giorgi.
\newblock Sulla differenziabilit\`a e l'analiticit\`a delle estremali degli
  integrali multipli regolari.
\newblock {\em Mem. Accad. Sci. Torino. Cl. Sci. Fis. Mat. Nat. (3)}, 3:25--43,
  1957.

\bibitem{Delfour}
M.~C. Delfour and J.-P. Zol{\'e}sio.
\newblock {\em Shapes and geometries}, volume~22 of {\em Advances in Design and
  Control}.
\newblock Society for Industrial and Applied Mathematics (SIAM), Philadelphia,
  PA, second edition, 2011.
\newblock Metrics, analysis, differential calculus, and optimization.

\bibitem{Demengel}
F.~Demengel, G.~Demengel, and R.~Ern{\'e}.
\newblock {\em Functional Spaces for the Theory of Elliptic Partial
  Differential Equations}.
\newblock Universitext. Springer London, 2012.

\bibitem{Pruss}
R.~Denk, M.~Hieber, and J.~Pr{\"u}ss.
\newblock Optimal {$L^p$}-{$L^q$}-estimates for parabolic boundary value
  problems with inhomogeneous data.
\newblock {\em Math. Z.}, 257(1):193--224, 2007.

\bibitem{DesFellner}
L.~Desvillettes and K.~Fellner.
\newblock Exponential decay toward equilibrium via entropy methods for
  reaction-diffusion equations.
\newblock {\em J. Math. Anal. Appl.}, 319(1):157--176, 2006.

\bibitem{Desvillettes2014}
L.~Desvillettes and K.~Fellner.
\newblock Exponential convergence to equilibrium for nonlinear
  reaction-diffusion systems arising in reversible chemistry.
\newblock In C.~P{\"o}tzsche, C.~Heuberger, B.~Kaltenbacher, and F.~Rendl,
  editors, {\em System Modeling and Optimization: 26th IFIP TC 7 Conference,
  CSMO 2013, Klagenfurt, Austria, September 9-13, 2013, Revised Selected
  Papers}, pages 96--104. Springer Berlin Heidelberg, Berlin, Heidelberg, 2014.

\bibitem{DiBenedetto}
E.~DiBenedetto.
\newblock {\em Degenerate Parabolic Equations}.
\newblock Universitext (Berlin. Print). Springer New York, 1993.

\bibitem{DGV}
E.~DiBenedetto, U.~Gianazza, and V.~Vespri.
\newblock {\em Harnack's Inequality for Degenerate and Singular Parabolic
  Equations}.
\newblock Springer Monographs in Mathematics. Springer New York, 2011.

\bibitem{Dziuk2006}
G.~Dziuk and C.~M. Elliott.
\newblock {Finite elements on evolving surfaces}.
\newblock {\em IMA Journal of Numerical Analysis}, 27(2):262--292, 2006.

\bibitem{Dziuk2013}
G.~Dziuk and C.~M. Elliott.
\newblock {Finite element methods for surface PDEs}.
\newblock {\em Acta Numerica}, 22:289--396, 2013.

\bibitem{ERV}
C.~M. {Elliott}, T.~{Ranner}, and C.~{Venkataraman}.
\newblock {Coupled bulk-surface free boundary problems arising from a
  mathematical model of receptor-ligand dynamics}.
\newblock {\em SIAM J. Math. Anal. (to appear)}, 2017.

\bibitem{Elliott2012}
C.~M. Elliott and V.~Styles.
\newblock An {ALE} {ESFEM} for solving {PDE}s on evolving surfaces.
\newblock {\em Milan J. Math.}, 80(2):469--501, 2012.

\bibitem{FellnerLaamri}
K.~Fellner and E.-H. Laamri.
\newblock Exponential decay towards equilibrium and global classical solutions
  for nonlinear reaction-diffusion systems.
\newblock {\em J. Evol. Equ.}, 16(3):681--704, 2016.

\bibitem{Fellner2015}
K.~Fellner and B.~A. O.~Q. Tang.
\newblock {Explicit exponential convergence to equilibrium for mass action
  reaction-diffusion systems}.
\newblock pages 1--29, 2015.

\bibitem{Feng}
W.~Feng.
\newblock Coupled system of reaction-diffusion equations and applications in
  carrier facilitated diffusion.
\newblock {\em Nonlinear Anal.}, 17(3):285--311, 1991.

\bibitem{BioPaper}
P.~Garc{\'{\i}}a-Pe{\~n}arrubia, J.~J. G{\'a}lvez, and J.~G{\'a}lvez.
\newblock Mathematical modelling and computational study of two-dimensional and
  three-dimensional dynamics of receptor-ligand interactions in signalling
  response mechanisms.
\newblock {\em J. Math. Biol.}, 69(3):553--582, 2014.

\bibitem{Grisvard}
P.~Grisvard.
\newblock {\em Elliptic Problems in Nonsmooth Domains}.
\newblock Classics in Applied Mathematics. Society for Industrial and Applied
  Mathematics, 2011.

\bibitem{MR3423226}
S.~Gross, M.~A. Olshanskii, and A.~Reusken.
\newblock A trace finite element method for a class of coupled bulk-interface
  transport problems.
\newblock {\em ESAIM Math. Model. Numer. Anal.}, 49(5):1303--1330, 2015.

\bibitem{jungel}
A.~J{\"u}ngel.
\newblock {\em Entropy Methods for Diffusive Partial Differential Equations}.
\newblock SpringerBriefs in Mathematics. Springer International Publishing,
  2016.

\bibitem{Lady}
O.~A. Lady{\v{z}}enskaja, V.~A. Solonnikov, and N.~N. Ural{\cprime}ceva.
\newblock {\em Linear and quasilinear equations of parabolic type}.
\newblock Translated from the Russian by S. Smith. Translations of Mathematical
  Monographs, Vol. 23. American Mathematical Society, Providence, R.I., 1968.

\bibitem{LU}
O.~A. Lady{\v{z}}enskaja and N.~N. Ural{\cprime}ceva.
\newblock A boundary-value problem for linear and quasi-linear parabolic
  equations. {I}, {II}, {III}.
\newblock {\em Iaz. Akad. Nauk SSSR Ser. Mat. 26 (1962), 5-52; ibid. 26 (1962),
  753- 780; ibid.}, 27:161--240, 1962.

\bibitem{MacKenzie}
G.~MacDonald, J.~Mackenzie, M.~Nolan, and R.~Insall.
\newblock A computational method for the coupled solution of reaction-diffusion
  equations on evolving domains and manifolds: application to a model of cell
  migration and chemotaxis.
\newblock {\em Journal of Computational Physics}, 309:207--226, March 2016.

\bibitem{Madzvamuse2014}
A.~Madzvamuse, H.~S. Ndakwo, and R.~Barreira.
\newblock {Cross-diffusion-driven instability for reaction-diffusion systems:
  analysis and simulations}.
\newblock {\em Journal of Mathematical Biology}, 70(4):709--743, 2014.

\bibitem{Madzvamuse2016}
A.~Madzvamuse, H.~S. Ndakwo, and R.~Barreira.
\newblock {Stability analysis of reaction-diffusion models on evolving domains
  the effects of cross-diffusion}.
\newblock {\em Discrete and Continuous Dynamical Systems- Series A},
  36(4):2133--2170, 2016.

\bibitem{Marth2014}
W.~Marth and A.~Voigt.
\newblock {Signaling networks and cell motility: A computational approach using
  a phase field description}.
\newblock {\em Journal of Mathematical Biology}, 69(1):91--112, 2014.

\bibitem{Morgan}
J.~Morgan.
\newblock Global existence for semilinear parabolic systems.
\newblock {\em SIAM Journal on Mathematical Analysis}, 20(5):1128--1144, 1989.

\bibitem{Ni}
W.-M. Ni.
\newblock Diffusion, cross-diffusion, and their spike-layer steady states.
\newblock {\em Notices Amer. Math. Soc.}, 45(1):9--18, 1998.

\bibitem{Nittka}
R.~Nittka.
\newblock Inhomogeneous parabolic {N}eumann problems.
\newblock {\em Czechoslovak Math. J.}, 64(139)(3):703--742, 2014.

\bibitem{Ratz2012}
A.~R{\"{a}}tz and M.~R{\"{o}}ger.
\newblock {Turing instabilities in a mathematical model for signaling
  networks}.
\newblock {\em Journal of Mathematical Biology}, 65(6-7):1215--1244, 2012.

\bibitem{RatzRogerSymmetry}
A.~R\"atz and M.~R\"oger.
\newblock {Symmetry breaking in a bulk-surface reaction-diffusion model for
  signalling networks}.
\newblock {\em Nonlinearity}, 27(8):1805, 2014.

\bibitem{Rothe}
F.~Rothe.
\newblock {\em Global solutions of reaction-diffusion systems}, volume 1072 of
  {\em Lecture Notes in Mathematics}.
\newblock Springer-Verlag, Berlin, 1984.

\bibitem{SHIGESADA197983}
N.~Shigesada, K.~Kawasaki, and E.~Teramoto.
\newblock Spatial segregation of interacting species.
\newblock {\em Journal of Theoretical Biology}, 79(1):83 -- 99, 1979.

\bibitem{Stampacchia1965}
G.~Stampacchia.
\newblock {Le probl\`{e}me de Dirichlet pour les \'{e}quations elliptiques du
  second ordre \`{a} coefficients discontinus}.
\newblock {\em Annales de l'institut Fourier}, 15(1):189--257, 1965.

\bibitem{Stroock1993}
D.~W. Stroock.
\newblock {Logarithmic Sobolev inequalities for Gibbs states}.
\newblock In G.~Dell'Antonio and U.~Mosco, editors, {\em Dirichlet Forms:
  Lectures given at the 1st Session of the Centro Internazionale Matematico
  Estivo (C.I.M.E.) held in Varenna, Italy, June 8--19, 1992}, pages 194--228.
  Springer Berlin Heidelberg, Berlin, Heidelberg, 1993.

\bibitem{B813825G}
V.~K. Vanag and I.~R. Epstein.
\newblock Cross-diffusion and pattern formation in reaction-diffusion systems.
\newblock {\em Phys. Chem. Chem. Phys.}, 11:897--912, 2009.

\bibitem{Vierling}
M.~Vierling.
\newblock Parabolic optimal control problems on evolving surfaces subject to
  point-wise box constraints on the control---theory and numerical realization.
\newblock {\em Interfaces Free Bound.}, 16(2):137--173, 2014.

\bibitem{Weidemaier2002}
P.~Weidemaier.
\newblock {Maximal regularity for parabolic equations with inhomogeneous
  boundary conditions in Sobolev spaces with mixed $L_p$-norm}.
\newblock {\em Electronic Research Announcements of the American Mathematical
  Society}, 8(02):47----51, 2002.

\end{thebibliography}
\end{document}